\pdfoutput=1
\documentclass[11pt, a4paper, twoside,leqno]{amsart}
\usepackage[centering, totalwidth = 380pt, totalheight = 600pt]{geometry}
\usepackage{amssymb, amsmath, amsthm}
\usepackage{microtype, stmaryrd, url, lmodern, eucal,extdash}
\usepackage[shortlabels]{enumitem}
\usepackage[latin1]{inputenc}
\usepackage{color}
\definecolor{darkgreen}{rgb}{0,0.45,0}
\usepackage[colorlinks,citecolor=darkgreen,linkcolor=darkgreen]{hyperref}
\usepackage[british]{babel}

\makeatletter
\def\@cite#1#2{[{#1\if@tempswa ,~#2\fi}]}
\makeatother

\DeclareMathAlphabet{\mathbf}{OT1}{cmr}{b}{n}

\usepackage[arrow, matrix, tips, curve, graph, rotate]{xy}
\SelectTips{cm}{10}

\makeatletter
\def\matrixobject@{%
  \edef \next@{={\DirectionfromtheDirection@ }}%
  \expandafter \toks@ \next@ \plainxy@
  \let\xy@@ix@=\xyq@@toksix@
  \xyFN@ \OBJECT@}
\let\xy@entry@@norm=\entry@@norm
\def\entry@@norm@patched{%
  \let\object@=\matrixobject@
  \xy@entry@@norm }
\AtBeginDocument{\let\entry@@norm\entry@@norm@patched}
\makeatother

\newcommand{\twosymb}[3][0.5]{\ar@{}[#2] \save ?(#1)*{#3}\restore}
\newcommand{\twocong}[2][0.5]{\ar@{}[#2] \save ?(#1)*{\cong}\restore}
\newcommand{\twoeq}[2][0.5]{\ar@{}[#2] \save ?(#1)*{=}\restore}
\newcommand{\rtwocell}[3][0.5]{\ar@{}[#2] \ar@{=>}?(#1)+/l 0.2cm/;?(#1)+/r 0.2cm/^{#3}}
\newcommand{\ltwocell}[3][0.5]{\ar@{}[#2] \ar@{=>}?(#1)+/r 0.2cm/;?(#1)+/l 0.2cm/^{#3}}
\newcommand{\ltwocello}[3][0.5]{\ar@{}[#2] \ar@{=>}?(#1)+/r 0.2cm/;?(#1)+/l 0.2cm/_{#3}}
\newcommand{\dtwocell}[3][0.5]{\ar@{}[#2] \ar@{=>}?(#1)+/u  0.2cm/;?(#1)+/d 0.2cm/^{#3}}
\newcommand{\dltwocell}[3][0.5]{\ar@{}[#2] \ar@{=>}?(#1)+/ur  0.2cm/;?(#1)+/dl 0.2cm/^{#3}}
\newcommand{\drtwocell}[3][0.5]{\ar@{}[#2] \ar@{=>}?(#1)+/ul  0.2cm/;?(#1)+/dr 0.2cm/^{#3}}
\newcommand{\dthreecell}[3][0.5]{\ar@{}[#2] \ar@3{->}?(#1)+/u  0.2cm/;?(#1)+/d 0.2cm/^{#3}}
\newcommand{\utwocell}[3][0.5]{\ar@{}[#2] \ar@{=>}?(#1)+/d 0.2cm/;?(#1)+/u 0.2cm/_{#3}}
\newcommand{\dtwocelltarg}[3][0.5]{\ar@{}#2 \ar@{=>}?(#1)+/u  0.2cm/;?(#1)+/d 0.2cm/^{#3}}
\newcommand{\utwocelltarg}[3][0.5]{\ar@{}#2 \ar@{=>}?(#1)+/d  0.2cm/;?(#1)+/u 0.2cm/_{#3}}

\newcommand{\pullbackcorner}[1][dr]{\save*!/#1-1.2pc/#1:(-1,1)@^{|-}\restore}
\newcommand{\pullbackcornerl}[1][dl]{\save*!/#1-1.5pc/#1:(-1,1)@^{|-}\restore}

\newdir{(}{{}*!<0em,-.14em>-\cir<.14em>{l^r}}
\newdir{ (}{{}*!/-5pt/\dir{(}}
\newdir{ >}{{}*!/-5pt/\dir{>}}
\newcommand{\sh}[2]{**{!/#1 -#2/}}

\newcommand{\cat}[1]{\mathrm{\mathcal #1}}
\newcommand{\thg}{{\mathord{\text{--}}}}

\newcommand{\dbr}[1]{{\left\llbracket{#1}\right\rrbracket}}
\newcommand{\res}[2]{\left.{#1}\right|_{#2}}

\newcommand{\spn}[1]{{\langle{#1}\rangle}}

\newcommand{\defeq}{\mathrel{\mathop:}=}
\newcommand{\cd}[2][]{\vcenter{\hbox{\xymatrix#1{#2}}}}

\renewcommand{\phi}{\varphi}
\newcommand{\A}{{\mathcal A}}
\newcommand{\B}{{\mathcal B}}
\newcommand{\C}{{\mathcal C}}
\newcommand{\D}{{\mathcal D}}
\newcommand{\E}{{\mathcal E}}

\newcommand{\G}{{\mathcal G}}
\renewcommand{\H}{{\mathcal H}}
\newcommand{\I}{{\mathcal I}}

\renewcommand{\L}{{\mathcal L}}
\newcommand{\M}{{\mathcal M}}

\renewcommand{\O}{{\mathcal O}}
\renewcommand{\P}{{\mathcal P}}

\newcommand{\R}{{\mathcal R}}
\let\sec=\S
\renewcommand{\S}{{\mathcal S}}
\newcommand{\T}{{\mathcal T}}

\newcommand{\xtor}[1]{\cdl[@1]{{} \ar[r]|-{\object@{|}}^{#1} & {}}}

\makeatletter

\def\hookleftarrowfill@{\arrowfill@\leftarrow\relbar{\relbar\joinrel\rhook}}
\def\twoheadleftarrowfill@{\arrowfill@\twoheadleftarrow\relbar\relbar}
\def\leftbararrowfill@{\arrowdoublefill@{\leftarrow\mkern-5mu}\relbar\mapstochar\relbar\relbar}
\def\Leftbararrowfill@{\arrowdoublefill@{\Leftarrow\mkern-2mu}\Relbar\Mapstochar\Relbar\Relbar}
\def\leftringarrowfill@{\arrowdoublefill@{\leftarrow\mkern-3mu}\relbar{\mkern-3mu\circ\mkern-2mu}\relbar\relbar}
\def\lefttriarrowfill@{\arrowfill@{\mathrel\triangleleft\mkern0.5mu\joinrel\relbar}\relbar\relbar}
\def\Lefttriarrowfill@{\arrowfill@{\mathrel\triangleleft\mkern1mu\joinrel\Relbar}\Relbar\Relbar}

\def\hookrightarrowfill@{\arrowfill@{\lhook\joinrel\relbar}\relbar\rightarrow}
\def\twoheadrightarrowfill@{\arrowfill@\relbar\relbar\twoheadrightarrow}
\def\rightbararrowfill@{\arrowdoublefill@{\relbar\mkern-0.5mu}\relbar\mapstochar\relbar\rightarrow}
\def\Rightbararrowfill@{\arrowdoublefill@{\Relbar\mkern-2mu}\Relbar\Mapstochar\Relbar\Rightarrow}
\def\rightringarrowfill@{\arrowdoublefill@\relbar\relbar{\mkern-2mu\circ\mkern-3mu}\relbar{\mkern-3mu\rightarrow}}
\def\righttriarrowfill@{\arrowfill@\relbar\relbar{\relbar\joinrel\mkern0.5mu\mathrel\triangleright}}
\def\Righttriarrowfill@{\arrowfill@\Relbar\Relbar{\Relbar\joinrel\mkern1mu\mathrel\triangleright}}

\def\leftrightarrowfill@{\arrowfill@\leftarrow\relbar\rightarrow}
\def\mapstofill@{\arrowfill@{\mapstochar\relbar}\relbar\rightarrow}

\renewcommand*\xleftarrow[2][]{\ext@arrow 20{20}0\leftarrowfill@{#1}{#2}}
\providecommand*\xLeftarrow[2][]{\ext@arrow 60{22}0{\Leftarrowfill@}{#1}{#2}}
\providecommand*\xhookleftarrow[2][]{\ext@arrow 10{20}0\hookleftarrowfill@{#1}{#2}}
\providecommand*\xtwoheadleftarrow[2][]{\ext@arrow 60{20}0\twoheadleftarrowfill@{#1}{#2}}
\providecommand*\xleftbararrow[2][]{\ext@arrow 10{22}0\leftbararrowfill@{#1}{#2}}
\providecommand*\xLeftbararrow[2][]{\ext@arrow 50{24}0\Leftbararrowfill@{#1}{#2}}
\providecommand*\xleftringarrow[2][]{\ext@arrow 10{26}0\leftringarrowfill@{#1}{#2}}
\providecommand*\xlefttriarrow[2][]{\ext@arrow 80{24}0\lefttriarrowfill@{#1}{#2}}
\providecommand*\xLefttriarrow[2][]{\ext@arrow 80{24}0\Lefttriarrowfill@{#1}{#2}}

\renewcommand*\xrightarrow[2][]{\ext@arrow 01{20}0\rightarrowfill@{#1}{#2}}
\providecommand*\xRightarrow[2][]{\ext@arrow 04{22}0{\Rightarrowfill@}{#1}{#2}}
\providecommand*\xhookrightarrow[2][]{\ext@arrow 00{20}0\hookrightarrowfill@{#1}{#2}}
\providecommand*\xtwoheadrightarrow[2][]{\ext@arrow 03{20}0\twoheadrightarrowfill@{#1}{#2}}
\providecommand*\xrightbararrow[2][]{\ext@arrow 01{22}0\rightbararrowfill@{#1}{#2}}
\providecommand*\xRightbararrow[2][]{\ext@arrow 04{24}0\Rightbararrowfill@{#1}{#2}}
\providecommand*\xrightringarrow[2][]{\ext@arrow 01{26}0\rightringarrowfill@{#1}{#2}}
\providecommand*\xrighttriarrow[2][]{\ext@arrow 07{24}0\righttriarrowfill@{#1}{#2}}
\providecommand*\xRighttriarrow[2][]{\ext@arrow 07{24}0\Righttriarrowfill@{#1}{#2}}

\providecommand*\xmapsto[2][]{\ext@arrow 01{20}0\mapstofill@{#1}{#2}}
\providecommand*\xleftrightarrow[2][]{\ext@arrow 10{22}0\leftrightarrowfill@{#1}{#2}}
\providecommand*\xLeftrightarrow[2][]{\ext@arrow 10{27}0{\Leftrightarrowfill@}{#1}{#2}}

\makeatother

\numberwithin{equation}{section}

\theoremstyle{plain}
\newtheorem{Thm}{Theorem}[section]
\newtheorem*{Thm*}{Theorem}
\newtheorem*{Prop*}{Proposition}
\newtheorem{Prop}[Thm]{Proposition}
\newtheorem{Cor}[Thm]{Corollary}
\newtheorem{Lemma}[Thm]{Lemma}

\theoremstyle{definition}
\newtheorem{Defn}[Thm]{Definition}

\newtheorem{Ex}[Thm]{Example}
\newtheorem{Exs}[Thm]{Examples}
\newtheorem{Rk}[Thm]{Remark}

\newcommand{\sheq}[2]{{\dbr{{#1}\mathrel{\!\texttt{\upshape =}\!}{#2}}}}
\newcommand{\fneq}[3][P]{{\dbr{{#2}\mathbin{\texttt{\upshape =}_{#1}}{#3}}}}
\newcommand{\hcon}[1][jr]{\mathrm{#1}\cat{Cat}\mathord{\mkern1mu/\!/_{\!h}\mkern1mu} \C}
\newcommand{\hconn}[1][jr]{\mathrm{#1}\cat{Cat}\mathord{\mkern1mu/\!/_{\!h}\mkern1mu} (\C',\C)}
\newcommand{\rcat}[1][r]{\mathrm{#1}\cat{Cat}\mathord{\mkern1mu/\!/\mkern1mu}\C}
\newcommand{\rcatt}[1][r]{\mathrm{#1}\cat{Cat}\mathord{\mkern1mu/\!/\mkern1mu}(\C',\C)}
\newcommand{\parte}{\mathrm{pe}\cat{Cat}_c(\C)}
\newcommand{\partee}{\mathrm{pe}\cat{Cat}_c(\C',\C)}
\newcommand{\partc}{\mathrm{p}\cat{Cat}_c(\C)}
\newcommand{\partcc}{\mathrm{p}\cat{Cat}_c(\C',\C)}
\newcommand{\tot}{\C \mathord{\mkern1.5mu/_{\!t}\mkern1.5mu}X}
\newcommand{\lh}{\C \mathord{\mkern1.5mu/_{\!\ell h}\mkern1.5mu}X}
\newcommand{\shv}{\cat{Sh}(X)}
\newcommand{\pshv}{\cat{Psh}(X)}

\begin{document}
\leftmargini=2em
\title[Generalising the \'etale groupoid--pseudogroup correspondence]{Generalising the \'etale groupoid--\\complete pseudogroup correspondence}
\author{Robin Cockett}
\address{Department of Computer Science, University of Calgary, 2500 University Dr. NW, Calgary, Alberta, Canada T2N 1N4}
\email{robin@ucalgary.ca}
\author{Richard Garner} 
\address{Centre of Australian Category Theory, Department of
  Mathematics \& Statistics, Macquarie University, NSW 2109, Australia} 
\email{richard.garner@mq.edu.au}

\subjclass[2000]{Primary: }
\date{\today}

\thanks{The support of Australian Research Council grants DP160101519
  and FT160100393 is gratefully acknowledged.}

\begin{abstract}
  We prove a generalisation of the correspondence, due to Resende and
  Lawson--Lenz, between \'etale groupoids---which are topological groupoids
  whose source map is a local homeomorphisms---and complete
  pseudogroups---which are inverse monoids equipped with a particularly
  nice representation on a topological space.

  Our generalisation improves on the existing functorial
  correspondence in four ways. Firstly, we enlarge the classes of maps
  appearing to each side. Secondly, we generalise on one side from
  inverse monoids to inverse categories, and on the other side, from
  \'etale groupoids to what we call \emph{partite} \'etale groupoids.
  Thirdly, we generalise from \'etale groupoids to source-\'etale
  categories, and on the other side, from inverse monoids to
  restriction monoids. Fourthly, and most far-reachingly, we
  generalise from topological \'etale groupoids to \'etale groupoids
  internal to any \emph{join restriction category} $\C$ with local
  glueings; and on the other side, from complete pseudogroups to
  ``complete $\C$-pseudogroups'', i.e., inverse monoids with a nice
  representation on an object of $\C$. Taken together, our results
  yield an equivalence, for a join restriction category $\C$ with
  local glueings, between join restriction categories with a
  well-behaved functor to $\C$, and partite source-\'etale internal
  categories in $\C$. In fact, we obtain this by cutting
  down a larger adjunction between arbitrary restriction categories
  over $\C$, and partite internal categories in $\C$.

  Beyond proving this main result, numerous applications are given,
  which reconstruct and extend existing correspondences in the
  literature, and provide general formulations of completion processes.
\end{abstract}
\maketitle
\tableofcontents
\section{Introduction}
\label{sec:introduction}
It is a well-known and classical
fact~\cite[Theorem~II.1.2.1]{Godement1958Topologie} that sheaves on a
topological space $X$ can be presented in two ways. On the one hand,
they can be seen as functors $\O(X)^\mathrm{op} \rightarrow \cat{Set}$
defined on the poset of open subsets of $X$ satisfying a glueing
condition. On the other hand, they can be seen as local homeomorphisms
$Y \rightarrow X$ over $X$. The functor
$\O(X)^\mathrm{op} \rightarrow \cat{Set}$ associated to a local
homeomorphism $p \colon Y \rightarrow X$ has its value at
$U \in \O(X)$ given by the set of partial sections of $p$ defined on
the open set $U$; while the local homeomorphism $p \colon Y \rightarrow X$
associated to a functor
$F \colon \O(X)^\mathrm{op} \rightarrow \cat{Set}$ has $Y$ given by
glueing together copies of open sets in $X$, taking one copy of $U$
for each element of $FU$, and $p$ given by the induced map from this glueing
into $X$.


These different views on sheaves underlie a richer correspondence
between \emph{\'etale groupoids} and \emph{complete pseudogroups},
whose basic idea dates back to work of
Ehresmann~\cite{Ehresmann1954Structures} and
Haefliger~\cite{Haefliger1958Structures}. On the one hand, an \emph{\'etale groupoid}
is a topological groupoid whose source map (and hence also target map)
is a local homeomorphism. On the other hand, a \emph{complete
  pseudogroup} is a certain kind of \emph{inverse monoid}, i.e., a
monoid $S$ such that for every $s \in S$ there is a unique
$s^\ast \in S$ with $ss^\ast s = s$ and $s^\ast s s^\ast = s^\ast$.
The key example of a complete pseudogroup is $\I(X)$, the monoid of
partial self-homeomorphisms of a space $X$; a general complete
pseudogroup is an inverse monoid $S$ ``modelled on some
$\I(X)$'', meaning that it is equipped with a monoid homomorphism
$\theta \colon S \rightarrow \I(X)$ such that:
\begin{enumerate}[(i),itemsep=0.25\baselineskip]
\item $S$ has all joins of compatible
families of elements, in a sense to be made precise in
Section~\ref{sec:join-inverse-categ} below; we call such an $S$ a \emph{join inverse monoid}.
\item $\theta$ restricts to an isomorphism $E(S) \rightarrow E(\I(X))$
  on sets of idempotents; we call such a $\theta$
  \emph{hyperconnected}.
\end{enumerate}

The \'etale groupoid--complete
pseudogroup correspondence  builds on the sheaf correspondence
as follows. On the one hand, if $\mathbb{A}$ is an \'etale groupoid,
then the corresponding complete pseudogroup is of the form
$\theta \colon \Phi(\mathbb{A}) \rightarrow \I(\mathbb{A}_0)$, where
$\Phi(\mathbb{A})$ given by the set of \emph{partial bisections} of
$\mathbb{A}$, i.e., partial sections
$s \colon U \rightarrow \mathbb{A}_1$ of the source map
$\sigma \colon \mathbb{A}_1 \rightarrow \mathbb{A}_0$ whose composite
with the target map
$\tau \colon \mathbb{A}_1 \rightarrow \mathbb{A}_0$ is a partial
bijection. Now $\theta$ is the function $s \mapsto \tau \circ s$,
while the inverse monoid structure of $\Phi(\mathbb{A})$ is derived
from the groupoid structure on $\mathbb{A}$.

On the other hand, if $\theta \colon S \rightarrow \I(X)$ is a
complete pseudogroup, then the corresponding \'etale groupoid
$\Psi(S)$ is of the form
$\sigma \colon X \leftarrow Y \rightarrow X \colon \tau$, where $Y$ is
given by glueing together copies of open sets in $X$, taking one copy
of $U$ for each $s \in S$ with $\mathrm{dom}(\theta(s)) = U$; where
$\sigma$ is the induced map from this glueing into $X$; where $\tau$
acts as $\theta(s)$ on the patch corresponding to $s \in S$; and where
the groupoid structure comes from the inverse monoid
structure of $S$.

At this level of generality, the \'etale groupoid--pseudogroup
correspondence was first established in~\cite[Theorem
I.2.15]{Resende2006Lectures}. However, it was not
until~\cite{Lawson2013Pseudogroups} that the correspondence was made
\emph{functorial}, by equating not only the objects but also suitable
classes of morphisms to each side of the correspondence. The choices
of these morphism are non-obvious: between \'etale
groupoids,~\cite{Lawson2013Pseudogroups} takes them to be the \emph{covering
  functors} (also known as \emph{discrete opfibrations}), while
between complete pseudogroups, they are the rather delicate class of
\emph{callitic} morphisms. (Let us note also that the definition of
``complete pseudogroup'' in~\cite{Lawson2013Pseudogroups} is slightly
different to the above, involving a join inverse monoid $S$
\emph{without} an explicit representation in some $\I(X)$; we will
return to this point later.)



The objective of this article is to describe a wide-ranging
generalisation of the correspondences in~\cite{Resende2006Lectures,
  Lawson2013Pseudogroups}; we will describe natural extensions along
four distinct axes. These extensions allow for a range of practically
useful applications, and, we hope, shed light on what makes the
construction work.

\textbf{Richer functoriality}. Our first axis of generalisation
enlarges the classes of maps found in~\cite{Lawson2013Pseudogroups}.
The maps of complete pseudogroups we consider are much simpler: a map
from $\theta \colon S \rightarrow \I(X)$ to
$\chi \colon T \rightarrow \I(Y)$ is just a monoid homomorphism
$\alpha \colon S \rightarrow T$ and a continuous function
$\beta \colon Y \rightarrow X$, compatible with $\theta$ and $\chi$ in
the evident way. The corresponding maps of \'etale groupoids are
perhaps less natural: rather than functors, they are
\emph{cofunctors}. These are a class of maps between groupoids
introduced by Higgins and Mackenzie~\cite{Higgins1993Duality} as the
most general class with respect to which the passage from a Lie
groupoid to its associated Lie algebroid is functorial; since this
passage involves many of the same ideas as the passage from an \'etale
groupoid to a complete pseudogroup, it is perhaps unsurprising that
cofunctors arise here too.

\textbf{Many objects}. Our three other axes of generalisation enlarge
the classes of objects to each side of the correspondence. For the
first of these, we replace, on the one side, complete pseudogroups
with their many-object generalisations, ``complete pseudogroupoids''.
These are particular kinds of \emph{inverse categories}, i.e.,
categories $\C$ such that for each $s \in \C(x,y)$ there is a unique
$s^\ast \in \C(y,x)$ satisfying $ss^\ast s = s$ and
$s^\ast s s^\ast = s^\ast$. The key example of a complete
pseudogroupoid is the inverse category $\cat{Top}_{ph}$ of topological
spaces and partial homeomorphisms; the general example takes the form
$P \colon \A \rightarrow \cat{Top}_{ph}$ where $\A$ is an inverse
category $\A$ with joins, and $P$ is a \emph{hyperconnected} functor
(i.e., inducing isomorphisms on sets of idempotent arrows).

On the other side, we replace \'etale groupoids by what we call
\emph{partite \'etale groupoids}. These involve a set $I$ of
``components''; for each $i \in I$, a space of objects $A_i$; for
each $i,j \in I$, a space of morphisms $A_{ij}$; and
continuous~maps
\begin{equation*}
  A_i \xleftarrow{\sigma_{ij}} A_{ij} \xrightarrow{\tau_{ij}} A_j \quad
  \eta_i \colon A_i \rightarrow A_{ii} \quad \mu_{ijk} \colon A_{jk}
  \times_{A_j} A_{ij} \rightarrow A_{ik} \quad \sigma_{ij} \colon
  A_{ij} \rightarrow A_{ji}
\end{equation*}
with each $\sigma_{ij}$ a local homeomorphism, 
satisfying the obvious analogues of the groupoid axioms. 

\textbf{Drop invertibility}. For our next axis of generalisation, we
replace, to the one side, \'etale groupoids by \emph{source-\'etale
  categories}; these are topological categories whose source map is a
local homeomorphism. On the other side we replace complete
pseudogroups $\theta \colon S \rightarrow \I(X)$ by ``complete
pseudomonoids'' ${\theta \colon S \rightarrow \M(X)}$. Rather than
inverse monoids, these are examples of \emph{restriction
  monoids}~\cite{Jackson2001An-invitation, Cockett2002Restriction};
these are monoids $S$ endowed with an operation
$(\overline{\mathstrut\ }) \colon S \rightarrow S$ assigning to each
$s \in S$ an idempotent, called a \emph{restriction idempotent},
measuring its ``domain of definition''. The key example of a complete
pseudomonoid is the monoid $\M(X)$ of partial continuous endofunctions
of a space $X$; in general, they take the form
$\theta \colon S \rightarrow \M(X)$, where $S$ is a restriction monoid
with joins and $\theta$ is \emph{hyperconnected} in the sense of
inducing an isomorphism on restriction idempotents.

\textbf{Arbitrary base}. Our final axis of
generalisation is the most far-reaching: we replace \'etale
\emph{topological} groupoids with \'etale groupoids living in some
other world. By a ``world'', we here mean a \emph{join restriction
  category}~\cite{Cockett2011Differential,Guo2012Products}; these are
the common generalisation of join restriction monoids and join inverse
categories, and provide a purely algebraic setting for discussing
notions of partiality and glueing such as arise in sheaf theory.
A basic example is the category $\cat{Top}_p$ of topological spaces
and partial continuous maps, but other important examples include
$\cat{Smooth}_p$ (smooth manifolds and partial smooth maps) and
$\cat{Sch}_p$ (schemes and partial scheme morphisms).

In any join restriction category, we can define what it means for a
map to be a \emph{partial isomorphism} or a \emph{local
  homeomorphism}, and when $\C$ satisfies an additional condition of
\emph{having local glueings} (defined in
Section~\ref{sec:local-atlases} below), we can recreate the
correspondence between local homeomorphisms and sheaves within the
$\C$-world. We can then build on this, like before, to establish the
correspondence between \emph{\'etale groupoids internal to $\C$} and
\emph{complete $\C$-pseudogroups}. An \'etale groupoid internal to
$\C$ is simply an internal groupoid whose source map is a local
homeomorphism; while a complete $\C$-pseudogroup comprises a join
inverse monoid $S$, an object $X \in \C$, and a hyperconnected map
$S \rightarrow \I(X)$ into the join inverse monoid of all partial
automorphisms of $X \in \C$. 

The main result of this paper will arise from performing all four of
the above generalisations simultaneously; we can state it as
follows:
\begin{Thm*}
  Let $\C$ be a join restriction category with local glueings. There
  is an equivalence
  \begin{equation}\label{eq:65}
    \parte \simeq  \hcon
  \end{equation}
  between the category of source-\'etale partite internal categories
  in $\C$, with cofunctors as morphisms, and the category of join
  restriction categories with a hyperconnected functor to $\C$, with
  lax-commutative triangles as morphisms.
\end{Thm*}


In fact, we will derive this theorem from a more general one.
Rather than constructing the stated equivalence directly, we will
construct a larger adjunction, and then cut down the objects on
each side.

\begin{Thm*}
  Let $\C$ be a join restriction category with local glueings. There
  is an adjunction
  \begin{equation}\label{eq:64}
    \cd{
      {\partc} \ar@<-4.5pt>[r]_-{\Phi} \ar@{<-}@<4.5pt>[r]^-{\Psi} \ar@{}[r]|-{\bot} &
      {\rcat} 
    }
  \end{equation}
  between $\partc$, the category of partite internal categories in
  $\C$, with cofunctors as morphisms, and $\rcat$, the category of
  restriction categories with a restriction functor to $\C$, with
  lax-commutative triangles as morphisms.
\end{Thm*}

The left adjoint $\Psi$ of this adjunction \emph{internalises} a
restriction category over the base $\C$ to a source \'etale partite
category in $\C$; while the right adjoint $\Phi$ \emph{externalises} a
partite category in $\C$ to a join restriction category sitting above
$\C$ via a hyperconnected join restriction functor. These two
processes are evidently not inverse to each other, but constitute a
so-called \emph{Galois adjunction}; this means that applying either
$\Phi$ or $\Psi$ yields a \emph{fixpoint}, that is, an object at which
the counit $\Psi(\Phi(\mathbb{A})) \rightarrow \mathbb{A}$ or unit
$P \rightarrow \Phi(\Psi(P))$, as appropriate, is invertible. These
fixpoints turn out to be precisely the source-\'etale partite internal
categories, respectively the join restriction functors hyperconnected
over $\C$, and so by restricting the adjunction to these, we
reconstruct our main equivalence. On the other hand, we see that the
process $P \mapsto \Phi(\Psi(P))$ is a universal way of turning an
arbitrary restriction functor into a hyperconnected join restriction
functor; while $\mathbb{A} \mapsto \Psi(\Phi(\mathbb{A}))$ is a
universal way of turning an arbitrary internal partite category into a
source \'etale one.

We now give a more detailed overview of the contents of the paper
which, beyond proving our main theorem, also develops a body of
supporting theory, and gives a range of applications to problems of
practical interest. We begin in Section~\ref{sec:restr-categ-inverse}
by recalling necessary background on restriction categories and
inverse categories, along with a range of running examples. The basic
ideas here have been developed both by semigroup
theorists~\cite{Jackson2001An-invitation, Hollings2007Partial} and by
category theorists~\cite{Grandis1990Cohesive, Cockett2002Restriction},
but we tend to follow the latter---in particular because the further
developments around \emph{joins} in restriction categories are due to
this school~\cite{Cockett2011Differential, Guo2012Products}.

In Section~\ref{sec:etale-maps}, we develop the theory of \emph{local
  homeomorphisms} in a join restriction category $\C$. This builds on
a particular characterisation of local homeomorphisms of topological
spaces: they are precisely the total maps $p \colon Y \rightarrow X$
which can be written as a join (with respect to the inclusion ordering
on partial maps) of partial isomorphisms $Y \rightarrow X$. This
definition carries over unchanged to any join restriction category; we
show that it inherits many of the desirable properties of the
classical notion so long as the join restriction category $\C$ has
\emph{local glueings}---an abstract analogue of the property of being
able to build local homeomorphisms over $X$ by glueing together open
subsets of $X$.

Section~\ref{sec:local-home-sheav} exploits the preceding material to
explain how the correspondence between sheaves as functors, and
sheaves as local homeomorphisms, generalises to any join restriction
category $\C$ with local glueings. We show that, for any object
$X \in \C$, there is an equivalence as to the left in:
\begin{equation*}
  \cd{
    {\lh} \ar@<-4.5pt>[r]_-{\Gamma} \ar@{<-}@<4.5pt>[r]^-{\Delta} \ar@{}[r]|-{\sim} &
    {\shv}
  } \qquad \qquad 
  \cd{
    {\tot} \ar@<-4.5pt>[r]_-{\Gamma} \ar@{<-}@<4.5pt>[r]^-{\Delta} \ar@{}[r]|-{\bot} &
    {\pshv} 
  }
\end{equation*}
between local homeomorphisms over $X$ in $\C$, and sheaves on
$X \in \C$. Here, a \emph{sheaf} on $X \in \C$ can be defined as a
functor $\O(X)^\mathrm{op} \rightarrow \cat{Set}$ satisfying a glueing
condition, where here $\O(X)$ is the complete lattice of restriction
idempotents on $X \in \C$; in fact, we prefer to use an equivalent
formulation due to Fourman and Scott~\cite{Fourman1979Sheaves}. Just
as in the classical case (and as in our main theorem), we obtain the
desired equivalence by cutting down a Galois adjunction, as to the
right above, between \emph{total} maps over $X$ and \emph{presheaves}
on $X$.

In Section~\ref{sec:main-theorem}, we are finally ready to prove our
main result: first constructing the Galois adjunction~\eqref{eq:64},
and then restricting it back to the
equivalence~\eqref{eq:65}---exploiting along the way the
correspondence between local homeomorphisms and sheaves of the
preceding section. Having established this theorem, the remainder of
the paper is used to illustrate some of its consequences.

In Section \ref{sec:groupoid-case}, we roll back some generality by
adapting our main results to the groupoid case. An \emph{internal
  partite groupoid} is an internal partite category equipped with
inverse operations on arrows satisfying the expected identities; and
one can ask to what these correspond under the
equivalence~\eqref{eq:65}. We provide two different---albeit
equivalent---answers. The first and most immediate answer is that they
correspond to \emph{\'etale} join restriction categories
hyperconnected over the base. A join restriction category is
\emph{\'etale} (see Definition~\ref{def:24}) if every map is a join of
partial isomorphisms. The second answer, which corresponds more
closely to that provided in the theory of pseudogroups, is that they
correspond to join inverse categories (see Definition \ref{def:18})
which are hyperconnected over the base. The answers are equivalent as
the categories of join inverse categories and the category of \'elate
categories are equivalent (see Corollary \ref{cor:6}). Finally, by
restricting the generality even further back, one obtains the
correspondence between (non-partite) internal \'etale groupoids and
one-object join inverse categories hyperconnected over the base
$\C$---in other words, \emph{complete $\C$-pseudogroups}. 

Even when $\C = \cat{Top}_p$, this last result goes beyond those in
the literature by virtue of its richer functoriality. In
Section~\ref{sec:local-hyperc-maps}, we consider how this too may be
rolled back. As we have mentioned, in~\cite{Lawson2013Pseudogroups}
the morphisms considered between \'etale topological groupoids are the
so-called \emph{covering functors}; we characterise these as the
partite internal cofunctors which are bijective on components and
arrows, and show that under externalization, these correspond to what
we call \emph{localic} join restriction functors
(Definition~\ref{def:41})---our formulation of the \emph{callitic}
morphisms of~\cite{Lawson2013Pseudogroups}. The names \emph{localic}
and \emph{hyperconnected} for classes of functors originate in topos
theory, where they provide a fundamental factorization system on
geometric morphisms~\cite[\sec A4.6]{Johnstone2002Sketches}. For mere
restriction categories, a (localic, hyperconnected) factorisation
system was described and used in \cite{Cockett2014Restriction}; the
corresponding factorisation system for \emph{join} restriction
categories is, not surprisingly, fundamental to this development. The
factorization of a join restriction functor into localic and
hyperconnected parts is achieved as a direct application of our main
theorem: by internalising and then externalising the given join
restriction functor over its codomain, we reflect it into a
\emph{hyperconnected} join restriction functor, which is the second
part of the desired (localic, hyperconnected) factorisation.

In Section \ref{sec:applications} we turn to applications. First we
consider the analogues of \emph{Haefliger groupoids} in our setting.
These originate in the observation of
Ehresmann~\cite{Ehresmann1954Structures} that the space of germs of
partial isomorphisms between any two spaces is itself a space
$\Pi(A,B)$: Haefliger~\cite{Haefliger1971Homotopy} considered smooth
manifolds and the spaces $\Pi(A) = \Pi(A,A)$, which form Lie groupoids
with object space $A$. Generalising this, we define the
\emph{Haefliger category} of a join restriction category $\C$ with
local glueings as the source-\'etale partite internal category
$\H(\C)$ in $\C$ obtained by internalising the \emph{identity} functor
$\C \rightarrow \C$, and the \emph{Haefliger groupoid} $\Pi(\C)$ of
$\C$ as the internalisation of the inclusion
${\sf PIso}(\C) \subseteq \C$ of the category of partial isomorphisms
in $\C$. Because our Haefliger groupoid is \emph{partite}, it includes
in the topological case \emph{all} of the $\Pi(A,B)$'s as its spaces
of arrows. As an application, we use these ideas to show that the
category $\mathrm{Lh}(\C)$ of local homeomorphisms in $\C$ has binary
products; this generalises a result of
Selinger~\cite{Selinger1994LH-has-nonempty} for spaces.



In Sections~\ref{sec:resende-corr} and~\ref{sec:laws-lenz-corr}, we
relate our main results to the motivating ones in the literature. We
first consider the correspondence, due to
Resende~\cite{Resende2007Etale}, between \'etale groupoids in the
category of locales $\cat{Loc}_p$ and what he calls \emph{abstract
  complete pseudogroups}; in our nomenclature, these are simply join
inverse monoids. \emph{A priori}, our results yield a correspondence
between \'etale localic groupoids and complete
$\cat{Loc}_p$-pseudogroups; however, $\cat{Loc}_p$ has the special
property that any join restriction category $\A$ has an
essentially unique hyperconnected map
$\O \colon \A \rightarrow \cat{Loc}_p$ to $\cat{Loc}_p$ called the
\emph{fundamental functor}---and this allows us to identify complete
$\cat{Loc}_p$-pseudogroups with abstract complete pseudogroups, so
re-finding (with additional functoriality) Resende's result.


We then turn to the Lawson--Lenz
correspondence~\cite{Lawson2013Pseudogroups} for \'etale topological
groupoids. As in the Resende correspondence, the structures to which
such groupoids are related are not the expected complete
$\cat{Top}_p$-pseudogroups, but rather abstract complete pseudogroups;
this time, however, the correspondence is only a Galois adjunction,
restricting to an equivalence only for \emph{sober} \'etale groupoids
and \emph{spatial} complete pseudogroups. We explain this in terms of
the Galois adjunction between $\cat{Top}_p$ and $\cat{Loc}_p$, induced
by the fundamental functor
$\O \colon \cat{Top}_p \rightarrow \cat{Loc}_p$ and its right adjoint.

In Section~\ref{sec:ehresm-sche-namb}, we describe how the
Ehresmann--Schein--Nambooripad correspondence, as generalized by
DeWolf and Pronk~\cite{DeWolf2018The-Ehresmann-Schein-Nambooripad} can
be explained in terms of our correspondence; and finally in Section
\ref{sec:completion-processes} we exploit the power of having the
adjunction~\eqref{eq:64}, rather than merely the
equivalence~\eqref{eq:65}, to describe the construction of \emph{full
monoids} and \emph{relative join completions}.

%

\section{Background}
\label{sec:restr-categ-inverse}

\subsection{Restriction categories}
\label{sec:restr-categ}

We begin by recalling the notion of \emph{restriction category} which
will be central to our investigations.

\begin{Defn}\cite{Cockett2002Restriction}
  \label{def:15}
  A \emph{restriction category} is a category $\C$ equipped with an
  operation assigning to each map $f \colon A \to B$ in $\C$ a map
  $\bar f \colon A \to A$, called its \emph{restriction}, subject to
  the four axioms:
\begin{enumerate}[(i)]
\item $f \bar f = f$ for all $f \colon A \to B$;
\item $\bar f \bar g = \bar g \bar f$ for all $f \colon A \to B$ and $g \colon A \to C$;
\item $\overline{g \bar f} = \bar g \bar f$ for all $f \colon A \to B$
  and $g \colon A \to C$; 
\item $\bar g f = f \overline{gf}$ for all $f \colon A \to B$ and $g \colon B \to C$.
\end{enumerate}
\end{Defn}
Just as an inverse monoid is an abstract monoid of partial
automorphisms, so a restriction category is an abstract category of
partial maps. We now illustrate this with a range of examples of
restriction categories.
\begin{Exs}
  \label{ex:2}
\begin{itemize}[itemsep=0.25\baselineskip]
\item The category $\cat{Set}_p$ of sets and partial functions, where
  the restriction of a partial function $f \colon A \rightharpoonup B$
  is taken as the partial function $\bar f \colon A \rightharpoonup A$
  with $\bar f(a)$ taken to be $a$ if $f(a)$ is defined, and to be undefined
  otherwise. 
\item The category $\cat{Top}_p$ of topological spaces and partial
  continuous maps defined on an open subset of the domain. In this
  example, and in most of the others which follow, the
  restriction is defined just as in $\cat{Set}_p$.
\item The category $\cat{Pos}_p$ of partially ordered sets and
  partial monotone maps defined on a down-closed subset of the domain.
\item The category $\cat{Loc}_p$ of locales and partial locale maps. A
  \emph{locale} is a complete lattice $A$ satisfying the infinite
  distributive law $a \wedge \bigvee_i b_i = \bigvee_i a \wedge b_i$;
  a partial locale map $f \colon A \rightharpoonup B$ is a monotone
  function $f^\ast \colon B \rightarrow A$ preserving binary meets and
  arbitrary joins; we call $f^\ast$ the \emph{inverse image map} of
  $f$. The restriction $\bar f \colon A \rightharpoonup A$ of $f$ has
  inverse image map $a \mapsto f^\ast(\top) \wedge a$.
\item The category $\cat{Man}_p$ of topological manifolds and partial
  continuous maps defined on an open sub-manifold. Here, by a
  topological manifold, we simply mean a topological space which is
  locally homeomorphic to a Euclidean space; we do not require
  conditions such as being Hausdorff or paracompact.
\item The category $\cat{Smooth}_p$ of smooth manifolds (again, with
  no further topological conditions) and partial
  smooth maps defined on an open sub-manifold.
\item The category $\cat{L\R\T op}_p$ of locally ringed topological
  spaces and partial continuous maps. Objects are pairs $(X, \O_X)$ of
  a space $X$ and a sheaf of local rings $\O_X$ on $X$
  (cf.~\cite[Chapter II.2]{Hartshorne1977Algebraic}); a map
  $(f, \varphi) \colon (X, \O_X) \rightarrow (Y, \O_Y)$ is a
  map $f \colon X \rightarrow Y$ of $\mathrm{Top}_p$ and a local homomorphism
  $\varphi \colon f^\ast(\O_Y) \rightarrow \smash{\res{\O_X}{\mathrm{dom}(f)}}$ of sheaves
  of rings on $\mathrm{dom}(f) \subseteq X$. The restriction of $(f, \varphi)$ is
  $(\smash{\bar f}, \mathrm{id})$ where the first component is restriction in
  $\cat{Top}_p$.
\item The full subcategory $\cat{Sch}_p$ of $\cat{L\R\T op}_p$ whose
  objects are schemes, i.e., locally ringed spaces $(X, \O_X)$ for
  which there is an open cover $X = \bigcup_i U_i$ such that each $(U_i,
  \res{\O_X}{U})$ is isomorphic to an affine scheme
  $(\mathrm{Spec}\,R, \O_{\mathrm{Spec}\,R})$.

\item The \emph{category of partial maps} $\cat{Par}(\C, \M)$, where $\C$ is
  a category and $\M$ a pullback-stable, composition-closed class
  of monics in $\C$. Its objects are those
  of $\C$; maps $X \rightharpoonup Y$ are isomorphism-classes of
  spans $m \colon X \leftarrowtail X' \rightarrow Y \colon f$ with $m
  \in \M$; and the restriction of $[m,f] \colon X
  \rightharpoonup Y$ is $[m,m] \colon X \rightharpoonup X$. 
\item The category $\cat{Bun}(\C)$ of bundles in a restriction
  category $\C$. Objects of $\cat{Bun}(\C)$ are \emph{total} maps (as
  defined below) $x \colon X' \rightarrow X$ in $\C$, while morphisms
  are commuting squares. Restriction is given componentwise:
  $\overline{(f,g)} = (\overline{f}, \overline{g})$.
  
\item Any meet-semilattice $M$ can be seen as a one-object restriction
  category $\Sigma M$, where maps are the elements of $M$, composition
  is given by binary meet, and every map is its own restriction.
\end{itemize}
\end{Exs}

It follows from the axioms that the operation $f \mapsto \bar f$ in a
restriction category is idempotent, and that maps of the form $\bar f$
are themselves idempotent; we call them \emph{restriction
  idempotents}, and write $\O(A)$ for the set of restriction
idempotents of $A \in \C$. As presaged by the last example above,
$\O(A)$ is a meet-semilattice under the operation of composition.
We call a map $f \in \C(A,B)$ \emph{total} if $\overline{f} = 1_A$.
Each hom-set $\C(A,B)$ in a restriction category comes endowed with
two relations: the \emph{canonical partial order} $\leqslant$, and the
\emph{compatibility relation} $\smile$, defined by
\begin{equation*}
  f \leqslant g \text{ iff } f = g\overline{f} \qquad \text{and}
  \qquad f \smile g \text{ iff } f \overline{g} = g \overline{f}\rlap{ .}
\end{equation*}
When $\C = \cat{Set}_p$, the natural partial order is given by
inclusion of graphs, while $f \smile g$ precisely when $f$ and $g$
agree on their common domain of definition.


\begin{Defn}
  \label{def:19}
  A functor $F \colon \C \rightarrow \D$ between restriction
  categories is a \emph{restriction functor} if
  $F\overline{f} = \overline{Ff}$ for all $f \in \C(A,B)$. A
  restriction functor is \emph{hyperconnected} if each of the induced
  functions $\O(A) \rightarrow \O(FA)$ is an isomorphism. We write
  $\mathrm{r}\cat{Cat}$ for the $2$-category of restriction
  categories, restriction functors, and natural transformations with
  total components.
\end{Defn}


\subsection{Join restriction categories}
\label{sec:join-restr-categ}
We now turn our attention to restriction categories in which
compatible families of parallel maps can be glued together.
\begin{Defn}
  \label{def:17}
  A restriction category $\C$ is a \emph{join restriction
    category}~\cite{Cockett2011Differential,Guo2012Products} if every
  pairwise-compatible family of maps in $\C(A,B)$ admits a join $\bigvee_i f_i$
  with respect to the natural partial order, and these joins are
  preserved by precomposition, i.e.,
  \begin{equation}\label{eq:27}
    \textstyle(\bigvee_{i}f_i) g = \bigvee_{i}(f_ig) 
    \text{for all $g \in \C(A',A)$.}
  \end{equation}
\end{Defn}

The restriction to pairwise-compatible families in this definition is
necessary; indeed, if the parallel maps $\{f_i\}$ have \emph{any}
upper bound $h$, then
$f_i \overline{f_j} = h \overline{f_i}\, \overline{f_j} = h
\overline{f_j}\, \overline{f_i} = f_j \overline{f_i}$ so that
$f_i \smile f_j$ for each $i,j$.

\begin{Lemma}
  \label{lem:18}
  In a join restriction category $\C$ we have that:
  \begin{enumerate}[(i)]
  \item $\overline{\bigvee_{i} f_i} = \bigvee_{i}\overline{f_i}$ for
    all compatible families $\{f_i \colon A \rightarrow B\}$;
  \item $h(\bigvee_i f_i) = \bigvee_i hf_i$ for all compatible $\{f_i
    \colon A \rightarrow B\}$ 
    and all $h \colon B \rightarrow B'$.
  \end{enumerate}
\end{Lemma}
\begin{proof}
  See~\cite[Proposition~2.14]{Cockett2011Differential}.
\end{proof}

\begin{Exs}
  \label{ex:3}
  \begin{itemize}[itemsep=0.25\baselineskip]
  \item   Each of $\cat{Set}_p$, $\cat{Top}_p$, $\cat{Pos}_p$,
  $\cat{Man}_p$ and $\cat{Smooth}_p$ is a join
  restriction category. In $\cat{Set}_p$, a family of partial
  functions $f_i \colon A \rightharpoonup B$ is pairwise-compatible
  just when the union of the graphs of the $f_i$'s is again the graph
  of a partial function---which is then the join $\bigvee_i f_i$. Much
  the same thing happens in $\cat{Top}_p$, $\cat{Pos}_p$,
  $\cat{Man}_p$ and $\cat{Smooth}_p$. 
\item   $\cat{Loc}_p$ is a join restriction category. A family of partial
  locale maps $f_i \colon A \rightharpoonup B$ is compatible when
  $f_i^\ast(b) \wedge f_j^\ast(\top) \leqslant f_j^\ast(b)$ for all
  $i,j$, and the corresponding join has inverse image map
  $(\bigvee f_i)^\ast(b) = \bigvee_i (f_i^\ast b)$.
\item $\cat{L\R\T op}_p$ is a join restriction category. Given a
  compatible family of maps
  $(f_i, \varphi_i) \colon (X, \O_X) \rightarrow (Y, \O_Y)$, where
  each $f_i$ is defined on $U_i \subseteq X$, we first glue the
  compatible maps $f_i$ to a continuous map $f \colon X \rightarrow Y$
  defined on $U = \bigcup_i U_i$. Now on $U$ we may consider the sheaf
  $\mathrm{Hom}_\ell(f^\ast \O_Y, \res {\O_X} U)$ of local
  homomorphisms of rings; the family of maps $\varphi_i$ is a matching
  family of local sections, and so can be glued to a global section
  $\varphi \colon f^\ast \O_Y \rightarrow \res{\O_X}U$. The map
  $(f, \varphi)$ so obtained is the desired join of the
  $(f_i, \varphi_i)$'s. It follows that the full subcategory
  $\cat{Sch}_p$ of schemes is also a join restriction category.

\item   A restriction category of the form $\cat{Par}(\C, \M)$ is a join
  restriction category precisely when $\C$ admits pullback-stable unions
  of $\M$-subobjects which are computed as the pushout over the
  pairwise intersections; see~\cite[Theorem~11]{Lin2019Presheaves}. In
  particular, if $\E$ is a Grothendieck topos, and $\M$ comprises all
  the monomorphisms, then $\cat{Par}(\E, \M)$ is a join restriction
  category.
\item If $\C$ is a join restriction category, then so too is the
  restriction category of bundles $\cat{Bun}(\C)$, with joins computed pointwise.
\item If $M$ is a meet-semilattice, then the corresponding one-object
  restriction category $\Sigma M$  is a join restriction category just
  when $M$ is a locale.
  \end{itemize}
\end{Exs}
In general, if $A$ is an object of a join restriction category, then
any subset of the set of restriction idempotents $\O(A)$ is
compatible; whence each $\O(A)$ is a
locale.
The joins of this locale endow each hom-set $\C(A,B)$ with an
``$\O(A)$-valued equality'', which will be important in what follows.

\begin{Defn}
  \label{def:34}
  For maps $f, g \colon A \rightrightarrows B$ in a join
  restriction category $\C$ we define
  \begin{equation*}
    \sheq f g = \bigvee \{e \in \O(A) : e \leqslant \bar f \bar g
    \text{ and } fe = ge\}\rlap{ .}
  \end{equation*}
\end{Defn}
\begin{Lemma}
  \label{lem:30}
  For any $f,g,h \colon A \rightarrow B$ in a join restriction
  category $\C$:
  \begin{enumerate}[(i)]
  \item If $f \smile g$ then $\sheq f g = \bar f \bar g$; in
    particular, if $f \leqslant g$ then $\sheq f g = \bar f$.
  \item $\sheq f g = \sheq g f$.
  \item $\sheq g h \sheq f g \leqslant \sheq f h$.
  \item $f \sheq f g = g \sheq f g$, and this is the meet $f \wedge g$ 
    in $(\C(A,B), \leqslant)$.
  \end{enumerate}
\end{Lemma}
(Note that the meets described in (iv) do not necessarily make $\C$
into a \emph{meet restriction category} in the sense
of~\cite[Chapter~4]{Guo2012Products}; one might call the present meet
an ``interior'' meet which arises purely from the join structure.)
\begin{proof}
  (i) and (ii) are clear. For (iii), by distributivity of
  joins we have
  \begin{equation*}
    \sheq g h\sheq f g = \bigvee \{ed : d \leqslant
    \overline{f}\overline{g}, e \leqslant \overline{g}\overline{h},
    fd = gd, ge = he\}\rlap{ ,}
  \end{equation*}
  so it suffices to show each $ed$ as to the right is an
  element of the join defining $\sheq f h$. For this we observe that
  $ed \leqslant \bar g \bar h \bar f \bar g \leqslant \bar f\, \bar h$
  and $fed = fde = gde = ged = hed$, as desired. Finally, for (iv), we
  have by distributivity that
  \begin{equation*}
    f \sheq f g = \bigvee \{fe : e \leqslant \overline{f}
    \overline{g}, fe=ge\} = \bigvee \{ge : e \leqslant \overline{f}
    \overline{g}, fe=ge\} = g \sheq f g\rlap{ .}
  \end{equation*}
  To prove this element is $f \wedge g$, observe that if
  $h \leqslant f$ and $h \leqslant g$ then $h = f \bar h = g \bar h$, so
  that $\bar h \leqslant \sheq f g$ and so $h = f \bar h \leqslant f
  \sheq f g$ as desired.
\end{proof}

Just as before, join restriction categories assemble into a $2$-category.

\begin{Defn}
  \label{def:32}
  A \emph{join restriction functor} between join restriction
  categories is a restriction functor $F \colon \C \rightarrow \D$
 which preserves joins in that $F(\bigvee_i f_i) = \bigvee_i Ff_i$.
  We write $\mathrm{jr}\cat{Cat}$ for the $2$-category of join
  restriction categories, join restriction functors and total
  transformations.
\end{Defn}

\begin{Lemma}
  \label{lem:19}
  A restriction functor $F \colon \C \rightarrow \D$ between join
  restriction categories preserves joins whenever it preserves joins
  of restriction idempotents---in particular, whenever $F$ is hyperconnected.
\end{Lemma}
\begin{proof}
  We always have $\bigvee_i Ff_i \leqslant F(\bigvee_i f_i)$. If $F$
  preserves joins of restriction idempotents, then 
  $\overline{\bigvee_i Ff_i} = \bigvee_i F\overline{f_i} = F\bigvee_i
  \overline{f_i} = \overline{F\bigvee_i f_i }$ whence
  $\bigvee_i Ff_i = F(\bigvee_i f_i)$.
\end{proof}

\subsection{Inverse categories}
\label{sec:inverse-categories}

In the introduction, we defined an \emph{inverse category} to be a
category $\I$ such that for each $s \in \I(A,B)$, there is a unique
$s^\ast \in \I(B,A)$ such that $ss^\ast s = s$ and
$s^\ast ss^\ast = s^\ast$. By a standard
argument~\cite[Theorem~5]{Lawson1998Inverse}, it is equivalent to ask
that there exists \emph{some} $s^\ast$ with these properties, and that
additionally all idempotents in $\C$ commute. Since inverse structure
is equationally defined, it is preserved by any functor; it thus makes
sense to give:

\begin{Defn}
  \label{def:21}
  $\mathrm{i}\cat{Cat}$ is the $2$-category of inverse
  categories, arbitrary functors, and arbitrary natural isomorphisms.
\end{Defn}

Importantly, any inverse category $\I$ is a restriction category on
defining the restriction of $s \colon A \rightarrow B$ to be
$\bar s = s^\ast s \colon A \rightarrow A$;
see~\cite[Theorem~2.20]{Cockett2002Restriction}. In this way, we
obtain a full embedding of $2$-categories
$\mathrm{i}\cat{Cat} \rightarrow \mathrm{r}\cat{Cat}$. The image of
this embedding can be characterised in terms of the following notion.

\begin{Defn}
  \label{def:16}
  A \emph{partial isomorphism} in a restriction category $\C$ is a map
  ${s \colon A \rightarrow B}$ with a \emph{partial
    inverse}: a map $s^\ast \colon B \rightarrow A$ with
  $s^\ast s = \overline{s}$ and $ss^\ast = \overline{s^\ast}$.
\end{Defn}
\begin{Exs}
  \label{ex:11}
  \begin{itemize}[itemsep=0.25\baselineskip]
  \item A partial isomorphism $s \colon X \rightharpoonup Y$ in
    $\cat{Set}_p$ is one whose graph describes a bijection between a
    subset $A \subseteq X$ and a subset $B \subseteq Y$. The same
    holds in $\cat{Top}_p$, $\cat{Pos}_p$, $\cat{Man}_p$ and
    $\cat{Smooth}_p$, where, for example, in $\cat{Top}_p$ this involves a
    bijection between \emph{open} subsets of $X$ and $Y$.
  \item A partial locale map $f \colon A \rightharpoonup B$ is a
    partial isomorphism when there exist $a \in A$ and $b \in B$ such that $f^\ast$
    maps $\mathord\downarrow b$ bijectively onto $\mathord\downarrow
    a$ and we have a factorisation
    \begin{equation*}
      f^\ast = B \xrightarrow{(\thg) \wedge b} \mathord\downarrow b
      \xrightarrow{f^\ast} \mathord\downarrow a
      \xhookrightarrow{\text{include}} A\rlap{ .}
    \end{equation*}
  \item Partial isomorphisms in $\cat{Par}(\C, \M)$ are spans 
    of the form $m \colon X \leftarrowtail Z \rightarrowtail Y \colon
    n$ for which both $m$ and $n$ are in $\M$.
  \end{itemize}
\end{Exs}
The following lemma describes the key properties of partial
isomorphisms.
\begin{Lemma}
  \label{lem:9}
  \begin{enumerate}[(i)]
  \item If $s^\ast$ is partial inverse to $s$, then $ss^\ast s = s$
    and $s^\ast s s^\ast = s$.
  \item Partial inverses are unique when they exist.
  \item Partial isomorphisms are composition-closed with $(st)^\ast =
    t^\ast s^\ast$.
  \item The idempotent partial isomorphisms are exactly the restriction
    idempotents.
  \item If $s, t \colon A \rightarrow B$ are partial isomorphisms,
    then $s \smile t$ just when $st^\ast$ is a restriction
    idempotent.
  \end{enumerate}
\end{Lemma}
\begin{proof}
  For (i), we have $s = s \bar s = ss^\ast s$ and $s^\ast = s^\ast \bar
  {s^\ast} = s^\ast ss^\ast$. For (ii), if $t$
  and $u$ are partial inverse to $s$ then $t = tst = \bar s t = ust = u \bar
  t$. So $t \leqslant u$; dually $u \leqslant t$ and so $t =
  u$.
  For (iii), we calculate that $s^\ast t^\ast ts = s^\ast
  \overline{t}s = s^\ast s \overline{ts} = \overline{s}\,\overline{ts}
  = \overline{ts}$ and dually $tss^\ast t^\ast = \overline{s^\ast
    t^\ast}$; whence $(ts)^\ast = s^\ast t^\ast$ by (ii).

  For (iv), each restriction idempotent is clearly its own partial
  inverse; conversely, if $e$ is an idempotent partial isomorphism
  then
  $\overline{e} = e^\ast e = e^\ast ee = \overline{e}e = e
  \overline{ee} = e\overline{e} = e$, so that $e$ is a restriction
  idempotent. Finally, for (v), if $s \smile t$ then
  $st^\ast t = ts^\ast s$, whence
  $st^\ast = st^\ast tt^\ast = ts^\ast st^\ast = \overline{st^\ast}$
  is a restriction idempotent. Conversely, if $st^\ast$ is a
  restriction idempotent, then it is its own partial inverse, so 
  $s t^\ast = ts^\ast$ and
  \begin{equation*}
  s \overline{t} = st^\ast t = st^\ast st^\ast t = ts^\ast st^\ast t
  = t \overline{s}\,\overline{t} = t \overline{s}\rlap{ .}\qedhere
\end{equation*}
\end{proof}
Using this, we thus obtain the following, which is again~\cite[Theorem~2.20]{Cockett2002Restriction}:
\begin{Prop}
\label{prop:28}
A restriction category $\C$ is in the essential image of
$\mathrm{i}\cat{Cat} \rightarrow \mathrm{r}\cat{Cat}$ just when every
morphism in $\C$ is a partial isomorphism.
\end{Prop}
We are thus justified in confusing an inverse category with the
corresponding restriction category. We also obtain Proposition~2.24 of
\emph{loc.~cit.}

\begin{Cor}
  \label{cor:4}
  The full inclusion $2$-functor
  $\mathrm{i}\cat{Cat} \rightarrow \mathrm{r}\cat{Cat}$ has a right
  $2$-adjoint
  $\mathrm{PIso} \colon \mathrm{r}\cat{Cat} \rightarrow
  \mathrm{i}\cat{Cat}$ sending each restriction category $\C$ to its
  subcategory $\mathrm{PIso}(\C)$ comprising all objects and all
  partial isomorphisms between them.
\end{Cor}

For example, $\mathrm{PIso}(\cat{Set}_p)$ is the inverse category
$\cat{P\I nj}$ of
sets and partial isomorphisms, $\mathrm{PIso}(\mathrm{Top}_p)$ is the
category of spaces and partial homeomorphisms, and so on.

\subsection{Join inverse categories}
\label{sec:join-inverse-categ}
Joins in inverse categories are slightly different from joins in
restriction categories. We say that parallel maps
$s,t \colon A \rightarrow B$ in an inverse category are
\emph{bicompatible}, written $s \asymp t$, if both $s \smile t$ and
$s^\ast \smile t^\ast$.

\begin{Defn}
  \label{def:18}
  An inverse category $\I$ is a \emph{join inverse category} if every
  pairwise-bicompatible family of maps $\{s_i\} \subseteq \I(A,B)$ admits a join
  $\bigvee_i s_i$ with respect to the natural partial order, and these
  joins are preserved by precomposition. We write
  $\mathrm{ji}\cat{Cat}$ for the $2$-category of join inverse
  categories, functors preserving joins, and natural isomorphisms.
\end{Defn}
This time, it is pairwise-\emph{bi}compatibility that is a necessary
condition for a join to exist in an inverse category. 
The gap between
compatibility and bicompatibility complicates the relation between
join inverse and join restriction categories. The more
straightforward direction is from join restriction to join inverse.
\begin{Lemma}
  \label{lem:22}
  Let $\{s_i\}$ be a compatible family of partial isomorphisms in a
  join restriction category. The join $\bigvee_i s_i$ is a partial
  isomorphism if and only if $\{s_i\}$ is a bicompatible family.
\end{Lemma}
\begin{proof}
  If $\{s_i\}$ is bicompatible, then $\{s_j^\ast\}$ is compatible, and
  so $\bigvee_j s_j^\ast$ exists. Now compatibility of $\{s_i\}$ shows
  by Lemma~\ref{lem:9}(v) that
  each $s_i s_j^\ast$ is a restriction idempotent, whence
  $s_i s_j^\ast = \overline{s_i s_j^\ast} \leqslant
  \overline{s_j^\ast} = s_j s_j^\ast$. This gives the starred equality in
  \begin{equation*}
    \textstyle(\bigvee_i s_i)(\bigvee_j s_j^\ast) = \bigvee_{i,j}
    s_is_j^\ast \stackrel{(\ast)}{=} \bigvee_j s_js_j^\ast  =  
    \bigvee_j \overline{s_j^\ast} = \overline{\bigvee_j s_j^\ast} \rlap{ .}
  \end{equation*}
  This, together with the dual calculation (swapping the $s_i$'s and
  $s_i^\ast$'s) shows that
  $\bigvee_i s_i$ and $\bigvee_j s_j^\ast$ are partial inverse.
  Suppose conversely that $s = \bigvee_i s_i$ is a partial
  isomorphism. Since $s_i \leqslant s$ for each $i$, also
  $s_i^\ast \leqslant s^\ast$ by Lemma~\ref{lem:9}(iii), and so both
  $\{s_i\}$ and $\{s_i^\ast\}$ are pairwise-compatible, i.e.,
  $\{s_i\}$ is pairwise-bicompatible.
\end{proof}
It follows from this that:
\begin{Prop}\label{prop:16}
  If $\C$ is a join restriction category, then $\mathrm{PIso}(\C)$ is
  a join inverse category; this assignation yields a $2$-functor
  $\mathrm{PIso} \colon \mathrm{jr}\cat{Cat} \rightarrow
  \mathrm{ji}\cat{Cat}$.
\end{Prop}

In the other direction, however, a join inverse category \emph{qua}
restriction category is \emph{not} typically a join restriction
category, since it has only joins of bicompatible families and not
compatible ones. However, we can freely adjoin these:
\begin{Defn}
  \label{def:22}
  Let $\I$ be a join inverse category. Its \emph{join restriction
    completion} $\mathrm{jr}(\I)$ has the same objects as $\I$, while
  a morphism $A \rightarrow B$ is a compatible subset
  $S \subseteq \I(A,B)$ which is closed downwards and under joins of
  bicompatible families. The identity on $A$ is the downset
  $\mathop\downarrow 1_A$; the composite of $S \colon A \rightarrow B$
  and $T \colon B \rightarrow C$ is the closure of
  $\{ts : t \in T, s\in S\}$ downwards and under bicompatible joins;
  and the restriction $\overline{S} \colon A \rightarrow A$ is 
  the downset of $\bigvee_{s\in S} \overline{s}$.
\end{Defn}

\begin{Prop}
  \label{prop:12}
  $\mathrm{jr}(\I)$ is a join restriction category for any join
  inverse category $\I$. It provides the value at $\I$ of a left
  $2$-adjoint to
  $\mathrm{PIso} \colon \mathrm{jr}\cat{\cat{Cat}} \rightarrow
  \mathrm{ji}\cat{Cat}$.
\end{Prop}
\begin{proof}
  See~\cite[Theorem~3.1.33]{Guo2012Products}.
\end{proof}

The unit of the $2$-adjunction $\mathrm{jr} \dashv \mathrm{PIso}$ at
$\I \in \mathrm{ji}\cat{Cat}$ is the join-preserving functor
$\eta_\I \colon \I \rightarrow \mathrm{PIso}(\mathrm{jr}(\I))$ sending
$f$ to $\mathop \downarrow f$; while the counit at
$\C \in \mathrm{jr}\cat{Cat}$ is the join restriction functor
$\varepsilon_\C \colon \mathrm{jr}(\mathrm{PIso}(\C)) \rightarrow \C$ sending
$S \colon A \rightarrow B$ to
$\bigvee S \colon A \rightarrow B$. What was not observed
in~\cite{Guo2012Products} is that:

\begin{Prop}
  \label{prop:14}
  The unit $\eta$ of the $2$-adjunction $\mathrm{jr} \dashv \mathrm{PIso}$ is
  an isomorphism. Each counit component $\varepsilon_\C$ is bijective on objects and
  faithful, and is full precisely when each map in $\C$ is a join of
  partial isomorphisms.
\end{Prop}
\begin{proof}
  Clearly each
  $\eta_\I \colon \I \rightarrow \mathrm{PIso}(\mathrm{jr}(\I))$ is
  bijective on objects and faithful. To show it is full, consider a map
  $S \colon A \rightarrow B$ in $\mathrm{jr}(\I)$ which is a
  partial isomorphism. Note that $S = \bigvee_{s \in S} \eta_\I(s)$ in
  $\mathrm{jr}(\I)$, so that by Proposition~\ref{prop:15},
  $\{\eta_\I(s) : s \in S\}$ is a bicompatible family. By fidelity of
  $\eta_\I$, also $S$ is a bicompatible family in
  $\I$ so that $\bigvee S$ exists. But now
  $\eta_\I(\bigvee S) = \bigvee_{s \in S} \eta_\I(s) = S$ so that $\eta_\I$
  is full.

  Turning now to
  $\varepsilon_\C \colon \mathrm{jr}(\mathrm{PIso}(\C)) \rightarrow
  \C$, clearly it is bijective on objects. For fidelity, we claim
  that any $S \colon A \rightarrow B$ in
  $\mathrm{jr}(\mathrm{PIso}(\C))$ is determined by
  $\varepsilon_\C(S)$ as
  $S = \{\, t \in \C(A,B) : t \text{ a partial isomorphism and }t
  \leqslant \varepsilon_\C(S)\,\}$. For the non-trivial inclusion,
  suppose that $t \leqslant \varepsilon_\C(S) = \bigvee S$ is a
  partial isomorphism; then we have
  $t = (\bigvee S) \overline{t} = \textstyle\bigvee_{s \in S}
  s\overline{t}$. The family $\{s \overline{t} : s \in S\}$ is
  contained in $S$, since $S$ is down-closed; it is also bicompatible,
  since it is bounded above by the partial isomorphism $t$. Since $S$
  is closed under bicompatible joins, it follows that
  $t = \bigvee_{s} s\overline{t} \in S$ as required.

  Finally, we consider when $\varepsilon_\C$ is full. If
  $p = \bigvee_{i}p_i \colon A \rightarrow B$ is any join of partial
  isomorphisms in $\C$, then on taking $S \subseteq \C(A,B)$ to be the
  closure of $\{p_i\}$ downwards and under bicompatible joins, we have
  $p = \bigvee_{i}p_i = \bigvee S = \varepsilon_\C(S)$. So the maps in
  the image of $\varepsilon_\C$ are precisely the joins of partial
  isomorphisms, and $\varepsilon_\C$ is full just when every map in
  $\C$ is of this form.
  %
\end{proof}
This leads us to make:
\begin{Defn}
  \label{def:24}
  An \emph{\'etale} map in a join restriction category $\C$ is a join
  of partial isomorphisms. We call $\C$ an \emph{\'etale join
    restriction category} if all its maps are~\'etale.
\end{Defn}
And we thus obtain:
\begin{Cor}
  \label{cor:6}
  The $2$-functor
  $\mathrm{jr} \colon \mathrm{ji}\cat{Cat} \rightarrow
  \mathrm{jr}\cat{Cat}$ is a full coreflective embedding of
  $2$-categories; its essential image comprises the \'etale join
  restriction categories. The induced coreflector
  $\mathrm{jr} \circ \mathrm{PIso}$ on $\mathrm{jr}\cat{Cat}$ sends a
  join restriction category to its subcategory $\mathrm{Et}(\C)$ of
  \'etale maps.
\end{Cor}

\section{The theory of local homeomorphisms}
\label{sec:etale-maps}

The previous result highlights the importance of \'etale maps in join
restriction categories. Of particular relevance to us will be the
total \'etale maps, which we call \emph{local homeomorphisms}.
\begin{Exs}
  \label{ex:12} 
  \begin{itemize}[itemsep=0.25\baselineskip]
  \item Any map in $\cat{Set}_p$ is a local homemorphism.
  \item The local homeomorphisms in $\cat{Top}_p$ or $\cat{Man}_p$ are
    the local
    homeomorphisms in the usual sense, while in $\cat{Smooth}_p$ they
    are the local diffeomorphisms. 
  \item A total map $f \colon P \rightarrow Q$ in $\cat{Pos}_p$ is
    a local homeomorphism if and only if it is a \emph{discrete
      fibration}: that is, for any $p \in P$ and $q \leqslant f(p)$ in
    $Q$, there exists a unique $p' \leqslant p$ in $P$ with
    $f(p') = q$.
  \item A total map $f \colon L \rightarrow M$ in $\cat{Loc}_p$ is a
    local homeomorphism just when there is a decomposition $\top_L =
    \bigvee_i \varphi_i$ and corresponding elements $\psi_i \in M$
    such that for each $i$, the map $m \mapsto f^\ast(m) \wedge
    \varphi_i$ gives a
    bijection $\mathord \downarrow \psi_i \rightarrow
    \mathord\downarrow \varphi_i$.
  \item A total map
    $(f, \varphi) \colon (X, \O_X) \rightarrow (Y, \O_Y)$ in
    $\cat{L\R\T op}_p$ is a local homeomorphism just when
    $f \colon X \rightarrow Y$ is a local homeomorphism of spaces and
    $\varphi$ is invertible; and similarly in the full subcategory of
    schemes (such maps are usually called \emph{local isomorphisms} of
    schemes.)
  \item In a join restriction category 
    $\cat{Par}(\C, \M)$, the total map specified by 
    $f \colon X \rightarrow Y$ in $\C$ is a local homeomorphism 
    just when there is a jointly surjective family of $\M$-maps
    $(m_i \colon X_i \rightarrowtail X)_{i \in I}$ such that each
    $fm_i \colon X_i \rightarrow Y$ is in $\M$.
  \item If $\C$ is a join restriction category, then a local
    homeomorphism in $\cat{Bun}(\C)$ over an object $x \colon X'
    \rightarrow X$ can be shown to be a \emph{pullback} square
    \begin{equation}\label{eq:58}
      \cd{
        {A'} \ar[r]^-{a} \ar[d]_{p'} &
        {A} \ar[d]^{p} \\
        {X'} \ar[r]_-{x} &
        {X}
      }
    \end{equation}
    where both $p$ and $p'$ are local homeomorphisms.
  \end{itemize}
\end{Exs}
In this section, we redevelop the classical theory of local
homeomorphisms between spaces in the context of a general join
restriction category $\C$.

%

\subsection{Local atlases and local glueings }
\label{sec:local-atlases}


If $A \rightarrow X$ is a local homeomorphism in a join restriction
category, then we may think of $A$ as being built by patching together
local fragments of $X$. The following structure specifies the data for
such a patching; in the nomenclature of~\cite{Fourman1979Sheaves}, it
would be called an \emph{$\O(X)$-set}.
\begin{Defn}
  \label{def:23}
  A \emph{local atlas} on an object $X$ of a join restriction category
  comprises a set $I$ and a matrix of restriction idempotents
  $\bigl(\varphi_{ij} \in \O(X) : i,j \in I\bigr)$ such that
  \begin{equation*}
    \varphi_{ij} = \varphi_{ji} \qquad \text{and} \qquad 
    \varphi_{jk} \varphi_{ij} \leqslant \varphi_{ik}.
  \end{equation*}
\end{Defn}
The manner in which a local homeomorphism induces a local atlas is
given by:

\begin{Lemma}
  \label{lem:24}
  Let $p = \bigvee_i p_i \colon A \rightarrow X$ be a local
  homeomorphism, and let each $p_i$ have the partial inverse
  $s_i \colon X \rightarrow A$. The family $\varphi_{ij} = p_j s_i$ is
  a local atlas on $X$.
\end{Lemma}
\begin{proof}
  By compatibility of the $p_i$'s and Lemma~\ref{lem:9}(v), each
  $p_j s_i$ is a restriction idempotent. Since a restriction
  idempotent is its own restriction inverse, it follows that
  $p_j s_i = p_i s_j$. Finally, we calculate that
  $p_k s_j p_j s_i = p_k \overline{p_j} s_i \leqslant p_k s_i$.
\end{proof}

In fact, the local atlas associated to a local homeomorphism $p$
determines it up to unique isomorphism over $X$. This follows \emph{a
  fortiori} from the following result, which states that maps out of
the domain of a local homeomorphism $A \rightarrow X$ are determined
by their action on each of the patches of $X$ from which $A$ is built.
(In fact, maps into $A$ are also so determined, but we will not need this).
\begin{Lemma}
  \label{lem:10}
  Let $p = \bigvee_i p_i \colon A \rightarrow X$ be a local
  homeomorphism where each $p_i$ has partial inverse $s_i$. Let
  $\varphi_{ij} = p_j s_i$ be the associated local
  atlas. 
Precomposition with the $s_i$'s
  induces a bijection between maps $f \colon A \rightarrow B$ and
  families of maps $(f_i \colon X \rightarrow B)_{i \in I}$ such that
  \begin{equation}\label{eq:5}
    f_i \varphi_{ii} = f_i \qquad \text{and} \qquad f_j \varphi_{ij}
    \leqslant f_i \qquad \text{for all $i,j \in I$.}
  \end{equation}
  Under this bijection, $f$ is total if and only if $\overline{f_i} =
  \varphi_{ii}$ for each $i \in I$.
\end{Lemma}
\begin{proof}
  Since $s_i \varphi_{ii} = s_i p_i s_i = s_i$ and
  $s_j \varphi_{ij} = s_j p_j s_i = \overline{p_j} s_i \leqslant s_i$,
  any family of maps $(fs_i)$ satisfies~\eqref{eq:5}. Conversely, for
  a family $(f_i)$ satisfying~\eqref{eq:5}, we have 
  $f_j p_j \overline{f_i p_i} \leqslant f_j p_j \overline{p_i} = f_j
  p_j s_i p_i = f_j \varphi_{ij} p_i \leqslant f_i p_i$ so that the
  family  $(f_i p_i \colon A \rightarrow B)$ is compatible. Let
  $f = \bigvee_i f_i p_i$ and note 
  $fs_i = \bigvee_j f_j p_j s_i = f_i \varphi_{ii} \vee \bigvee_{j
    \neq i} f_j \varphi_{ij} = f_i$, where the last equality
  uses both clauses in~\eqref{eq:5}. On the other hand, if
  $g \colon A \rightarrow B$ satisfies $gs_i = f_i$ for all $i$, then
  $g = \bigvee_i g \,\overline{p_i} = \bigvee_i g s_i p_i = \bigvee_i f s_i p_i = f$. Finally, if $f$
  is total, then
  $\overline{fs_i} = \overline{s_i} = p_i s_i = \varphi_{ii}$.
  Conversely, if $\overline{f_i} = \varphi_{ii}$ for each $i$, then
  $\overline{f} = \bigvee_i \overline{f_ip_i} = \bigvee_i
  \overline{\overline{f_i}p_i} = \bigvee_i
  \overline{\varphi_{ii}p_i} = \bigvee_i \overline{p_i s_i p_i} =
  \bigvee_i \overline{p_i} = 1_A$.
\end{proof}

The preceding result implies that local homeomorphisms are completely
determined by their local atlases. However, some local atlases in $\C$
may not be induced by any local homeomorphism; the following
definition rectifies this.

\begin{Defn}
  \label{def:25}
  A join restriction category $\C$ has \emph{local glueings} if every
  local atlas $\varphi$ on every $X \in \C$ is induced by some local
  homeomorphism $p \colon A \rightarrow X$. We call $p$ (or sometimes
  merely the object $A$) a \emph{glueing} of the local atlas
  $\varphi$.
\end{Defn}

\begin{Ex}
  \label{ex:5}
  The join restriction category $\cat{Top}_p$ has local glueings. For
  any $X \in \cat{Top}_p$, we can identify $\O(X)$ with the lattice of
  open subsets of $X$. Thus an $I$-object local atlas on a
  space $X$ comprises a family of open subsets
  $(U_{ij})_{i,j \in I}$ satisfying $U_{ij} = U_{ji}$ and
  $U_{jk} \cap U_{ij} \subseteq U_{ik}$ for all $i,j,k \in I$. 

  In this situation, for each $x \in X$ we may define on the set
  $\{i \in I : x \in U_{ii}\}$ an equivalence relation $\sim_x$ with
  $i \sim_x j$ just when $x \in U_{ij}$. If $i \in I$ with
  $x \in U_{ii}$, then we write $i_x$ for its $\sim_x$-equivalence
  class, and call it the \emph{germ of $i$ at $x$}. We write $U_x$ for
  the set of all germs at $x$. Now the glueing $p \colon A \rightarrow X$
  of the given local atlas has $A \defeq \sum_{x \in X} U_x$ the
  disjoint union of the sets of $U$-germs, with topology generated by
  the basic open sets
  \begin{equation*}
    \spn{i,V} = \{ i_x : x \in V\} \qquad \text{for all $i \in I$ and open $V
      \subseteq U_{ii}$}\rlap{ .}
  \end{equation*}

  The projection $p \colon A \rightarrow X$ sends $i_x$ to $x$;
  it is a local homeomorphism since $p = \bigvee_i p_i$, where
  $p_i \colon A \rightarrow X$ is the partial isomorphism
  obtained by restricting $p$ to the open set $\spn{i, U_{ii}}$, with
  partial inverse $s_i \colon X \rightarrow A$ given by $x \mapsto
  i_x$ on the open set $U_{ii}$.
  For any $i,j \in I$, the partial identity $p_j s_i \colon X
  \rightarrow X$ is defined at $x$ just when $i_x$ and $j_x$ are
  defined and equal---that is, just when $x \in U_{ij}$. So $p \colon
  A \rightarrow X$ induces the local atlas $U$, as desired.
\end{Ex}

\begin{Ex}
  \label{ex:8}
  In the situation of the previous example, it is easy to see that the
  glueing of a local atlas on a discrete space is again discrete.
  Therefore $\cat{Set}_p$, identified with the full sub-restriction
  category of $\cat{Top}_p$ on the discrete spaces, also admits local
  glueings. Similarly, the glueing of a local atlas on an
  \emph{Alexandroff} space is again Alexandroff, so that
  $\cat{Pos}_p$, identified with the full subcategory of $\cat{Top}_p$
  on the Alexandroff spaces, has all local glueings.

  In a similar manner, the glueing of a local atlas on a topological
  manifold is again a topological manifold: indeed, if $X$ is locally
  homeomorphic to a Euclidean space, and $p \colon A \rightarrow X$ is
  a local homeomorphism, then $A$ is again locally homeomorphic to
  Euclidean space. So $\cat{Man}_p$ has local glueings. The property
  of smoothness is easily seen to be preserved in this situation so
  that $\cat{Smooth}_p$ also has local glueings.
\end{Ex}
\begin{Ex}
  \label{ex:7}
  The join restriction category $\cat{Loc}_p$ has local glueings. In
  this case, for any $X \in \cat{Loc}_p$ we can identify $\O(X)$ with
  $X$ itself, so that an $I$-object atlas comprises a family of
  elements $(\varphi_{ij} \in X)$ such that
  $\varphi_{ij} = \varphi_{ji}$ and
  $\varphi_{jk} \wedge \varphi_{ij} \leqslant \varphi_{ik}$.
  The corresponding glueing is the locale 
  \begin{equation*}
    A = \bigl\{(\theta_i \in X : i \in I) \mid \theta_i \leqslant
    \varphi_{ii} \text{ and } \theta_{j} \wedge \varphi_{ij} \leqslant
    \theta_i \text{ for all $i,j \in I$}\bigr\}
  \end{equation*}
  under the pointwise ordering. Joins and binary meets are computed
  pointwise, while the top element is $(\varphi_{ii} : i \in I)$. The
  total projection map $p \colon A \rightarrow X$ is defined by
  $p^\ast(\psi) = \bigl(\psi \wedge \varphi_{ii} : i \in I\bigr)$, and
  this is the join of the partial isomorphisms
  $p_i \colon A \rightarrow X$ defined by
  $p_i^\ast(\psi) = (\psi \wedge \varphi_{ij} : j \in I)$, with
  corresponding partial inverses $s_i \colon X \rightarrow A$ given by
  $s_i^\ast(\theta) = \theta_i$. Since $s_j p_i \colon X
  \rightarrow X$ is the partial identity $\psi \mapsto \psi \wedge
  \varphi_{ij}$, we see that $p \colon A \rightarrow X$ induces the
  local atlas $\varphi$, as desired.
\end{Ex}
\begin{Ex}
  \label{ex:10}
  The join restriction category $\cat{\L\R\T op}_p$ has local
  glueings. In this case, an atlas on $(X, \O_X)$ is simply an atlas
  on $X$, for which we have the associated glueing
  $p \colon A \rightarrow X$ in $\cat{Top}_p$; the corresponding
  glueing in $\cat{\L\R\T op}_p$ can be taken to be $(p, \mathrm{id})
  \colon (A, p^\ast \O_X) \rightarrow (X, \O_X)$. Just as in the case
  of manifolds, the glueing of any local atlas on a scheme will itself
  be a scheme, so that $\cat{Sch}_p$ also has all local glueings.
\end{Ex}
\begin{Ex}
  \label{ex:9}
  In a join restriction category of the form $\cat{Par}(\C, \M)$, an
  $I$-object local atlas on the object $X$ can be identified with a
  family of $\M$-subobjects
  $(\varphi_{ij} \colon U_{ij} \rightarrowtail X)$ in $\C$ subject to
  the usual axioms. In this case, a local glueing can be obtained as
  the colimit $A$ of a diagram of the form
  \begin{equation*}
    \cd{
      & U_{ij} \ar@{ >->}[dl]_-{} \ar@{ >->}[dr]^-{} & U_{ik} \ar@{ >->}[dll]_-{}
      \ar@{ >->}[drr]^-{} & \dots & U_{jk} \ar@{ >->}[dll]_-{} \ar@{ >->}[d]_-{} \\
      U_{ii} & & U_{jj} & \dots & U_{kk}\rlap{ ,}
    }
  \end{equation*}
  whenever this colimit exists. The local homeomorphism $p \colon A
  \rightarrow X$ is induced by the cocone of maps
  $\varphi_{ij} \colon U_{ij} \rightarrowtail X$ under this diagram.
\end{Ex}
\begin{Ex}
  \label{ex:14}
  If $\C$ is a join restriction category with local glueings, then so
  too is $\cat{Bun}(\C)$, with glueings computed componentwise.
  Indeed, a restriction idempotent on $x \colon X' \rightarrow X$ in
  $\cat{Bun}(\C)$ is a map of the form $(\overline{ex}, e)$ for some
  $e \in \O(X)$; it follows that an $I$-object local atlas on $x$
  comprises $I$-object atlases $\varphi$ on $X$ and $\varphi'$ on $X'$
  related by $\varphi'_{ij} = \overline{\varphi_{ij}x}$. If
  $p' \colon A' \rightarrow X'$ and $p\colon A \rightarrow X$ are
  glueings for these two atlases, then by applying the argument of
  Section~\ref{sec:pullb-comp-local} below, we obtain a pullback
  square of the required form~\eqref{eq:58}.
\end{Ex}
As these examples show, many join restriction categories of practical
interest already admit local glueings. We now show that, if we do not
have local glueings, then we can always adjoin them in a
straightforward manner. This is a special case of Grandis'
\emph{manifold completion}~\cite{Grandis1990Cohesive}, and so we only
sketch the details.

\begin{Prop}
  \label{prop:15}
  The inclusion of the full sub-$2$-category
  $\mathrm{jr}\cat{Cat}_{lg} \rightarrow \mathrm{jr}\cat{Cat}$ of join
  restriction categories with local glueings has a left biadjoint
  $\cat{Gl}$; each unit component
  $\C \rightarrow \cat{Gl}(\C)$ of the biadjunction is
  injective on objects and 
  fully faithful, and is an equivalence whenever $\C$ has local
  glueings.
\end{Prop}
\begin{proof}[Proof (sketch)]
  $\cat{Gl}(\C)$ has as objects, pairs
  $(X,\varphi)$ where $X \in \C$ and $\varphi$ is a
   local atlas on $X$, and as maps 
  $(X, \varphi) \rightarrow (Y, \psi)$, 
  families of
  maps $(f_{ik} \colon X \rightarrow Y)$ in $\C$
  with
  \begin{equation*}
    \psi_{k\ell}f_{ik} = f_{i\ell} \overline{f_{ik}} \quad
    \text{and} \quad
    f_{ik}\varphi_{ii} = f_{ik} \quad
    \text{and} \quad f_{jk}\varphi_{ij} \leqslant
    f_{ik}\rlap{ .}
  \end{equation*}
  The identity on $(X, \varphi)$ has
  components
  $\varphi_{ij} \colon X \rightarrow X$;
the composite of maps
$f \colon (X,\varphi) \rightarrow (Y, \psi)$ and
  $g \colon (Y, \psi) \rightarrow (Z, \rho)$ has 
  \begin{equation*}
  \textstyle  (gf)_{im} \defeq \bigl(\bigvee_k g_{km}f_{ik}\bigr)_{i \in I, m \in M}\rlap{ ;}
  \end{equation*}
  while the restriction of
  $f \colon (X, \varphi) \rightarrow (Y, \psi)$ is 
  $(\bar f)_{ij} = \varphi_{ij} \bigl(\bigvee_k
  \overline{f_{ik}}\bigr)$.
  
  With extensive verification, we can check that these data are
  well-defined, and make $\cat{Gl}(\C)$ into a join restriction
  category. 
  To show it has local glueings, consider a $K$-object
  atlas $(\psi_{k\ell})$ on $(X, \varphi)$; this is given by a family of
  restriction idempotents $(\psi_{k\ell})_{ij}$ on $X$ satisfying
  various conditions. From these data, we obtain a new $I \times K$-object
  atlas on $X$ with
  $\gamma_{(i,k),(j, \ell)} = (\psi_{k\ell})_{ij}$, and obtain for each
  $\ell \in K$ a partial isomorphism
  $p_\ell \colon (X, \gamma) \rightarrow (X, \varphi)$ in $\cat{Gl}(\C)$
  with components $(p_\ell)_{(i,k),j} = (\psi_{k\ell})_{ij}$. Taking
  joins yields a local homeomorphism
  $p = \bigvee_\ell {p_\ell} \colon (X, \gamma) \rightarrow (X,
  \varphi)$ with associated local atlas $\psi$, as desired.

  This shows $\cat{Gl}(\C) \in \mathrm{jr}\cat{Cat}_{lg}$.
  There is  
  %
  a fully faithful join restriction functor
  $\eta_\C \colon \C \rightarrow \cat{Gl}(\C)$ with
  $\eta_\C(X) = (X, \{1_X\})$, which we claim exhibits
  $\mathrm{Gl}(\C)$ as a bireflection of $\C$ into
  $\mathrm{jr}\cat{Cat}_{lg}$. This means showing that, for each join
  restriction category $\D$ with local glueings, the functor
  \begin{equation}\label{eq:23}
    (\thg) \circ \eta_\C \colon \mathrm{jr}\cat{Cat}(\cat{Gl}(\C), \D) \rightarrow
    \mathrm{jr}\cat{Cat}(\C, \D)
  \end{equation}
  is an equivalence of categories. A suitable pseudoinverse functor
  takes a join restriction functor $F \colon \C \rightarrow \D$ to the
  join restriction functor
  $\tilde F \colon \cat{Gl}(\C) \rightarrow \D$ whose value at
  $(X, \varphi)$ is some chosen glueing of the atlas $F\varphi$ in
  $\D$.
  %
  Finally, if $\C$ itself has glueings, then the extension
  $\widetilde {\mathrm{id}_\C} \colon \cat{Gl}(\C) \rightarrow \C$
  provides the desired pseudo-inverse to the functor
  $\eta_\C \colon \C \rightarrow \cat{Gl}(\C)$.
\end{proof}

\subsection{Partial sections}
\label{sec:partial-sections}

Classically, a local homeomorphism $A \rightarrow X$ can be analysed
in terms of its ``partial sections'' $U \rightarrow A$ defined on some
open subset $U \subseteq X$. In this section, we reconstruct this
analysis in the join restriction context.

\begin{Defn}
  \label{def:2}
  A \emph{partial section} of a map $p \colon A \rightarrow X$ in a
  restriction category is a map $s \colon X \rightarrow A$ such that
  $ps = \overline{s}$.
\end{Defn}

Clearly, any partial section satisfies $ps \leqslant 1$; conversely,
if $p \colon A \rightarrow X$ is a total map, then any $s$ with $ps
\leqslant 1$ will satisfy $ps = \overline{ps} = \overline{s}$ by
totality of $p$.

\begin{Lemma}
  \label{lem:25}
  Let $p = \bigvee_i p_i \colon A \rightarrow X$ be a local
  homeomorphism, and let each $p_i$ have partial inverse $s_i$.
  \begin{enumerate}[(i)]
  \item Each $s_i \colon X \rightarrow A$ is a partial section of $p$.
  \item If $s$ is a partial section of $p$, then
    $\sheq s {s_i} = p_is$; in particular,
    $\sheq {s_i} {s_j} = p_j s_i$.
  \item Any partial section $s$ of $p$ can be
    written as a join
    \begin{equation}\label{eq:24}
    \textstyle  s = \bigvee_{i}s_i \sheq s {s_i}\rlap{ .}
    \end{equation}
  \item Each partial section $s$ of $p$ is a partial isomorphism with
    $s^\ast \leqslant p$.
  \end{enumerate}
\end{Lemma}
\begin{proof}
  For (i), we have
  $ps_i = \bigvee_j p_j s_i = p_i s_i \vee \bigvee_{j \neq i} p_j s_i
  = p_i s_i \vee \bigvee_{j \neq i} p_j s_i \varphi_{ii} = p_i s_i =
  \overline{s_i}$. For (ii), note that
  $p_is \leqslant ps = \overline{s}$ is a restriction idempotent; and
  as $\overline{s_i} p_is = p_i s_i p_i s = p_i s$, we have also
  $p_i s \leqslant \overline{s_i}$. Since
  $s_i p_i s = \overline{p_i}s = s\overline{p_i s} = s p_i s$, we thus
  have $p_i s \leqslant \sheq s {s_i}$. On the other hand, by
  Lemma~\ref{lem:30}(iv), we have
  $p_i s\sheq s {s_i} = p_i s_i \sheq s {s_i} = \overline{s_i}\sheq s
  {s_i} = \sheq s {s_i}$ so that $\sheq s {s_i} \leqslant p_i s$.
  
  For (iii) we have $\bigvee_i s_i \sheq s {s_i} \leqslant s$ again by
  Lemma~\ref{lem:30}(iv); but since by (ii)
  $\overline{\bigvee_i s_i \sheq s {s_i}} = \bigvee_i p_i s = ps =
  \bar s$ we have $s = \bigvee_i s_i \sheq s {s_i}$.
  Finally, for (iv), note that the compatible family of partial
  isomorphisms $\{s_i \sheq s {s_i}\}$ in~\eqref{eq:24} has family of
  partial inverses $\{\sheq s {s_i} p_i\}$---and this is another
  compatible family since $\{p_i\}$ is so. So $\{s_i \sheq s {s_i}\}$
  is bicompatible, whence by Lemma~\ref{lem:22} its join $s$ is a
  partial isomorphism with $s^\ast = \bigvee_i \sheq s {s_i} p_i
  \leqslant p$.
\end{proof}

The following results explore the property in part (iii) of this
lemma further.
\begin{Lemma}
  \label{lem:27}
  Let $p \colon A \rightarrow X$ be a local homeomorphism, and let
  $(s_i \colon X \rightarrow A)_{i \in I}$ be a family of partial
  sections of $p$. The following are equivalent:
  \begin{enumerate}[(i)]
  \item $p = \bigvee_i s_i^\ast$;
  \item Any partial section $s$ of $p$ has the form $s =
    \bigvee_i s_i \sheq s {s_i}$;
  \item Any partial section $s$ of $p$ has the form $s = \bigvee_i s_i \theta_i$ for suitable
    $\theta_i \in \O(A)$;
  \item The $s_i$'s are jointly epimorphic.
  \end{enumerate}
\end{Lemma}
\begin{proof}
  We have (i) $\Rightarrow$ (ii) by Lemma~\ref{lem:25}(iii) and
  clearly (ii) $\Rightarrow$ (iii). For (iii)~$\Rightarrow$~(i), note
  each $s_i^\ast \leqslant p$ by Lemma~\ref{lem:25}(iv), whence
  $\bigvee_i s_i^\ast$ exists and is $\leqslant p$; it remains to
  show $p \leqslant \bigvee_i s_i^\ast$. Write $p = \bigvee_j p_j$ as a join of
  partial isomorphisms. Using (ii), we can write each $p_j^\ast$ as a
  join $\bigvee_i s_i \theta_{ij}$ for suitable
  $\theta_{ij} \in \O(X)$; taking partial inverses, we have
  $p_j = \bigvee_i \theta_{ij} s_i^\ast$ and so
  $p = \bigvee_{i,j} \theta_{ij}s_i^\ast \leqslant \bigvee_i
  s_i^\ast$, as desired.

  Now (i) $\Rightarrow$ (iv) by Lemma~\ref{lem:10}, and it remains
  to show (iv) $\Rightarrow$ (i). For all $i,j \in I$ we have
  $s_j^\ast s_i \leqslant ps_i = \overline{s_i}$ and 
  $s_i^\ast s_i = \overline{s_i}$, and so 
  $(\bigvee_i s_i^\ast) s_j = \overline{s_j}$ for each $j$; since also
  $ps_j = \overline{s_j}$, we must have $\bigvee_i s_i^\ast = p$ since
  the $s_j$'s are jointly epimorphic.
\end{proof}

In the sequel, we will call any family of partial sections $\{s_i\}$
as in the preceding Lemma a \emph{basis} for the local homeomorphism
$p$. The following result gives a precise formulation of what we mean
by ``suitable'' $\theta_i$'s in (iii) above; in the
terminology of~\cite{Fourman1979Sheaves}, these families would be
called \emph{singletons}.

\begin{Lemma}
  \label{lem:1}
  Let $\{s_i\}$ be a basis for the local homeomorphism $p \colon A
  \rightarrow X$. The assignation $s \mapsto (\dbr{s = s_i} : i \in I)$ sets up a
  bijection between partial sections of $p$, and families of
  restriction idempotents $(\theta_i \in \O(X) : i \in I)$ satisfying
    \begin{equation}\label{eq:40}
      \theta_j \theta_i \leqslant \sheq {s_i} {s_j} \quad \text{and}
      \quad
      \theta_j \sheq {s_i} {s_j} \leqslant \theta_i \qquad \text{for
        all }i,j \in I\rlap{ .}
    \end{equation}
\end{Lemma}
\begin{proof}
  If $s$ is a partial section of $p$, then by Lemma~\ref{lem:30} the
  family $\theta_i = \sheq s {s_i}$ satisfies~\eqref{eq:40}, and
  the assignation $s \mapsto (\sheq s {s_i} : i \in I)$ is injective
  by Lemma~\ref{lem:25}(iii). It remains to show surjectivity. Given a
  family $(\theta_i)$ satisfying~\eqref{eq:40}, the first condition
  implies that $\theta_i \leqslant \overline{s_i}$ and so that
  \begin{equation*}
    s_j \theta_j \overline{s_i \theta_i} = s_j \theta_j \theta_i
    \leqslant s_j \sheq {s_i} {s_j} \leqslant s_i\rlap{ .}
  \end{equation*}
  
  Thus the family $\{s_i \theta_i\}$ is compatible. Let $s$ be its
  join; by Lemma~\ref{lem:25}(ii) and the second
  condition in~\eqref{eq:40} we have
  $\sheq s {s_i} = p_i s = \textstyle \bigvee_j p_i s_j \theta_j =
    \bigvee_j
    \sheq {s_j} {s_i} \theta_j = \theta_i$, 
  so that $(\theta_i : i \in I)$ is in the image of $s \mapsto (\dbr{s
    = s_i} : i \in I)$ as required.
\end{proof}

\subsection{Composition and pullback}
\label{sec:pullb-comp-local}
We will later need some understanding of how to compose and pull back
local homeomorphisms in a join restriction category.  The astute reader will notice that here 
we should be talking about {\em restriction\/} pullbacks and limits, as described in  \cite{Cockett2007Restriction}.  
However, restriction limits behave on total maps exactly as ordinary limits behave and almost all the limits we 
will actually need use total maps, so we have elided this elaboration in order to simplify the exposition.

We first consider these operations on  partial sections.

\begin{Lemma}
  \label{lem:8}
  Let $\C$ be a restriction category.
  \begin{enumerate}[(i),itemsep=0.25\baselineskip]
  \item If $s \colon X \rightarrow A$ and $t \colon A \rightarrow A'$
    are partial sections of the total maps $p \colon A \rightarrow X$ and
    $q \colon A' \rightarrow A$ in $\C$, then
    $ts \colon X \rightarrow A'$ is a partial section of
    $pq \colon A' \rightarrow X$.
  \item Given a commutative triangle of total maps
      \begin{equation*}
    \cd[@-1em]{
      {A} \ar[rr]^-{f} \ar[dr]_-{p} & &
      {B} \ar[dl]^-{q} \\ &
      {X}
    }
  \end{equation*}
  and a partial section $s$ of $p$, the composite $fs$ is a partial
  section of $q$.
\item Given a pullback square in $\C$ of total maps, as in the solid
  part of:
  \begin{equation*}
    \cd[@R+0.3em]{
      {A \times_X Y} \ar[r]^-{\pi_1} \ar[d]^-{\pi_2}
      \ar@{<--}@/_0.7em/[d]_{f^\ast s} \pullbackcorner &
      {A} \ar[d]^{p} \ar@{<--}@/_0.7em/[d]_{s} \\
      {Y} \ar[r]_-{f} &
      {X}
    }
  \end{equation*}
  and a partial section $s$ of $p$, the map $f^\ast s =
  (sf, \overline{sf}) \colon Y \rightarrow A \times_X Y$ is
  is the unique partial section of $\pi_2$ making the upwards-pointing
  square commute.
\end{enumerate}
\end{Lemma}
\begin{proof}
  For (i), we have
  $pqts = p \overline{t} s = ps \overline{ts} = \overline{s}\,
  \overline{ts} = \overline{ts}$ as desired. For (ii), we have
  $qfs = ps = \overline{s} = \overline{fs}$ since $f$ is total. For
  (iii), if $f^\ast s$ is to be a partial section of $\pi_2$ making
  the upwards square commute, we must have
  \begin{equation*}
    \pi_1 \circ f^\ast s = sf \qquad \text{and} \qquad \pi_2 \circ f^\ast s =
    \overline{f^\ast s} = \overline{\pi_1
      \circ f^\ast s} = \overline{sf}
  \end{equation*}
  using that $\pi_1$ is total. So the given definition of $f^\ast$ is
  forced; note it is well-defined since
  $psf = \overline{s}f = f\overline{sf}$, and is a partial section of
  $\pi_2$  since 
  $\pi_2 \circ f^\ast s = \overline{sf} \leqslant 1_Y$.
\end{proof}

We now consider composition and pullback of local homeomorphisms.

\begin{Lemma}
  \label{lem:26}
  \begin{enumerate}[(i),itemsep=0.25\baselineskip]
  \item If $p \colon A \rightarrow X$ and $q \colon A' \rightarrow A$
    are local homeomorphisms, then so is $pq$.
    If $\{s_i\}_{i \in I}$ and $\{t_k\}_{k \in K}$ are bases of
    partial sections for $p$ and $q$, with corresponding local atlases
    $\varphi$ and $\psi$, then $\{t_k s_i\}_{(i,k) \in I \times K}$ is
    a basis of partial sections for $pq$ with corresponding local
    atlas
    \begin{equation}\label{eq:30}
      \theta_{(i,k), (j,\ell)} = \varphi_{ij} \circ
      \overline{\psi_{k\ell}s_i}\rlap{ .}
    \end{equation}
  \item In a join restriction category with local glueings, the
  pullback of a local homeomorphism
  $p \colon A \rightarrow X$ along any total 
  $f \colon X' \rightarrow X$ exists, and is a local homeomorphism
  $q \colon A' \rightarrow X'$. If $p$ has a basis of partial sections
  $\{s_i\}$ inducing the local atlas $\varphi$, then
  $q$ has basis $\{f^\ast(s_i)\}$ inducing the
  local atlas
  \begin{equation}\label{eq:29}
    \psi_{ij} = \overline{\varphi_{ij}f}\rlap{ .}
  \end{equation}
  \end{enumerate}
\end{Lemma}

\begin{proof}
  For (i), write $p_i = s_i^\ast$ and $q_i = t_i^\ast$ so that
  $p = \bigvee_i p_i$ and $q = \bigvee_i q_i$. Now by distributivity
  of joins over composition $pq = \bigvee_{i,k} p_iq_k$, where each
  $p_i q_k$ is a partial isomorphism by Lemma~\ref{lem:9}; now the
  corresponding partial inverses $t_k s_i$ are the desired basis for
  $pq$. Finally, we calculate the induced local atlas to be
  \begin{equation*}
 \theta_{(i,k), (j, \ell)} = p_j q_\ell t_k s_i = p_j \psi_{k\ell} s_i
 = p_j s_i \overline{\psi_{k\ell}s_i} = \varphi_{ij} \circ \overline{\psi_{k\ell}s_i}\rlap{ .}
 \end{equation*}
%

 We now turn to (ii). Like before, write $p_i = s_i^\ast$ so that
 $p = \bigvee_i p_i$. Now the idempotents in~\eqref{eq:29} clearly
 satisfy $\psi_{ij} = \psi_{ji}$, while
 $\psi_{jk}\psi_{ij} =
 \overline{\varphi_{jk}f}\,\overline{\varphi_{ij}f} =
 \overline{\varphi_{jk}f\overline{\varphi_{ij}f}} =
 \overline{\varphi_{jk}\varphi_{ij}f} \leqslant
 \overline{\varphi_{ik}f} = \psi_{ik}$, and so we have a local atlas
 on $X'$; we note for future use that
 $\varphi_{ij}f = f \overline{\varphi_{ij}f} = f \psi_{ij}$. Let
 $q = \bigvee_i q_i \colon A' \rightarrow X'$ be a glueing for this
 local atlas, with corresponding basis of partial sections $\{t_i\}$.
 Now observe that the family of maps $s_i f \colon X' \rightarrow A$
 satisfy
 \begin{equation*}
 s_i f \psi_{ij} = s_i
 \varphi_{ij} f \leqslant s_j f \qquad \text{and} \qquad \overline{s_i f}
 = \overline{\overline{s_i} f} = \overline{\varphi_{ii}f} =
 \psi_{ii}\rlap{ ,}
 \end{equation*}
 so that by Lemma~\ref{lem:10}, there is a unique total map
 $g \colon A' \rightarrow A$ with $gt_i = s_if$ for each $i$.
 Moreover, since
 $p_jgt_i = p_js_if = \varphi_{ij}f = f \psi_{ij} = fq_jt_i$ for each
 $i,j$, we conclude by Lemma~\ref{lem:10} that $p_jg = fq_j$ for each
 $j$. Taking joins over $j$, we have that $pg = fq$, yielding the
 commuting square bottom right in:
  %
 \begin{equation*}
 \cd[@-0.5em]{
 {Y} \ar@{-->}[dr]^-{h} \ar@/^1em/[drr]^-{u} \ar@/_1em/[ddr]_{v} \\
 & A' \ar[r]^-{g} \ar[d]_-{q} &
 {A} \ar[d]^{p} \\ &
 {X'} \ar[r]_-{f} &
 {X}\rlap{ .}
 }
 \end{equation*}
 We claim this square is a pullback in $\C$. Indeed, given maps $u,v$
 with $fv = pu$  as in the outside above, we may
 consider the maps $t_i v \overline{p_i u} \colon Y \rightarrow A'$.
 These are compatible as a consequence of the calculation
 \begin{align*}
 t_j v\, \overline{p_j u}\, \overline{t_i v \overline{p_i u}} &
 \leqslant t_j v\, \overline{p_j u}\,\overline{p_i u} =
 t_j v\, \overline{p_j u \overline{p_i u}} = 
 t_j v\, \overline{p_j \overline{p_i} u} = 
 t_j v\, \overline{p_j s_i p_i u} \\ &=
 t_j v\, \overline{\varphi_{ij} p_i u} \leqslant
 t_j v\, \overline{\varphi_{ij} p u} =
 t_j v\, \overline{\varphi_{ij} f v} =
 t_j v\, \overline{\psi_{ij} v} =
 t_j \psi_{ij} v \leqslant t_i v\rlap{ ,}
 \end{align*}
 and so we can form the join $h = \bigvee_i t_i v \overline{p_i u}
 \colon Y \rightarrow A'$. Noting that
 \begin{equation*}
 gt_i v \overline{p_iu} = s_i fv \overline{p_i u} = s_i p u
 \overline{p_i u} = s_i p_i u = \overline{p_i} u
 \end{equation*}
 we obtain on taking joins that $gh = \textstyle\bigvee_i
 gt_i v \overline{p_i u} = \bigvee_i \overline{p_i}u = u$. 
Moreover,
 we have
 \begin{equation*}
 qt_i v \overline{p_i u} = 
 \overline{t_i} v \overline{p_i u} =
 v \overline{t_i v} \,\overline{p_i u}
 = v \overline{gt_i v} \,\overline{p_i u}
 = v \overline{p_i u}
\end{equation*}
and so, on taking joins, $qh = \bigvee_i qt_i v \overline{p_i u} = v
\bigvee_i \overline{p_i u} = v \overline{pu} = v\overline{fv} =
\overline{f}v = v$.
So $h$ makes both triangles above commute; it remains to show that it
is unique such. But if $k$ is another such map, then we have, since
$fq = gp$ is total that
$\overline{q_i} = \overline{fq}\, \overline{q_i} =
 \overline{fq\overline{q_i}} = \overline{fq_i}$
and so
\begin{equation*}
 q_i k = qk \overline{q_ik} = v \overline{q_i k} = 
 v\overline{fq_ik} = v \overline{p_i gk} = v \overline{p_i u}\rlap{ .}
\end{equation*}
It follows that
$k = \bigvee_{i}\overline{q_i} k = \bigvee_i t_i q_i k = \bigvee_i t_i
v \overline{p_i u} = h$ as desired. Finally, note that since for each
$i$ we have $gt_i = s_i f$, we must have $t_i = f^\ast(s_i)$ by
unicity in Lemma~\ref{lem:8}(iii).
\end{proof}

\section{Local homeomorphisms and sheaves}
\label{sec:local-home-sheav}

In this section, as a warm-up for our main result, we explain how the
classical correspondence between local homeomorphisms over a space $X$
and sheaves on $X$ generalises to objects $X$ of a suitable join
restriction category~$\C$. Much as in the classical case, the
correspondence sought will arise as part of a larger adjunction.

\subsection{The categories of interest}
\label{sec:categories-interest}

One side of our correspondence will involve local homeomomorphisms
over an object $X \in \C$, or more generally, total maps over $X$.
\begin{Defn}
  \label{def:1}
  Let $\C$ be a join restriction category and $X \in \C$. We write
  $\tot$ for the category with total maps $A \rightarrow X$ as objects,
  and commuting triangles
  \begin{equation}\label{eq:26}
    \cd[@-1em]{
      {A} \ar[rr]^-{f} \ar[dr]_-{p} & &
      {B} \ar[dl]^-{q} \\ &
      {X}
    }
  \end{equation}
  as morphisms. We write $\lh$ for the full subcategory of
  $\tot$ whose objects are the local homeomorphisms.
\end{Defn}

It is perhaps worth noting:
\begin{Lemma}
  \label{lem:34}
  If~\eqref{eq:26} is a map in $\tot$ then $f$, as well as
  $p$ and $q$, is total; if it is a map in $\lh$ then $f$,
  as well as $p$ and $q$, is a local homeomorphism.
\end{Lemma}
\begin{proof}
  If $p$ and $q$ are total, then
  $\overline{f} = \overline{qf} = \overline{p} = 1$ so that $f$ is
  also total. Now let $p$ and $q$ be local homeomorphisms, and write
  $p = \bigvee_i p_i$ with corresponding partial sections $s_i$. By
  Lemma~\ref{lem:8}(ii), each $fs_i$ is a partial section of $q$,
  whence a partial isomorphism by Lemma~\ref{lem:25}(iv); and so each
  $fs_ip_i$ is a partial isomorphism. Now
  $f = f\overline{p} = \bigvee_i f\overline{p_i} = \bigvee_i fs_i p_i$
  so that $f$ is a local homeomorphism as desired.
\end{proof}

The other side of the correspondence involves \emph{presheaves} and
\emph{sheaves} on $X \in \C$; these are presheaves and sheaves in the
classical sense on the locale of restriction idempotents $\O(X)$. Our
preferred formulation will follow Fourman and
Scott~\cite{Fourman1979Sheaves}.

\begin{Defn}
  \label{def:26}
  A \emph{presheaf} on an object $X$ of a join restriction category is
  a set $A$ with an associative, unital right action
  $A \times \O(X) \rightarrow A$, written
  $a, e \mapsto ae$, and an \emph{extent} operation
  $A \rightarrow \O(X)$, written $a \mapsto \bar a$, obeying the axioms:
  \begin{enumerate}[(i)]
  \item $a \bar a = a$ for all $a \in A$;
  \item $\overline{a e} = \overline{a} e$ for all $a \in
    A$ and $e \in \O(A)$.
  \end{enumerate}
  Just as in a restriction category, we can define a partial order
  $\leqslant$ and compatibility relation $\smile$ on a presheaf $A$ by
  \begin{equation*}
    a \leqslant b \text{ iff } a = b\overline{a} \qquad \text{and}
    \qquad
    a \smile b \text{ iff } a\overline{b} = b \overline{a}\rlap{ ;}
  \end{equation*}
  and we call $A$ a \emph{sheaf} if every compatible family of elements
  has a join with respect to $\leqslant$. We write $\pshv$ for
  the category whose objects are presheaves, and whose maps $A
  \rightarrow B$ are functions which preserve the right action
  and the extent, and write $\shv$ for the full subcategory of sheaves.
\end{Defn}
\begin{Ex}
  \label{ex:4}
  If $\C$ is a join restriction category then each hom-set $\C(X,Y)$
  is an $\O(X)$-sheaf, with the right $\O(X)$-action given by
  composition and the extent operation given by restriction. More
  generally, if $F \colon \B \rightarrow \C$ is a hyperconnected join
  restriction functor, then each hom-set $\B(X,Y)$ becomes an
  $\O(FX)$-sheaf by transporting the $\O(X)$-sheaf structure along the
  isomorphism $\O(X) \rightarrow \O(FX)$.
\end{Ex}
By analogy with Lemma~\ref{lem:18}, we have the following standard
result~\cite{Fourman1979Sheaves}:
\begin{Lemma}
  \label{lem:6}
  If $X$ is an object of a join restriction category and $A \in
  \shv$, then:
  \begin{enumerate}[(i)]
  \item $\overline{\bigvee_{i} a_i} = \bigvee_{i}\overline{a_i}$ for
    all compatible families $\{a_i\} \subseteq A$;
  \item $(\bigvee_i a_i)e = \bigvee_i (a_ie)$ for all
    compatible families $\{a_i\} \subseteq A$ and all $e \in \O(X)$;
  \item $a(\bigvee_i e_i) = \bigvee_i (ae_i)$ for all $a \in
    A$ and all
    families $\{e_i\} \subseteq \O(X)$.
  \end{enumerate}
  Moreover, if $f \colon A \rightarrow B$ in $\shv$, then
    $f(\bigvee_i a_i) = \bigvee_i f(a_i)$ for all compatible families
    $\{a_i\} \subseteq A$.
\end{Lemma}
\begin{proof}
  (i) is as in Lemma~\ref{lem:18}. For (ii), we easily have
  $\bigvee_i (a_i e) \leqslant (\bigvee_i a_i)e$, but
  using (i) and~\eqref{eq:27} in $\C$, we also have
  $\overline{\bigvee_i (a_i e)} = \bigvee_i \overline{a_i
    e} = \bigvee_i (\overline{a_i}e) = (\bigvee_i
  \overline{a_i})e = (\overline{\bigvee_i a_i})e =
  \overline{(\bigvee_i a_i) e}$; whence equality. (iii) follows
  like before from (i) and (ii)
  using~\cite[Lemma~3.1.8(iii)]{Guo2012Products}. For the final claim,
  we easily have $\bigvee_i f(a_i) \leqslant f(\bigvee_i a_i)$, but
  also that
  $\overline{\bigvee_i f(a_i)}= \bigvee_i \overline{f(a_i)} =
  \bigvee_i \overline{a_i} = \overline{\bigvee_i a_i} =
  \overline{f(\bigvee_i a_i)}$ using (i) and the fact that $f$
  preserves restriction. Thus $f(\bigvee_i a_i) = \bigvee_i f(a_i)$ as desired.
\end{proof}

The $\O(X)$-valued equality of Definition~\ref{def:34}
also makes sense for presheaves:

\begin{Defn}
  \label{def:10}(\cite{Fourman1979Sheaves})
  For any $X \in \C$,  any $A \in \pshv$ and any $a,b \in A$,
  we define
  \begin{equation*}
    \sheq a b \defeq \bigvee \{e \in \O(X) : e \leqslant
    \overline{a}\overline{b} \text{ and } ae = be\}\rlap{ .}
  \end{equation*}
\end{Defn}
And we have the following standard result, which can be proved exactly
as in Lemma~\ref{lem:30} above, using the previous lemma for part
(iv).
\begin{Lemma}
\label{lem:14}(\cite{Fourman1979Sheaves})
  For any $X \in \C$, any $A \in \pshv$ and any $a,b,c \in A$
  we have
  \begin{enumerate}[(i)]
  \item If $a \leqslant b$ then $\sheq a b = \bar b$;
  \item $\sheq a b = \sheq b a$;
  \item $\sheq b c \sheq a b \leqslant \sheq a c$;
  \item If $A$ is a sheaf, then $a\sheq a b = b \sheq a b$.
  \end{enumerate}
\end{Lemma}

\subsection{The adjunction}
\label{sec:adjunction}
Our objective in this section is to construct, for any object $X$ of a
join restriction category with local glueings, an adjunction
\begin{equation}\label{eq:25}
  \cd{
    {\tot} \ar@<-4.5pt>[r]_-{\Gamma} \ar@{<-}@<4.5pt>[r]^-{\Delta} \ar@{}[r]|-{\bot} &
    {\pshv} \rlap{ .}
  }
\end{equation}
We begin with the right adjoint $\Gamma$, which is given by taking
partial sections.

\begin{Defn}
  \label{def:20}
  Let $p \colon A \rightarrow X$ be a total map in a join restriction
  category. We write $\Gamma(p)$ for the $X$-presheaf of partial
  sections of $p$, with right $\O(X)$-action $s,e \mapsto s \circ e$
  and extent operation given by restriction $s \mapsto \bar s$.
\end{Defn}

\begin{Prop}
  \label{prop:17}
  Let $X$ be an object of a join restriction category $\C$. The
  assignation $p \mapsto \Gamma(p)$ underlies a functor
  $\Gamma \colon \tot \rightarrow \pshv$ whose image
  lands inside the subcategory $\shv$.
\end{Prop}
\begin{proof}
  Each $f \colon p \rightarrow q$ of $\tot$, as
  in~\eqref{eq:26}, induces the map
  $\Gamma(f) \colon \Gamma(p) \rightarrow \Gamma(q)$ given by
  $s \mapsto fs$, which is well-defined by Lemma~\ref{lem:8}(ii). It
  is clear that $\Gamma(f)$ commutes with the right $\O(X)$-action and
  the extent operations, and that $f \mapsto \Gamma(f)$ is itself
  functorial. It remains to show that the image of $\Gamma$ lands
  inside $\shv$, i.e., that each $\Gamma(p)$ is a sheaf. The
  relations $\leqslant$ and $\smile$ on $\Gamma(p)$ are the same as
  those on $\C(X,A)$, so that any compatible family
  $\{s_i \colon X \rightarrow A\} \subseteq \Gamma(p)$ admits a join
  $\bigvee_i s_i \colon X \rightarrow A$ \emph{qua} family in $\C$.
  This join is again a partial section since
  $p\bigvee_i s_i = \bigvee_i ps_i = \bigvee_i \overline{s_i} =
  \overline{\bigvee_i s_i}$, and as such is the join
  in $\Gamma(p)$.
\end{proof}

We now show that, in the presence of local glueings, the
functor $\Gamma \colon \tot \rightarrow \pshv$ of
Proposition~\ref{prop:17} has a left adjoint $\Delta$, defined on
objects as follows.

\begin{Defn}
  \label{def:27}
  Let $\C$ be a join restriction category with local glueings. For any
  $X \in \C$ and $A \in \pshv$, we have by
  Lemma~\ref{lem:14}(ii) and (iii) an $A$-object local atlas on $X$
  with $\varphi_{ab} = \sheq a b$. We denote by
  $p_A \colon \Delta A \rightarrow X$ a glueing of this atlas, with
  basis of partial sections $\{s_a\}_{a \in A}$ and corresponding
  partial inverses $p_a$.
\end{Defn}

We now define the unit of the desired adjunction and verify the
adjointness property; functoriality of $\Delta$ will follow from this.
\begin{Lemma}
  \label{lem:21}
  The assignation $a \mapsto s_a$ gives a 
  presheaf map $\eta_A \colon A \rightarrow \Gamma \Delta A$.
\end{Lemma}
\begin{proof}
  $\eta_A$ preserves extents since
  $\overline{s_a} = p_a s_a = \sheq a a = \overline{a}$. To see it
  preserves the right $\O(X)$-action, note that for any $e \in \O(X)$,
  we have
  $s_a e = s_a \overline{a}e = s_a \sheq {ae} a = s_a p_a s_{ae}
  \leqslant s_{ae}$; but since
  $\overline{s_a e} = \overline{ae} = \overline{s_{ae}}$ we have
  $s_a e = s_{ae}$ as desired.
\end{proof}

\begin{Lemma}
  \label{lem:3}
  For each $A \in \pshv$, the morphism $\eta_A \colon A
  \rightarrow \Gamma\Delta A$ exhibits $p_A \colon \Delta A
  \rightarrow X$ as the value at $A$ of a left adjoint $\Delta$ to
  $\Gamma$.
\end{Lemma}
\begin{proof}
  Let $q \colon B \rightarrow X$ in $\tot$. We must show each
  presheaf map $f \colon A \rightarrow \Gamma q$ is a composite
  $\Gamma g \circ \eta_A \colon A \rightarrow \Gamma \Delta A
  \rightarrow \Gamma q$ for a unique $g$ in a commuting triangle
  \begin{equation*}
    \cd[@!C@C-1em]{
      {\Delta A} \ar[rr]^-{g} \ar[dr]_-{p_A} &  &
      {B} \ar[dl]^-{q} \\ &
      {X}\rlap{ .}
    }
  \end{equation*}
  Now if $\Gamma(g) \circ \eta_A = \Gamma(g') \circ \eta_A$, then
  $gs_a = g's_a$ for each $a \in A$, whence $g = g'$ since the $s_a$'s
  are a basis. For existence of $g$, consider the partial sections
  $f_a \colon X \rightarrow B$ giving the values of the presheaf map
  $f \colon A \rightarrow \Gamma q$. Since $f$ is a presheaf map, we
  have $\overline{f_a} = \overline{a}$ and $f_a e = f_{ae}$ for any
  $e \in \O(X)$; we thus have
  $\overline{f_a} = \overline{a} = \sheq a a$ and
  \begin{align*}
    f_b \sheq a b &= \bigvee \{f_b e : e \leqslant
    \overline{a}\overline{b}, ae = be\} =
    \bigvee \{f_{b e} : e \leqslant
    \overline{a}\overline{b}, ae = be\} \\ &=
    \bigvee \{f_{a e} : e \leqslant
    \overline{a}\overline{b}, ae = be\} =
    \bigvee \{f_{a}e : e \leqslant
    \overline{a}\overline{b}, ae = be\} \leqslant
     f_a\rlap{ .}
  \end{align*}
  So by Lemma~\ref{lem:10}, there is a unique total map $g \colon \Delta
  A \rightarrow B$ with $gs_a = f_a$ for each $a \in A$; this says
  exactly that $\Gamma(g) \circ \eta_A = f$. Finally,
  since
  $qgs_a = qf_a = \overline{f_a} = \overline{a} = p_A s_a$
  for each $a$, we have since the $s_a$'s are a basis that $qg =
  p_A$. 
\end{proof}
This completes the construction of the adjunction~\eqref{eq:25}.

\subsection{The fixpoints}
\label{sec:fixpoints}
By a \emph{right fixpoint} of an adjunction
$F \dashv G \colon \B \rightarrow \A$, we mean an object $A \in \A$ at
which the unit map $\eta_A \colon A \rightarrow GFA$ is invertible;
similarly, a \emph{left fixpoint} is an object $B \in \B$ with
$\varepsilon_B$ invertible. Each adjunction restricts to an
equivalence between the full subcategories on the left fixpoints, and
the right fixpoints. In this section, we show that, for the
adjunction~\eqref{eq:25}, this yields the desired equivalence between
local homeomorphisms and sheaves.



\begin{Prop}
  \label{prop:18}
  The component $\eta_A \colon A \rightarrow \Gamma \Delta A$ of the
  unit of~\eqref{eq:25} is invertible if and only if $A$ is a sheaf.
\end{Prop}
\begin{proof}
  Each $\Gamma(p)$ is a sheaf, so if $\eta_A$ is invertible then $A
  \cong \Gamma\Delta A$ is a sheaf. Suppose conversely that $A$ is
  a sheaf. To see $\eta_A$ is injective, note that if $s_a = s_b$ then
  $a = a\sheq a a = ap_a s_a = a p_a s_b = a \sheq a b = b \sheq a b$,
  using Lemma~\ref{lem:14}(iv) at the last step. So $a \leqslant b$
  and by symmetry $b \leqslant a$, whence $a = b$.

  To show $\eta_A$ is surjective, we must show each partial section
  $s$ of $p_A \colon \Delta A \rightarrow X$ is equal to $s_a$ for
  some $a \in A$. By Lemma~\ref{lem:25}(ii), we can write
  $s = \bigvee_b s_b \sheq s {s_b}$. In particular, the family
  $\{s_{b}\sheq s {s_b}\}$ is compatible in $\Gamma \Delta A$. Because
  $\eta_A$ is an injective presheaf map, it follows that
  $\{b \sheq s {s_b}\}$ is a compatible family in $A$. Since $A$ is a
  sheaf, we can form $a = \bigvee_b b \sheq s {s_b}$, and because any
  presheaf map preserves joins by Lemma~\ref{lem:6}, we have $s_a =
  \bigvee_b s_b \sheq s {s_b} = s$
  as desired.
\end{proof}
\begin{Prop}
  \label{prop:19}
  The component
  \begin{equation*}
    \cd[@!C@C-1.5em@-0.5em]{
      {\Delta \Gamma(q)} \ar[rr]^{\varepsilon_q} \ar[dr]_-{p_{\Gamma q}} & &
      {B} \ar[dl]^-{q} \\ &
      {X}
    }
  \end{equation*}
  of the counit of~\eqref{eq:25}
  is invertible if and only if $q$ is a local homeomorphism.
\end{Prop}
\begin{proof}
  By construction each $\Delta A \rightarrow X$ is a local
  homeomorphism; so if $\varepsilon_q$ is invertible then $q$ is a
  local homeomorphism since $\Delta \Gamma(q)$ is. Conversely, let
  $q = \bigvee_i q_i$ be a local homeomorphism with corresponding
  basis of partial sections
  $\{t_i\}$. 
  Since by 
  Proposition~\ref{prop:18}
  $\eta_{\Gamma(q)} \colon \Gamma(q) \rightarrow
  \Gamma\Delta\Gamma(q)$ is invertible, the partial sections
  $\{s_{t_i}\}$ are a basis for $p_{\Gamma q} \colon \Delta \Gamma(q)
  \rightarrow X$, whence by Lemma~\ref{lem:27} we have $p_{\Gamma(q)}
  = \bigvee_i p_{t_i}$.
  It follows that $p_{\Gamma q}$ is a glueing of the $I$-object local
  atlas $p_{t_j}s_{t_i}$. But by definition of $p_{\Gamma q}$ and
  Lemma~\ref{lem:25}(ii) we have
  $p_{t_j}s_{t_i} = \sheq {t_i}{t_j} = q_jt_i$, so that $q$ is a
  glueing of this same atlas. Thus, by Lemma~\ref{lem:10}, we must have
  $\Delta \Gamma(q) \cong B$ over $X$, via the unique map
  $h \colon \Delta \Gamma(q) \rightarrow B$ with $hs_{t_i} = t_i$ for
  all $i$. Since $\varepsilon_q$ satisfies this property, we must have
  $\varepsilon_q = h$ and so $\varepsilon_q$ is invertible.
\end{proof}
We have thus proven:
\begin{Thm}
  \label{thm:1} 
  Restricting~\eqref{eq:25} to its 
  fixpoints yields an equivalence
  \begin{equation*}
    \cd{
      {\lh} \ar@<-4.5pt>[r]_-{\Gamma} \ar@{<-}@<4.5pt>[r]^-{\Delta} \ar@{}[r]|-{\sim} &
      {\shv} \rlap{ .}
    }
  \end{equation*}
\end{Thm}
Just as in the classical case, we can say slightly more. We call an
adjunction $F \dashv G \colon \B \rightarrow \A$ 
\emph{Galois} if every $FA$ is a left fixpoint,
or equivalently, every $GB$ is a right fixpoint. The full
subcategory of left fixpoints is then coreflective in $\B$ via $FG$,
while the right fixpoints are reflective in $\A$
via $GF$. For the adjunction~\eqref{eq:25}, each $\Delta A$ is a
local homeomorphism, hence a left fixpoint, and so:

\begin{Cor}
  \label{cor:7}
  The adjunction~\eqref{eq:25} is Galois. In particular, $\shv$ is
  reflective in $\mathrm{Psh}(X)$ via $\Gamma \Delta$, while $\lh$ is
  coreflective in $\tot$ via $\Delta \Gamma$.
\end{Cor}

The reflector
$\Gamma \Delta \colon \pshv \rightarrow \shv$ performs
\emph{sheafification}, and we can read off an explicit description (in
fact that given in~\cite{Fourman1979Sheaves}). If
$A \in \pshv$ then the local homeomorphism
$p_A \colon \Delta A \rightarrow X$ has basis of local sections
$(s_a : a \in A)$ and associated local atlas $\sheq a b$. So by
Lemma~\ref{lem:1}, $\Gamma \Delta A$ is the set of all families
\begin{equation*}
  (\theta_a \in \O(X) : a \in A) \quad \text{such that } \theta_b
  \theta_a \leqslant \sheq a b, \theta_b \sheq a b \leqslant \theta_a
  \text{ for all $a,b \in A$}
\end{equation*}
made into a sheaf via the pointwise $\O(X)$-action and extent
operation $\overline{\theta}= \bigvee_i \theta_i$. The unit
$\eta \colon A \rightarrow \Gamma\Delta A$ sends $b$ to
$\bigl(\sheq a b : a \in A\bigr)$.

\begin{Rk}
  \label{rk:5}
  Note that the category $\cat{Sh}(X)$ to the right of our main
  equivalence does not really depend on $X \in \C$, but only on the
  locale of restriction idempotents $\O(X)$. In particular, this
  allows us to identify $\lh$ with
  $\D \mathord{\mkern1.5mu/_{\!\ell h}\mkern1.5mu}Y$ whenever we have
  an identification of $\O(X)$ with $\O(Y)$.
  
  A particularly natural example of this is as follows. Suppose that
  $F \colon \C \rightarrow \D$ is any join restriction functor between
  join restriction categories with local glueings. For any $X \in \C$,
  the action of $F$ induces a functor
  $\lh \rightarrow \D \mathord{\mkern1.5mu/_{\!\ell h}\mkern1.5mu}FX$.
  Now if $F$ is hyperconnected, then this functor is an equivalence,
  with an explicit pseudoinverse being given as follows. Given
  $p \colon A \rightarrow FX$ a local homeomorphism in $\D$, we first
  construct the associated local atlas $\varphi$ on $FX$, then use
  hyperconnectedness to lift this to a local atlas $\psi$ on $X$, and
  finally form the glueing $q \colon B \rightarrow X$ of $\psi$. Since
  $q$ has associated atlas $\psi$, $Fq$ will have associated atlas
  $\varphi$ and so $Fq \cong p$ as desired.
\end{Rk}

\section{The main theorem}
\label{sec:main-theorem}

In this section we prove the main result of the paper, establishing an
equivalence between source-\'etale partite internal categories in
$\C$, and join restriction categories hyperconnected over $\C$, for
any join restriction category $\C$ with local glueings.

\subsection{The categories of interest}
\label{sec:categories-interest-2}
As in the previous section, the desired equivalence will arise as the
fixpoints of a Galois adjunction between larger categories; we begin
by introducing these. On one side, we have the categories which will
play a role analogous to $\pshv$ and $\shv$ in the
previous section.

\begin{Defn}
  \label{def:28}
  The category $\rcat$ has as objects, pairs of a
  restriction category $\A$ and a restriction functor
  $P \colon \A \rightarrow \C$; and has maps
  $(F, \alpha) \colon (\A, P) \rightarrow (\B, Q)$ given by a
  restriction functor $F$ and total natural transformation $\alpha$
  as in:
  \begin{equation}\label{eq:28}
    \cd{
      {\A} \ar[rr]^-{F} \ar[dr]_-{P} & \ltwocello{d}{\alpha} &
      {\B\rlap{ .}} \ar[dl]^-{Q} \\ &
      {\C}
    }
  \end{equation}
  Composition is given by
  $(G, \beta) \circ (F, \alpha) = (GF, \beta F \circ \alpha)$
  while $1_{(\A, P)} = (1_\A, 1_P)$. We write $\hcon$ for the full
  subcategory of $\rcat$ whose objects are those
  $P \colon \A \rightarrow \C$ for which $\A$ is a \emph{join}
  restriction category, and $P$ is a \emph{hyperconnected} restriction
  functor, necessarily join-preserving by Lemma~\ref{lem:19}.
\end{Defn}
The sense in which $\hcon$ is analogous to $\shv$ is found in
Example~\ref{ex:4}: for any $P \colon \A \rightarrow \C$ in $\hcon$,
each hom-set $\A(x,y)$ is an $\O(Px)$-sheaf. Under this analogy, the
category corresponding to $\pshv$ would most rightly be
$\mathrm{r}\cat{Cat} \mathord{/\!/_h} \C$; but it costs us nothing to
consider the more general $\rcat$, and so we do so.

It is reasonable to ask why we do not restrict the
\emph{maps} in $\hcon$ to be those~\eqref{eq:28} for which
$F$ is a join restriction functor. In fact, this is unneccessary.

\begin{Lemma}
  \label{lem:33}
  If~\eqref{eq:28} is a map in $\hcon$, then $F$, as well as $P$ and
  $Q$, is join-preserving.
\end{Lemma}
\begin{proof}
  By Lemma~\ref{lem:19}, it suffices to show that $F$ preserves any
  join of restriction idempotents $\bigvee_i e_i \in \O(X)$.
  Since $Q$ is hyperconnected, it suffices for this to show that
  $QF(\bigvee_i e_i) = \bigvee_i QFe_i$. Clearly
  $\bigvee_i QFe_i \leqslant QF(\bigvee_i e_i)$, so it is enough to
  show both sides have the same restriction. But
  \begin{align*}
    \textstyle \overline{QF(\bigvee_i e_i)} &=
    \textstyle \overline{\alpha_X \circ QF(\bigvee_i e_i)} =
    \overline{P(\bigvee_i e_i) \circ \alpha_X} =
    \overline{(\bigvee_i Pe_i) \circ \alpha_X} =
    \overline{\bigvee_i (Pe_i \circ \alpha_X)} \\
    &= \textstyle\overline{\bigvee_i (\alpha_X \circ QFe_i)}
    = \overline{\alpha_X \circ \bigvee_i QFe_i} = \overline{\bigvee_i
      QFe_i}
  \end{align*}
  using totality of $\alpha_X$ and the fact that $P$ preserves joins.
\end{proof}

We now turn to the categories on the other side of our adjunction,
which will play the roles that $\tot$ and $\lh$ took in
the previous section. The objects of these categories will be what we call \emph{partite} internal
categories in $\C$.

\begin{Defn}
  \label{def:29}
  Let $I$ be a set (of ``components''). An \emph{$I$-partite internal
    category} $\mathbb{A}$ in $\C$ comprises:
  \begin{itemize}[itemsep=0.2\baselineskip]
  \item \textbf{Objects of objects} $A_i$ for each $i \in I$;
  \item \textbf{Objects of arrows} $A_{ij}$ for each $i,j \in I$, together with
    a source-target span
  \begin{equation*}
    A_i \xleftarrow{\sigma_{ij}} A_{ij} \xrightarrow{\tau_{ij}} A_j
  \end{equation*}
  of total maps for which $\sigma_{ij}$ admits a pullback along any
  total map in $\C$;
\item \textbf{Identities} maps $\eta_i \colon A_i \rightarrow A_{ii}$ for each
  $i \in I$, compatible with source and target in the sense that
  $\sigma_{ii}\eta_i = 1 = \tau_{ii}\eta_i$; and
\item \textbf{Composition} maps
  $\mu_{ijk} \colon A_{jk} \times_{A_j} A_{ij} \rightarrow A_{ik}$ for
  each $i,j,k \in I$, compatible with source and target in
  the sense that $\sigma_{ik} \mu_{ijk} = \sigma_{ij}\pi_2$ and
  $\tau_{ik}\mu_{ijk} = \tau_{jk}\pi_1$,
\end{itemize}
such that the following diagrams commute for all $i,j,k,\ell \in I$:
  \begin{equation*}
    \cd{
      A_{ij} \times_{A_i} A_{ii} \ar[dr]_-{\mu_{iij}} & A_{ij} \ar@{=}[d]_-{}
      \ar[l]_-{(1, \eta_i\sigma_{ij})} \ar[r]^-{(\eta_j\tau_{ij}, 1)} & A_{jj} \times_{A_j}
      A_{ij} \ar[dl]^-{\mu_{ijj}} \\ & A_{ij}
    } \     \cd{
      {A_{k\ell} \times_{A_k} A_{jk} \times_{A_j} A_{ij}} \ar[r]^-{1
        \times_{A_k} \mu_{ijk}} \ar[d]_-{\mu_{jk\ell} \times_{A_j} 1} &
      {A_{k\ell} \times_{A_k} A_{ik}} \ar[d]^-{\mu_{ik\ell}}\\
      {A_{j\ell} \times_{A_j} A_{ij}} \ar[r]_-{\mu_{ij\ell}} & A_{i\ell}\rlap{ .}
    }
  \end{equation*}
\end{Defn}

For example, a $1$-partite internal category in $\C$ is simply an
internal category. On the other hand, if $\C$ admits a
\emph{restriction terminal object} $\mathbf{1}$---that is, an object
to which each other object admits a unique total map---then a partite
internal category $\mathbb{A}$ in $\C$ with each $A_i = \mathbf{1}$ is
the same as a category \emph{enriched} in $\C$.

We now describe the relevant morphisms between partite internal
categories. As in the introduction, these will not be the obvious
internal functors, but rather internal \emph{cofunctors}. Cofunctors
first appeared in~\cite{Higgins1993Duality} under the name
``comorphism''; the alternative nomenclature we use follows
Aguiar~\cite{Aguiar1997Internal}. The notion is most likely not a
familiar one, so we take a moment to spell it out first in the case of
ordinary categories.

\begin{Defn}
  \label{def:3}
  A cofunctor between ordinary categories $\C \rightsquigarrow \D$
  comprises a mapping on objects
  $F \colon \mathrm{ob}(\D) \rightarrow \mathrm{ob}(\C)$, together
  with an operation assigning to each $d \in \mathrm{ob}(\D)$ and each
  map $f \colon Fd \rightarrow c$ in $\mathbb{C}$, an object
  $f_\ast(d)$ and map $F(d, f) \colon d \rightarrow f_\ast(d)$ in
  $\mathbb{D}$ satisfying $F(f_\ast d) = c$, and subject to the two
  functoriality conditions $F(d, 1_{Fd}) = 1_d$ and
  $F(d, gf) = F(f_\ast(d), g) \circ F(d,f)$.
\end{Defn}
For example, we can obtain a cofunctor from any discrete
  opfibration. A \emph{discrete opfibration} is a functor $F \colon \D
  \rightarrow \C$ with the property that, for any $d \in \D$ and map
  $f \colon Fd \rightarrow c$ in $\C$, there is a unique map $\tilde f
  \colon d \rightarrow d'$ of $\D$ with $F\tilde f = f$ (Lawson
  in~\cite{Lawson2013Pseudogroups} calls these \emph{covering
    functors}). We can define from this a
  cofunctor $\C \rightsquigarrow \D$ whose action on objects is that
  of $F$, and whose action on arrows takes $F(d,f)$ to be the unique
  map $\tilde f$ described above.

We now describe the appropriate adaption of the notion of cofunctor to
the case of partite internal categories.

\begin{Defn}
  \label{def:4}
  Let $\mathbb{A}$ be an $I$-partite and $\mathbb{B}$ a $J$-partite
  internal category in $\C$. A \emph{partite cofunctor} $F \colon \mathbb{A}
  \rightsquigarrow \mathbb{B}$ comprises:
  \begin{itemize}
  \item A mapping on component-sets $I \rightarrow J$ written $i \mapsto Fi$;
  \item An action on objects, given by total maps $F_i \colon B_{Fi} \rightarrow A_i$ for each $i \in
    I$; and
  \item An action on arrows, given by total maps
    $F_{ij} \colon A_{ij} \times_{A_i} B_{Fi} \rightarrow B_{Fi,Fj}$
    for each $i,j \in I$ which are compatible with source in the sense
    that $\sigma_{Fi,Fj} F_{ij} = \pi_2$,
  \end{itemize}
  all subject to the commutativity of the following diagrams (the
  \emph{target}, \emph{identity} and \emph{composition} axioms).
  \begin{equation*}
    \cd[@!C@C-3em@-0.5em]{
      {A_{ij} \times_{A_i} B_{Fi}} \ar[rr]^-{F_{ij}} \ar[d]_{\pi_1} & &
      {B_{Fi,Fj}} \ar[d]^{\tau_{Fi,Fj}} \\
      {A_{ij}} \ar[r]_-{\tau_{ij}} & A_j & B_{Fj} \ar[l]^-{F_{j}}
    } \qquad 
    \cd[@!C@C-3em]{
      B_{Fi} \ar[rr]^-{(\eta_i F_i,1)} \ar[dr]_-{\eta_{Fi}} & & A_{ii} \times_{A_i} B_{Fi}
      \ar[dl]^-{F_{ii}} \\ & B_{Fi,Fi}
 }
  \end{equation*}
  \begin{equation*}
    \cd{
      A_{jk} \times_{A_j} B_{Fi, Fj}
      \ar[r]^-{\cong} &
      \sh{l}{1em}{(A_{jk} \times_{A_j} B_{Fj}) \times_{B_{Fj}} B_{Fi, Fj}}
      \ar[r]^-{F_{jk} \times 1} &
      B_{Fj, Fk}\times_{B_{Fj}} B_{Fi, Fj}
      \ar[d]^-{\mu_{Fi,Fj,Fk}}\\
      A_{jk} \times_{A_j} A_{ij} \times_{A_i} B_{Fi}
      \ar[r]^-{\mu_{ijk} \times_{A_i} 1}
      \ar[u]_-{1 \times F_{ij}} &
      A_{ik} \times_{A_i} B_{Fi} \ar[r]^-{F_{ik}} & B_{Fi,Fk}\rlap{ .}
    }
  \end{equation*}
  The identity on $\mathbb{A}$ has all data given by
  identities, while the composition of 
  $F \colon \mathbb{A} \rightsquigarrow \mathbb{B}$ and
  $G \colon \mathbb{B} \rightsquigarrow \mathbb{C}$ has
  $(GF)(i) = G(F(i))$,
  $(GF)_i = F_i \circ G_{Fi} \colon C_{GFi} \rightarrow A_i$, and
  $(GF)_{ij}$ given by
  \begin{equation*}
    A_{ij} \times_{A_i} C_{GFi} \cong (A_{ij} \times_{A_i} B_{Fi})
    \times_{B_{Fi}} C_{GFi} \xrightarrow{\!F_{ij} \times 1\!} B_{Fi,Fj}
    \times_{B_{Fi}} C_{GFi} \xrightarrow{\!G_{Fi,Fj}\!} C_{GFi,
      GFj}\text{ .}
  \end{equation*}
  In this way, we obtain a category $\partc$ of partite
  internal categories and partite cofunctors in $\C$. We write
  $\parte$ for the full subcategory whose objects are the
  \emph{source-\'etale} partite internal categories---those whose
  source maps $\sigma_{ij}$ are all local homeomorphisms.
\end{Defn}


\subsection{The right adjoint}
\label{sec:right-adjoint}
We now start constructing the adjunction
\begin{equation*}
  \cd{
    {\partc} \ar@<-4.5pt>[r]_-{\Phi} \ar@{<-}@<4.5pt>[r]^-{\Psi} \ar@{}[r]|-{\bot} &
    {\rcat} 
  }
\end{equation*}
whose restriction to fixpoints will yield the desired equivalence
$\parte \simeq \hcon$. We begin by describing the
right adjoint $\Phi$; much like $\Gamma$ in the presheaf--total map
adjunction, this will be based on the idea of taking partial sections.
We first give its construction and then
verify well-definedness.

\begin{Defn}
  \label{def:5}
  Let $\mathbb{A}$ be an $I$-partite internal category in $\C$. Its
  \emph{externalisation} $\pi_\mathbb{A} \colon \Phi\mathbb{A}
  \rightarrow \C$ is the object of $\rcat$
  defined as follows:
  \begin{itemize}
  \item The object-set of $\Phi\mathbb{A}$ is $I$;
  \item $\Phi\mathbb{A}(i,j)$ is the set $\Gamma(\sigma_{ij})$ of
    partial sections of $\sigma_{ij} \colon A_{ij} \rightarrow A_i$;
  \item The identity map $1_i \in \Phi\mathbb{A}(i,i)$ is the (total)
    section $\eta_i \colon A_i \rightarrow A_{ii}$;
  \item The composite $t \ast s$ of $s \in \Phi\mathbb{A}(i,j)$ and
    $t \in \Phi\mathbb{A}(j,k)$ is the partial section
    \begin{equation*}
      A_{i} \xrightarrow{s} A_{ij} \xrightarrow{\tau_{ij}^\ast(t)} A_{jk}
      \times_{A_j} A_{ij} \xrightarrow{\mu_{ijk}} A_{ik}
    \end{equation*}
    of $\sigma_{ik}$ obtained by applying Lemma~\ref{lem:8} in the situation:
    \begin{equation}\label{eq:6}
      \cd[@!C@C-5em@-0.4em]{
        A_{jk} \ar[dr]^-{\sigma_{jk}} & & \ar[ll]_-{\pi_1} \pullbackcornerl
        A_{jk} \times_{A_j} A_{ij} \ar[dr]^-{\pi_2} \ar[rrrr]^-{\mu_{ijk}}
        & & & & A_{ik}\rlap{ ;} \ar[ddll]^-{\sigma_{ik}}
        \ar@/_1em/@{<--}[ddll]_-{t \ast s}\\
        & A_j \ar@/^1em/@{-->}[ul]^-{t} & & A_{ij}
        \ar@/^1em/@{-->}[ul]^(0.6){\tau_{ij}^\ast t} \ar[dr]^-{\sigma_{ij}}
        \ar[ll]^-{\tau_{ij}} \\ & &
        & & A_i \ar@/^1em/@{-->}[ul]^(0.6){s}
      }
    \end{equation}
  \item The restriction of $s \colon A_i \rightarrow A_{ij}$ in
    $\Phi\mathbb{A}(i,j)$ is
    $\eta_i\overline{s} \colon A_i \rightarrow A_{ii}$ in
    $\Phi\mathbb{A}(i,i)$;
  \item The functor
  $\pi_\mathbb{A} \colon \Phi \mathbb{A} \rightarrow \C$ is
  given on objects by $i \mapsto A_i$, and on maps by sending
  $s \in \Phi\mathbb{A}(i,j)$ to
  $\tau_{ij} \circ s \colon A_i \rightarrow A_j$.
  \end{itemize}
\end{Defn}

\begin{Defn}
  \label{def:35}
  Let $F \colon \mathbb{A} \rightsquigarrow \mathbb{B}$ be a partite
  cofunctor in $\C$. We define a map
  $(\Phi F, \varpi^F) \colon (\Phi \mathbb{A}, \pi_\mathbb{A})
  \rightarrow (\Phi \mathbb{B}, \pi_\mathbb{B})$ in $\rcat$ as
  follows:
  \begin{itemize}
  \item On objects, $\Phi F$ is defined by $i \mapsto Fi$;
  \item On morphisms, given $s \in \Phi \mathbb{A}(i,j)$ we define
    $\Phi F(s) \in \Phi \mathbb{B}(Fi,Fj)$ as the partial section
    \begin{equation*}
      B_{Fi} \xrightarrow{F_i^\ast s} A_{ij} \times_{A_i} B_{Fi}
      \xrightarrow{F_{ij}} B_{Fi,Fj}
    \end{equation*}
    of $\sigma_{Fi, Fj}$ obtained by applying Lemma~\ref{lem:8} in the
    situation
    \begin{equation}\label{eq:8}
      \cd[@!C@C-5em@R+0.3em]{
        A_{ij}\ar[dr]^-{s_{ij}} & & A_{ij} \times_{A_i} B_{Fi}
        \ar[dr]^-{\pi_2} \ar[ll]_-{\pi_1} \ar[rr]^-{F_{ij}} & & B_{Fi,Fj}\rlap{ ;}
        \ar[dl]^-{\sigma_{Fi,Fj}}\\
        & A_i \ar@/^1em/@{-->}[ul]|(0.6){s} & & B_{Fi} \ar[ll]^-{F_i}
        \ar@/^1em/@{-->}[ul]|(0.65){F_i^\ast s}
      }
    \end{equation}
  \item $\varpi^F \colon \pi_\mathbb{B}
    \circ \Phi F \Rightarrow \pi_\mathbb{A}$ has components
    $(\varpi^F)_i = F_i \colon B_{Fi} \rightarrow A_i$.
  \end{itemize}
\end{Defn}
 
\begin{Prop}
  \label{prop:2}
  The assignations of Definitions~\ref{def:5} and~\ref{def:35}
  underlie a well-defined functor $\Phi \colon \partc
  \rightarrow \rcat$ whose image lands inside $\hcon$.
\end{Prop}
\begin{proof}
  To begin with, let
  $\mathbb{A} \in \partc$. We will proceed in stages to verify that
  $(\Phi_\mathbb{A}, \pi_\mathbb{A})$ is a well-defined object of
  $\rcat$, and that it in fact lives in $\hcon$.

  \textbf{$\Phi \mathbb{A}$ is a category}. For the left unit law, note that
  $\tau_{jj}^\ast(\eta_j) = (\eta_j \tau_{jj}, \overline{\eta_j
    \tau_{jj}}) = (\eta_j \tau_{jj}, 1)$, so that for any
  $s \in \Phi\mathbb{A}(i,j)$ we have
  \begin{equation}\label{eq:15}
    \eta_j \ast s  = \mu_{ijj}(\eta_j \tau_{jj}, 1)s
  \end{equation}
  which equals $s$ by the left unit law for $\mathbb{A}$. For the
  right unit law, for any $t \in \Phi\mathbb{A}(i,j)$ we have
  $\tau_{ii}^\ast(t) \eta_i = (t \tau_{ii}\eta_i, \overline{t
    \tau_{ii}}\eta_i) = (t, \eta_i \overline{t\tau_{ii}\eta_i}) = (t,
  \eta_i \overline{t})$, and so
  \begin{equation}\label{eq:31}
    t \ast \eta_i  = \mu_{iij}(t,
    \eta_i \overline{t}) = \mu_{iij} (t, \eta_i \sigma_{ij}t)= \mu_{iij}(1,
    \eta_i \sigma_{ij})t
  \end{equation}
  which equals $t$ by the right unit law for $\mathbb{A}$. Finally,
  for associativity, given $s \in \Phi\mathbb{A}(i,j)$ and
  $t \in \Phi\mathbb{A}(j,k)$ and $u \in \Phi\mathbb{A}(k,\ell)$, we
  first calculate that
  \begin{align*}
    u \ast (t \ast s) &= \mu_{ik\ell} \circ \tau_{ik}^\ast u \circ (t \ast s) =
    \mu_{ik\ell} \circ \tau_{ik}^\ast u \circ \mu_{ijk} \circ
    \tau_{ij}^\ast t
    \circ s \\
    \text{and }(u \ast t) \ast s &= \mu_{ij\ell} \circ \tau_{ij}^\ast(u \ast t) \circ s\rlap{ .}
  \end{align*}
  Now consider the following diagrams; here, and subsequently, we may omit
  fibre product symbols for brevity.
\begin{equation*}
    \cd[@!C@C-5em@+1.5em]{
      A_{k\ell}A_{jk}A_{ij}
       \ar[rr]^-{\mu_{ijk} 1}
       \ar[d]|-{\pi_{23}}
       \ar@{<--}@/_1em/[d]_{\mu_{ijk}^\ast \tau_{ik}^\ast u} & & 
      A_{k\ell}A_{ik}
       \ar[d]|-{\pi_1}
       \ar@{<--}@/_1em/[d]_-{\tau_{ik}^\ast u} \\
      A_{jk}A_{ij}
      \ar[rr]^-{\mu_{ijk}}
      & &
      A_{ik}
     } \ \ \ 
    \cd[@C-0.3em@-0.6em]{
      A_{k\ell}A_{jk}A_{ij}
      \ar[rr]^-{\pi_{12}}
      \ar[dr]^-{1\mu_{jk\ell}}
      \ar[dd]^{\pi_{23}}
      \ar@{<--}@/_1em/[dd]_(0.4){\pi_1^\ast \tau_{jk}^\ast u} & &
      A_{k\ell}A_{jk}
      \ar[dr]^-{\mu_{jk\ell}}
      \ar[dd]^(0.68)*-<0.3em>{^{\pi_2}}|-\hole
      \ar@{<--}@/_1em/[dd]_(0.35){\tau_{jk}^\ast u}|-\hole\\ &
      A_{j\ell}A_{ij}
      \ar[rr]^(0.65)*-<0.7em>{^{\pi_1}}
      \ar[dd]^(0.65){\pi_2}
      \ar@{<--}@/_1em/[dd]|(0.3){\tau_{ij}^\ast(u \ast t)\ \ \ } & &
      A_{j\ell}
      \ar[dd]^-{\sigma_{j\ell}}
      \ar@{<--}@/_1em/[dd]_(0.4){u \ast t}\\
      A_{jk}A_{ij}
      \ar[rr]^(0.65){\pi_1}|-\hole
      \ar[dr]|-{\pi_2}
      \ar@{<--}@/_1em/[dr]_(0.35){\tau_{ij}^\ast t} & &
      A_{jk}
      \ar@{<--}@/_1em/[dr]_(0.35){t}
      \ar[dr]|-{\sigma_{jk}} \\ &
      A_{ij}
      \ar[rr]_-{\tau_{ij}} & &
      A_j\rlap{ .}
    }
  \end{equation*}

  The solid part of the left square is a pullback, and the section to
  its left obtained by pulling back the one to the right via
  Lemma~\ref{lem:8}(iii). In particular, the upwards-pointing square
  commutes, and so we can rewrite $u \ast (t \ast s)$ as
  \begin{equation}\label{eq:1}
    u \ast (t \ast s) = \mu_{ik\ell} \circ (\mu_{ijk} \times 1) \circ
    \mu_{ijk}^\ast \tau_{ik}^\ast u \circ \tau_{ij}^\ast t
    \circ s\rlap{ .}
  \end{equation}
  
  On the other hand, the front, back, top and bottom faces of the
  cube right above are all pullbacks, and so we can pull back the commuting
  diagram of sections to the right, which defines $u \ast t$, to obtain
  a commuting diagram of sections giving an expression for
  $\tau_{ij}^\ast(u \ast t)$. Using this, we can rewrite
  $(u \ast t) \ast s$ as
  \begin{equation}\label{eq:11}
    \begin{aligned}
    (u \ast t) \ast s &= \mu_{ij\ell} \circ (1 \times \mu_{jk\ell}) \circ
    \pi_1^\ast \tau_{jk}^\ast u \circ \tau_{ij}^\ast t \circ s \\
    &= \mu_{ij\ell} \circ (1 \times \mu_{jk\ell}) \circ
    \mu_{ijk}^\ast \tau_{ik}^\ast u \circ \tau_{ij}^\ast t \circ s\rlap{ ,}
    \end{aligned}
  \end{equation}
  where from the first to the second line we have 
  $\pi_1^\ast \tau_{jk}^\ast u = \mu_{ijk}^\ast \tau_{ik}^\ast u$ by
  unicity in Lemma~\ref{lem:8}(iii), as both are pullbacks of the
  partial section $u$ along $\tau_{ik}\mu_{ijk} = \tau_{jk}\pi_1$.
  Comparing~\eqref{eq:1} and~\eqref{eq:11}, we conclude using
  associativity for $\mathbb{A}$.
  
  \textbf{$\Phi \mathbb{A}$ is a restriction category}. To avoid
  confusion we will in this proof write $(\tilde{\ \ })$ for the restriction in
  $\Phi \mathbb{A}$ and $(\bar{\ \ })$ for that in $\C$. Note first
  that, for any $s \in \Phi\mathbb{A}(i,j)$ and
  $t \in \Phi \mathbb{A}(i,k)$, our calculations for the right unit
  law show that
  \begin{equation}\label{eq:44}
    s \ast \tilde{t} = s\overline{t} \colon A_i \rightarrow A_{ij}
  \end{equation}
  Given this, the first three restriction axioms
  $s \ast \tilde{s} = s$,
  $\tilde{s} \ast \tilde{t} = \tilde{t} \ast \tilde{s}$ and
  $\tilde{t \ast \tilde{s}} = \tilde{t}\ast\tilde{s}$ each follow
  immediately from the corresponding axiom in $\C$. It remains only to
  verify that $\tilde{t} \ast s = s \ast \widetilde{t \ast s}$ for all
  $s \in \Phi \mathbb{A}(i,j)$ and $t \in \Phi \mathbb{A}(j,k)$. Now
  observe that
  $\tau_{ij}^\ast(\eta_j \overline{t}) = (\eta_j
  \overline{t}\tau_{ij}, \overline{\eta_j \overline{t} \tau_{ij}}) =
  (\eta_j \tau_{ij}\overline{t \tau_{ij}}, \overline{\eta_j
    \tau_{ij}\overline{t \tau_{ij}}}) = (\eta_{j}\tau_{ij},
  1)\overline{t \tau_{ij}}$, so that
  \begin{equation}\label{eq:45}
    \tilde{t} \ast s =
    \mu_{ijj}(\eta_j \tau_{ij}, 1)\overline{t \tau_{ij}}s =
    \overline{t \tau_{ij}}s = s \overline{t \tau_{ij}s}\rlap{ .}
  \end{equation}
  On the other hand, we have by totality of $\mu_{ijk}$ and $\pi_1$ that
  \begin{equation*}
    s \ast \widetilde{t \ast s} =
    s\,\overline{\mu_{ijk} \tau_{ij}^\ast (t) s} = s\,\overline{\tau_{ij}^\ast
      (t) s} = s\,\overline{\pi_1  \tau_{ij}^\ast(t)
       s} = s\, \overline{t \tau_{ij} s} \rlap{ .}
  \end{equation*}

  \textbf{$\Phi \mathbb{A}$ is a join restriction category}. Note
that since $t \ast \tilde s = t \bar s \colon A_i \rightarrow A_{ij}$,
the compatibility and order relations on
$\Phi \mathbb{A}(i,j)$ coincide with those of the sheaf
$\Gamma(\sigma_{ij})$; since the latter admits joins, so too does the
former. To see that these joins are preserved by precomposition in
$\Phi \mathbb{A}$, suppose that
$\{t_i\} \subseteq \Phi \mathbb{A}(j,k)$ is a compatible family and
$s \in \Phi\mathbb{A}(i,j)$. The unicity in Lemma~\ref{lem:8}(iii)
implies that the family $\tau_{ij}^\ast(t_i)$ is compatible and that
$\tau_{ij}^\ast(\bigvee_i t_i) = \bigvee_x \tau_{ij}^\ast(t_i)$. We
therefore conclude that
  \begin{equation*}
\textstyle    (\bigvee_i t_i) \ast s = \mu_{ijk}\tau_{ij}^\ast(\bigvee_i t_i)s =
    \mu_{ijk}\bigvee_i \tau_{ij}^\ast(t_i)s = \bigvee_i
    \mu_{ijk}\tau_{ij}^\ast(t_i)s = \bigvee_i (t_i \ast s)\text{ .}
  \end{equation*}
  
  \textbf{$\pi_\mathbb{A} \colon \Phi\mathbb{A} \rightarrow \C$ is a
  hyperconnected restriction functor}. For functoriality, note that
  $\pi_\mathbb{A}(1_i) = \tau_{ii} \eta_i = 1_{A_i}$, and that for
  $s \in \Phi\mathbb{A}(i,j)$ and $t \in \Phi\mathbb{A}(j,k)$ we have
\begin{align*}
  \pi_\mathbb{A}(t \ast s) = \tau_{ik} \mu_{ijk} \tau_{ij}^\ast(t) s
  = \tau_{jk}  \pi_1 \tau_{ij}^\ast (t)  s
  = \tau_{jk}  t  \tau_{ij} s = \pi_\mathbb{A}(t)
  \circ \pi_\mathbb{A}(s)\rlap{ .}
\end{align*}
Now $\pi_\mathbb{A}$ preserves restrictions as
$\pi_\mathbb{A}(\tilde{s}) = \pi_\mathbb{A}(\eta_i \overline{s}) =
\tau_{ii} \eta_i \overline{s} = \overline{s} = \overline{\tau_{ij} s}
= \overline{\pi_\mathbb{A}(s)}$; and is hyperconnected as the
mapping $\O_{\Phi \mathbb{A}}(i) \rightarrow \O_\C(A_i)$ has inverse
 $e \mapsto \eta_i e$.

This proves that $(\Phi_\mathbb{A}, \pi_\mathbb{A})$ is a
well-defined object of $\rcat$, and indeed of $\hcon$.
Consider now $F \colon \mathbb{A} \rightsquigarrow \mathbb{B}$ 
a partite cofunctor; we will show 
$(\Phi F, \varpi^F) \colon (\Phi \mathbb{A}, \pi_\mathbb{A})
\rightarrow (\Phi \mathbb{B}, \pi_\mathbb{B})$ is a well-defined map of $\hcon$. Again we
proceed by stages.

\textbf{$\Phi F$ is functorial}. For $i \in \Phi \mathbb{A}$,
  the identity $\eta_i \in \Phi \mathbb{A}(i,i)$ maps to
  \begin{equation*}
    (\Phi F)(\eta_i) = F_{ii}F_i^\ast(\eta_i) = F_{ii}(\eta_i F_i,
    \overline{\eta_i F_i}) = F_{ii}(\eta_i F_i, 1) = \eta_{Fi}
  \end{equation*}
  as required, using at the last step the unit axiom for $F$.
  For binary functoriality, let 
  $s \in \Phi \mathbb{A}(i,j)$ and $t \in \Phi\mathbb{A}(j,k)$.
  Writing $i',j',k'$ for $Fi,Fj,Fk$, we
  have
  \begin{equation*}
    \Phi F(t) \ast \Phi F(s) = \mu_{i',j',k'} \circ
    \tau_{i',j'}^\ast\bigl(\Phi F(t)\bigr) \circ \Phi F(s) = \mu_{i',j',k'} \circ
    \tau_{i',j'}^\ast\bigl(\Phi F(t)\bigr) \circ F_{ij}\circ F_i^\ast s\text{ .}
  \end{equation*}
  To calculate $\tau_{i',j'}^\ast\bigl(\Phi F(t)\bigr)$ we 
  consider the diagram left below, wherein $F_{jk}1$ to the left denotes the
  top side of the composition axiom in Definition~\ref{def:3}:
\begin{equation*}
    \cd[@C-0.5em]{
      & B_{j',k'}B_{i',j'}
      \ar[rr]^(0.68)*-<0.7em>{^{\pi_1}}
      \ar[dd]^(0.65){\pi_2}
      \ar@{<--}@/_1em/[dd]|(0.4){\tau_{i',j'}^\ast(\Phi F(t))\ \ \ \ \
        \ \ \ } & &
      B_{j',k'}
      \ar[dd]^-{\sigma_{j',k'}}
      \ar@{<--}@/_1em/[dd]|(0.4){\Phi F(t)\,}\\
      A_{jk}B_{i',j'}
      \ar[rr]^(0.65){1 \tau_{i',j'}}|-\hole
      \ar[dr]|-{\pi_2}
      \ar@{<--}@/_1em/[dr]_(0.35){\tau_{i',j'}^\ast F_j^\ast t}
      \ar[ur]^-{F_{jk}1} & &
      A_{jk}B_{j'}
      \ar@{<--}@/_1em/[dr]_(0.35){F_j^\ast t}
      \ar[dr]|-{\pi_2}
      \ar[ur]^-{F_{jk}} \\ &
      B_{i',j'}
      \ar[rr]_-{\tau_{i',j'}} & &
      B_{j'}
    } \ \ 
    \cd[@R+1em]{
      {A_{jk}A_{ij}B_{i'}}
      \ar[r]^-{1F_{ij}}
      \ar[d]^{\pi_{23}}
      \ar@{<--}@/_1em/[d]|-{F_{ij}^\ast\tau_{i',j'}^\ast F_j^\ast t\ \
      \ \ \ \ \ \ } &
      {A_{jk}B_{i',j'}} \ar[d]^{\pi_2}
      \ar@{<--}@/_1em/[d]_-{\tau_{i',j'}^\ast F_j^\ast t}\\
      {A_{ij}B_{i'}} \ar[r]_-{F_{ij}} &
      {B_{i',j'}}\rlap{ .}
    }
  \end{equation*}
  All three horizontal faces are pullbacks, and so we can pull back
  the commuting diagram of sections to the right to obtain a commuting
  diagram of sections giving an expression for
  $\tau_{i',j'}^\ast\bigl(\Phi F(t)\bigr)$. This gives the first
  equality in:
  \begin{equation}\label{eq:14}
    \begin{aligned}
      \Phi F(t) \ast \Phi F(s) &= \mu_{i',j',k'} \circ (F_{jk} \times
      1) \circ \tau_{i',j'}^\ast F_j^\ast t \circ
      F_{ij}\circ F_i^\ast s\\
      &= \mu_{i',j',k'} \circ (F_{jk} \times 1) \circ (1 \times
      F_{ij}) \circ F_{ij}^\ast\tau_{i',j'}^\ast F_j^\ast t \circ F_i^\ast s \\
      &= \mu_{i',j',k'} \circ (F_{jk} \times 1) \circ (1 \times
      F_{ij}) \circ \pi_1^\ast \tau_{ij}^\ast t \circ F_i^\ast s\rlap{
        .}
    \end{aligned}
  \end{equation}
  The second equality comes from pulling back the section
  $\tau^\ast_{i',j'}F_j^\ast t$ as in the square right above, while
  the third follows from the target axiom
  $F_j \tau_{i',j'}F_{ij} = \tau_{ij}\pi_1$.
  
  On the other hand $\Phi F(t \ast s) = F_{ik} \circ
  F_{i}^\ast(t \ast s)$, and by considering the diagram
  \begin{equation*}
    \cd[@C-0.3em@-0.6em]{
      A_{jk}A_{ij}B_{i'}
      \ar[rr]^-{\pi_{12}}
      \ar[dr]^-{1\mu_{jk\ell}}
      \ar[dd]^{\pi_{23}}
      \ar@{<--}@/_1em/[dd]_(0.4){\pi_1^\ast \tau_{ij}^\ast t} & &
      A_{ijk}
      \ar[dr]^-{\mu_{ijk}}
      \ar[dd]^(0.68)*-<0.3em>{^{\pi_2}}|-\hole
      \ar@{<--}@/_1em/[dd]_(0.35){\tau_{ij}^\ast t}|-\hole\\ &
      A_{ik}B_{i'}
      \ar[rr]^(0.68)*-<0.7em>{^{\pi_1}}
      \ar[dd]^(0.65){\pi_2}
      \ar@{<--}@/_1em/[dd]|(0.3){F_i^\ast(t \ast s)\ \ \ } & &
      A_{ik}
      \ar[dd]^-{\sigma_{ik}}
      \ar@{<--}@/_1em/[dd]_(0.4){t \ast s}\\
      A_{ij}B_{i'}
      \ar[rr]^(0.65){\pi_1}|-\hole
      \ar[dr]|-{\pi_2}
      \ar@{<--}@/_1em/[dr]_(0.35){F_i^\ast s} & &
      A_{ij}
      \ar@{<--}@/_1em/[dr]_(0.35){s}
      \ar[dr]|-{\sigma_{ij}} \\ &
      B_{i'}
      \ar[rr]_-{F_i} & &
      A_i\rlap{ .}
    }
  \end{equation*}
 we can expand out the term $F_i^\ast(t
  \ast s)$ and conclude that
  \begin{equation}
    \label{eq:13}
    \Phi F(t \ast s) = F_{ik} \circ (1 \times \mu_{jk\ell}) \circ
    \pi_1^\ast \tau_{ij}^\ast t \circ F_i^\ast s\rlap{ .}
  \end{equation}
  Comparing~\eqref{eq:14} and~\eqref{eq:13} yields the desired
  equality by rewriting using the composition axiom for the cofunctor
  $F$.
  
  \textbf{$\Phi F$ is a restriction functor}. We
  calculate for any $s \in \Phi \mathbb{A}(i,j)$ that
  \begin{align*}
    \Phi F(\tilde{s}) &= F_{ii} F_i^\ast(\eta_i \overline{s}) =
    F_{ii}(\eta_i \overline{s}F_i, \overline{\eta_i \overline{s}F_i})
    = F_{ii}(\eta_i F_i \overline{sF_i}, \overline{\eta_i F_i}\,
    \overline{sF_i}) = F_{ii}(\eta_i F_i, 1)\overline{sF_i} \\ &=
    \eta_{Fi} \overline{s F_i} = \eta_{Fi}\overline{F_i^\ast s} =
    \eta_{Fi}\overline{F_{ii} F_i^\ast s} = \widetilde{\Phi F(s)}\rlap{ ,}
  \end{align*}
  where from the first to the second line we use the unit axiom for $F$.

  \textbf{$\varpi^F$ is a total natural transformation $\pi_\mathbb{B}
    \circ \Phi F \Rightarrow \pi_\mathbb{A}$}. We compute that
  that
  \begin{align*}
    \varpi^F_j \circ \pi_\mathbb{B}(\Phi F(s)) = 
    F_j \tau_{Fi,Fj} F_{ij} F_i^\ast(s) = \tau_{ij}
    \pi_1  
    F_i^\ast(s) = \tau_{ij} s F_i = \pi_\mathbb{A}(s) \circ
    \varpi^F_i
  \end{align*}
  using the target axiom for $F$ at the second equality.

  We have thus verified well-definedness of $\Phi$ on morphisms, and
  it remains only to prove functoriality of $\Phi$ itself.
  Preservation of identities is easy. As for binary composition, given
  partite cofunctors $F \colon \mathbb{A} \rightsquigarrow \mathbb{B}$
  and $G \colon \mathbb{B} \rightsquigarrow \mathbb{C}$ it is clear
  that $\Phi G \circ \Phi F$ and $\Phi(GF)$ agree on objects. On maps,
  given $s \in \Phi \mathbb{A}(i,j)$, we have $\Phi G(\Phi F(s))$ and
  $\Phi(GF)(s)$ given by the upper and lower composites in:
  \begin{equation*}
    \cd{
      C_{GFi} \ar[r]^-{G_{Fi}^\ast(F_{ij} \circ F_i^\ast(s))}
      \ar[d]_-{{(F_i G_{Fi})^\ast(s)}} &
      B_{Fi,Fj} \times_{B_{Fi}} C_{GFi} \ar[r]^-{G_{Fi,Fj}} &
      C_{GFi,GFj} \\
      A_{ij} \times_{A_i} C_{GFi} \ar[r]_-{\cong} &
      (A_{ij} \times_{A_i} B_{Fi}) \times_{B_{Fi}} C_{GFi}
      \ar[u]_-{F_{ij} \times 1}\rlap{ .}
    }
  \end{equation*}
  So it suffices to show the left region commutes, which follows much
  as previously using Lemma~\ref{lem:8}. So
  $\Phi(GF) = \Phi(G)\circ \Phi(F)$; and finally, it is clear that the
  $2$-cell components $\varpi^F$ and $\varpi^G$ paste together to
  yield $\varpi^{GF}$.
\end{proof}



\subsection{The left adjoint}
\label{sec:left-adjoint}

In this section, we show that
$\Phi \colon \partc \rightarrow \rcat$ has a left adjoint
$\Psi$. Much like $\Delta$ in the presheaf--total map adjunction,
$\Phi$ will be constructed by glueing suitable local atlases built
from $\O(X)$-valued equality of maps. The equality required is a
slight variant of the one in Definition~\ref{def:34}.

\begin{Defn}
  \label{def:9}
  Let $P \colon \A \rightarrow \C$ in $\rcat$. Given maps
  $f,g \colon i \rightrightarrows j$ in $\A$ we define the restriction
  idempotent $\fneq f g \in \O(Pi)$ by
  \begin{equation*}
    \fneq f g \defeq \bigvee \{Pe : e \in \O(i), e \leqslant \bar f
    \bar g,  fe = ge\}\rlap{ .}
  \end{equation*}
\end{Defn}
\begin{Rk}
  \label{rk:2}
  When $\A$ has joins of restriction idempotents which are preserved
  by $P$, we can form $\sheq f g$ in $\A$, and by join-preservation
  have $\fneq f g = P\sheq f g$.
\end{Rk}
The following lemma is now a slight variant of Lemma~\ref{lem:30}.
\begin{Lemma}
  \label{lem:15}
  Given $P \colon \A \rightarrow \C$ in $\rcat$ and $f,g,h \colon
  i \rightarrow j$ in $\A$, we have that
  \begin{enumerate}[(i)]
  \item $\fneq f f = P\overline{f}$;
  \item $\fneq f g = \fneq g f$;
  \item $\fneq g h \circ \fneq f g \leqslant \fneq f h$;
  \item $Pf \circ \fneq f g = Pg \circ \fneq f g$.
  \end{enumerate}
  If $f,g \colon i \rightarrow j$ and $u,v \colon j \rightarrow k$
  then $\fneq f g\overline{\fneq u v \circ Pf} \leqslant \fneq {uf}{vg}$.
\end{Lemma}
\begin{proof}
  For (i)--(iv) we transcribe the proof of Lemma~\ref{lem:30}.
  For the final claim, it suffices by distributivity to prove
  that, if $d \leqslant \bar f \bar g$ and $e \leqslant \bar u \bar v$
  are such that $fd = gd$ and $ue=ve$, then $Pd \circ \overline{Pe
    \circ Pf} \leqslant \fneq {uf}{vg}$; and for this, it suffices to
  show $d\overline{ef} \leqslant \overline{uf}\overline{vg}$ and that
  $ufd\overline{ef} = vgd\overline{ef}$. Since $e \leqslant
  \overline{u}$ we have $d\overline{ef} \leqslant \overline{ef}
  \leqslant \overline{uf}$; moreover, we have
  $d \overline{ef} = \overline{ef}d = \overline{efd} = \overline{egd}
  \leqslant \overline{vg}$. Finally, we have $ufd\overline{ef} =
  uf\overline{ef}d = uefd = vefd = vf\overline{ef}d = vfd\overline{ef}
  = vgd\overline{ef}$, as desired.
\end{proof}
We are now ready to construct the left adjoint $\Psi$. We will first
define the action of $\Psi$ on objects, and the unit maps for the
desired adjunction. We then check the definitions are well-posed, and
finally verify adjointness.
\begin{Defn}
  \label{def:7}
  Let $P \colon \A \rightarrow \C$ be an object of $\rcat$. Writing
  $I = \mathrm{ob}(\A)$, the \emph{internalisation} of $P$
  is the $I$-partite internal category $\Psi P$ in $\C$ for which:
  \begin{itemize}[itemsep=0.20\baselineskip]
  \item The object of objects $(\Psi P)_i$ is $Pi$;
  \item The object of arrows with its source map
    $\sigma_{ij} \colon \Psi P_{ij} \rightarrow Pi$ is obtained
    as a glueing of the $\A(i,j)$-object local atlas on $Pi$ with
    components $\fneq f g$; this is a well-defined local atlas by
    Lemma~\ref{lem:15}. We write
    $\sigma_{ij} = \bigvee_{f \in \A(i,j)} p_f$ and write $s_f$ for
    the partial inverse of $p_f$.
  \item $\tau_{ij} \colon \Psi P_{ij} \rightarrow \Psi P_j$ is the
    unique map such that
    \begin{equation}\label{eq:9}
      \tau_{ij} \circ s_f = Pf \colon Pi \rightarrow Pj \qquad
      \text{for all $i,j \in \A(i,j)$.}
    \end{equation}
  \item $\eta_i \colon \Psi P_i \rightarrow \Psi P_{ii}$ is the
    partial section 
    $s_{1_i} \colon Pi \rightarrow \Psi P_{ii}$.
  \item $\mu_{ijk} \colon \Psi P_{jk} \times_{Pj} \Psi P_{ij}
    \rightarrow \Psi P_{ik}$ is the unique map such that
    \begin{equation}\label{eq:12}
      \mu_{ijk} \circ \tau_{ij}^\ast(s_g) \circ s_f = s_{gf} \qquad
      \text{for all $f \in \A(i,j), g \in \A(j,k)$.}
    \end{equation}
  \end{itemize}
  We also define a restriction functor $\eta_P \colon \A \rightarrow
  \Phi(\Psi P)$ where:
  \begin{itemize}
  \item On objects, $\eta_P$ is the identity;
\item On maps, $\eta_P(f \colon i \rightarrow j)$ is the partial
  section $s_f \colon Pi \rightarrow \Psi P_{ij}$ of $\sigma_{ij}$.
  \end{itemize}
  Finally, we define the \emph{unit map at $P$} to be the
  following map of $\rcat$:
  \begin{equation}\label{eq:16}
    \cd[@!C@C-2em]{
      {\A} \ar[rr]^-{\eta_P} \ar[dr]_-{P} & \twoeq{d}&
      {\Phi(\Psi P)}\rlap{ .} \ar[dl]^-{\pi_{\Psi P}} \\ &
      {\C}
    }
  \end{equation}
\end{Defn}

\begin{Prop}
  \label{prop:4}
  For any $P \colon \A \rightarrow \C$ in $\rcat$, $\Psi P$ is a
  well-defined object of $\parte$, and~\eqref{eq:16} is a
  well-defined map of $\rcat$.
\end{Prop}
\begin{proof}
  Like before, we check well-definedness  in stages.

  \textbf{The data of $\Psi P$ are well-defined}.
  For $\Psi P_i$, $\Psi P_{ij}$ and $\sigma_{ij}$ there is nothing to do.
  For $\tau_{ij}$, we observe that the family
  of maps $(Pf \colon Pi \rightarrow Pj)_{f \in \A(i,j)}$ satisfies
  \begin{equation*}
 Pf \circ \fneq f f = Pf \circ P\overline f = Pf \quad \text{and} \quad 
 Pg \circ \fneq f g = Pf \circ \fneq f g \leqslant Pf\rlap{ ;}
 \end{equation*}
 so by Lemma~\ref{lem:10}, there is a unique total map $\tau_{ij}$
 satisfying~\eqref{eq:9}. Next, the map $\eta_i = s_{1_i}$ satisfies
 $\overline{\eta_i} = \overline{s_{1_i}} = \fneq {1_i} {1_i} = 1_{Pi}$
 and so is total as desired. It remains to show well-definedness of
 $\mu_{ijk}$.
 By Lemma~\ref{lem:26}, the following pullback exists:
 \begin{equation*}
 \cd[@!C@C-5.6em@-0.6em]{
 & & \Psi P_{jk} \times_{Pj} \Psi P_{ij} \ar[dl]_-{\pi_2} \ar[dr]^-{\pi_1}
 {\save*!/d-1.2pc/d:(-1,1)@^{|-}\restore}\\
 & \Psi P_{ij} \ar[dl]_-{\sigma_{ij}} \ar[dr]^-{\tau_{ij}} & & \Psi P_{jk}
 \ar[dl]_-{\sigma_{jk}} \ar[dr]^-{\tau_{jk}} \\
 Pi & & Pj & & Pk\rlap{ ,}
 }
 \end{equation*}
 and the composite $s_{ij} \pi_2$ down the left is a local
 homeomorphism, with basis of partial sections
 $\{\tau_{ij}^\ast(s_g) s_f : f \in \A(i,j), g \in \A(j,k)\}$, and
 induced local atlas
 \begin{equation}
 \theta_{(f,g),(h,k)} = \fneq f h \overline{\fneq g k
 \tau_{ij}s_f} 
 = \fneq f h \overline{\fneq g k Pf}\rlap{ .}
\label{eq:10}
\end{equation}
Now consider the
family of maps $s_{gf} \colon Pi \rightarrow \Psi P_{ik}$ for
$f \in \A(i,j)$ and $g \in \A(j,k)$. By Lemma~\ref{lem:15}, these
satisfy
\begin{gather*}
 \theta_{(f,g),(f,g)} = \fneq f f \overline{\fneq g g 
 Pf} = P\overline{f} \,\overline{P\overline{g} Pf} =
 P\overline{gf} = \fneq {gf} {gf} = \overline{s_{gf}}\\
 \text{and } s_{kh} \theta_{(f,g),(h,k)} = s_{kh} \fneq f h \overline{\fneq g k
 Pf} \leqslant s_{kh} \fneq {gf}{kh} \leqslant s_{gf}\rlap{ ;}
 \end{gather*}
 so by Lemma~\ref{lem:10} there is a 
 total map $\mu_{ijk}$ uniquely determined by~\eqref{eq:12}.

 \textbf{$\eta$ and $\mu$ are compatible with source and target}. For
 compatibility of $\eta$, we note that
 $\sigma_{ii} \eta_i = \sigma_{ii}s_{1_i} = \overline{s_{1_i}} =
 1_{Pi}$ since $s_{1_i}$ is a partial section of $\sigma_{ii}$, and
 that $\tau_{ii}\eta_i = \tau_{ii} s_{1_i} = P(1_i) = 1_{Pi}$
 by~\eqref{eq:9}. For compatibility of $\mu$, we calculate that
 \begin{equation*}
 \sigma_{ij} \pi_2 \tau_{ij}^\ast(s_g) s_f = \theta_{(f,g),(f,g)}
 = \overline{s_{gf}} = \sigma_{ik} s_{gf} = \sigma_{ik} \mu_{ijk} \tau_{ij}^\ast(s_{g}) s_f 
 \end{equation*}
 for all $f \in \A(i,j)$ and $g \in \A(j,k)$. Since
 $\{\tau_{ij}^\ast(s_g)s_f\}$ is a basis for $\sigma_{ij}\pi_2$, we have
 $\sigma_{ij} \pi_2 = \sigma_{ik} \mu_{ijk}$ by Lemma~\ref{lem:27}.
 Similarly, $\tau_{ik}\mu_{ijk} = \tau_{jk}\pi_1$ follows since
 \begin{equation*}
 \tau_{jk}\pi_1 \tau_{ij}^\ast(s_g) s_f = \tau_{jk}s_g
 \tau_{ij}s_f = Pg \circ Pf = P(gf) = \tau_{ik}s_{gf} =
 \tau_{ik}\mu_{ijk}\tau_{ij}^\ast(s_g)s_f\rlap{ .}
 \end{equation*}

 \textbf{$\Phi P$ verifies the category axioms}. For this, it will be
 convenient to borrow the notation of Definition~\ref{def:5} and write
 $t \ast s$ for a composite of the form
 $\mu_{ijk} \circ \tau_{ij}^\ast(t) \circ s$. In these
 terms,~\eqref{eq:12} states that $s_g \ast s_f = s_{gf}$. Now, to
 verify the left unit law for $\Psi P$, we calculate that, for all
 $f \in \A(i,j)$ we have
 \begin{equation*}
 s_f = s_{1_j f} = s_{1_j} \ast s_f = \eta_j \ast s_f = \mu_{ijj}(\eta_j \tau_{ij},1)s_f 
 \end{equation*}
 where the last equality is~\eqref{eq:15}. It follows that
 $\mu_{ijj}(\eta_j \tau_{ij},1) = 1$ by Lemma~\ref{lem:10}. For the
 right unit law, we calculate similarly that, for all $f \in \A(i,j)$,
 we have
 \begin{equation*}
   s_f = s_{f1_i} = s_f \ast s_{1_i} = s_f \ast \eta_i = \mu_{iij}(1, \eta_i \sigma_{ij})s_f
 \end{equation*}
 where the last equality is~\eqref{eq:31}. Finally, for associativity,
 given $f \in \A(i,j)$, $g \in \A(j,k)$ and $h \in \A(k,\ell)$ we have
 that
 \begin{align*}
   s_{hgf} & = s_h \ast (s_g \ast s_f) = \mu_{ik\ell} \circ (\mu_{ijk} \times 1) \circ
    \mu_{ijk}^\ast \tau_{ik}^\ast(s_h) \circ \tau_{ij}^\ast(s_g)
    \circ s_f\\
    \text{and } s_{hgf} & = (s_h \ast s_g) \ast s_f = \mu_{ij\ell} \circ (1 \times \mu_{jk\ell}) \circ
    \mu_{ijk}^\ast \tau_{ik}^\ast(s_h) \circ \tau_{ij}^\ast(s_g) \circ
    s_f\rlap{ ,}
 \end{align*}
 where the last step in each line comes from~\eqref{eq:1},
 respectively~\eqref{eq:11}. The composition axiom follows since, by
 Lemma~\ref{lem:26}, the family of partial sections
 $\{\mu_{ijk}^\ast \tau_{ik}^\ast(s_h) \circ \tau_{ij}^\ast(s_g) \circ
 s_f\}$ constitute a basis for the local homeomorphism
 \begin{equation*}
 \Psi P_{k\ell}
 \times_{Pk} \Psi P_{jk} \times_{Pj} \Psi P_{ij} \xrightarrow{\pi_{23}}
 \Psi P_{jk} \times_{Pj} \Psi P_{ij} \xrightarrow{\pi_2}
 \Psi P_{ij} \xrightarrow{\sigma_{ij}} Pi\text{ .} 
\end{equation*}

\textbf{$\eta_P \colon \A \rightarrow \Phi(\Psi P)$ is a restriction
  functor, and~\eqref{eq:16} commutes}. Unitality of $\eta_P$ is the
fact that $s_{1_i} = \eta_i$; binary functoriality is precisely the
fact~\eqref{eq:12} that $s_g \ast s_f = s_{gf}$. Preservation of
restriction follows as
$s_{\bar f} = s_{1_i \bar f} = s_{1} \bar f = \eta_{1_i} \bar f =
\widetilde{s_f}$. Finally,~\eqref{eq:16} clearly commutes on
objects and commutes on maps by~\eqref{eq:9}.
\end{proof}

We are now in a position to prove:

\begin{Thm}
\label{thm:4}
  For each $P \colon \A \rightarrow \C$ in $\rcat$, the
  map~\eqref{eq:16} exhibits $\Psi P$ as the value at $P$ of a left
  adjoint to $\Phi \colon \partc \rightarrow \rcat$, yielding an
  adjunction
    \begin{equation}\label{eq:7}
  \cd{
    {\partc} \ar@<-4.5pt>[r]_-{\Phi} \ar@{<-}@<4.5pt>[r]^-{\Psi} \ar@{}[r]|-{\bot} &
    {\rcat}\rlap{ .} 
  }
\end{equation}
\end{Thm}
\begin{proof}
  Let $\mathbb{B}$ be a $J$-partite internal category in $\C$. We must
  show that every lax-commuting triangle as on the left in
  \begin{equation*}
    \cd[@!C@C-0.2em]{
      {\A} \ar[rr]^-{G} \ar[dr]_-{P}
      & \ltwocello{d}{\gamma} &
      {\Phi \mathbb{B}} \ar[dl]^-{\pi_\mathbb{B}} & \A
      \ar[r]^-{\eta_P} \ar[dr]_-{P} &
      {\Phi\Psi P} \ar[r]^-{\Phi F} \ar[d]_(0.4){\pi_{\Psi P}}
      \ltwocello[0.35]{dr}{\varpi^{F}} &
      {\Phi \mathbb{B}} \ar[dl]^-{\pi_\mathbb{B}} \\ &
      {\C} & & & \C & {}
    }
  \end{equation*}
  factors as to the right for a \emph{unique} partite cofunctor
  $F \colon \Psi P \rightsquigarrow \mathbb{B}$. We first show that the
  desired factorisation forces the definition of $F$. 
  
  \textbf{$F$ must act on component-sets by $i \mapsto Gi$}. This is
  as $\Phi F \circ \eta_P = G$ on objects.
  
  \textbf{$F$ must have action on objects given by
      $F_i = \gamma_i \colon B_{Gi} \rightarrow Pi$}. This follows
    since $\gamma = \varpi^{F} \circ \eta_P$ and $\eta_P$ is the
    identity on objects, so that $\gamma_i = \varpi^{F}_{i} = F_i$.

    \textbf{$F$'s action on arrows
      $F_{ij} \colon \Psi P_{ij} \times_{Pi} B_{Gi} \rightarrow
      B_{Gi,Gj}$ is uniquely determined by}
    \begin{equation}\label{eq:17}
      Gf = B_{Gi} \xrightarrow{\gamma_i^\ast(s_{f})} \Psi P_{ij} \times_{Pi} B_{Gi} \xrightarrow{F_{ij}} B_{Gi,Gj} \qquad \text{for all $f \in \A(i,j)$.}
    \end{equation}
    Indeed, the displayed equalities are forced since
    $\Phi F \circ \eta_P = G$ on morphisms; but by Lemma~\ref{lem:26},
    $\{\gamma_i^\ast(s_f) : f \in \A(i,j)\}$ is a basis of partial
    sections for the local homeomorphism
    $\pi_2 \colon \Psi P_{ij} \times_{Pi} B_{Gi} \rightarrow B_{Gi}$, so
    that there is at most one map $F_{ij}$ verifying these equalities.
    So the data of $F$ are uniquely determined; we
    now check that these data underlie a well-defined cofunctor.

  \textbf{The $F_{ij}$'s are well-defined}. 
  The local atlas $\psi$ for
  $\pi_2 \colon \Psi P_{ij} \times_{Pi} B_{Gi} \rightarrow B_{Gi}$
  associated to the basis of local sections $\{\gamma_i^\ast(s_f)\}$
  is, by Lemma~\ref{lem:26}, given by
  \begin{equation*}
    \psi_{f,g} = \overline{\fneq f g \gamma_i} \qquad \text{for
      $f,g \in \A(i,j)$.}
  \end{equation*}
  We must verify that the family of maps $(Gf : f \in \A(i,j))$
  satisfy the conditions of Lemma~\ref{lem:10} with respect to this
  local atlas. First, by naturality and totality of $\gamma$, and totality
  of $\tau_{ij}$, we have for any $f \in \A(i,j)$ that
  \begin{equation*}
    \overline{Pf \circ
      \gamma_i} = \overline{\gamma_j \circ \pi_\mathbb{B} (G(f))} =
    \overline{\gamma_j \circ \tau_{ij} \circ Gf} = \overline{Gf}
  \end{equation*}
  It follows by definition of $\fneq f g$ and distributivity of joins that
  \begin{align*}
    \psi_{fg} = \overline{\fneq f g\gamma_i} = \textstyle\bigvee \{\overline{Pe \circ \gamma_i}
    : e \leqslant \bar f \bar g, fe = ge\} = \bigvee \{\overline{Ge} : e \leqslant \bar f \bar g, fe = ge\}
  \end{align*}
  whence $\psi_{ff} = \overline{Gf}$ and
  \begin{align*}
    Gg\circ \psi_{fg} &= 
    \textstyle \bigvee \{Gg \circ \overline{Ge} : e \leqslant \bar f \bar g,
    fe = ge\}
    = \textstyle \bigvee \{Gg \ast Ge : e \leqslant \bar f \bar g,
    fe = ge\} \\ &= \textstyle\bigvee\{G(ge) : e \leqslant \bar f \bar g,
    fe = ge \} = \bigvee\{G(fe) : e \leqslant \bar f \bar g,
    fe = ge \} \leqslant Gf
  \end{align*}
  where on the first line, we have
  $Gg \circ \overline{Ge} = Gg \ast \tilde{Ge} = Gg \ast Ge$
  (composition in $\C$) just as in the proof of
  Proposition~\ref{prop:2}.

  \textbf{The maps $F_{ij}$ are compatible with source and target}.
  For the source axiom $\sigma_{Gi,Gj} F_{ij} = \pi_2$, we
  calculate that
  \begin{equation*}
    \sigma_{Gi,Gj} F_{ij} \gamma_i^\ast(s_f) =
    \sigma_{Gi,Gj} 
    Gf = \overline{Gf} = \psi_{ff} = \pi_2 \gamma_i^\ast(s_f)
  \end{equation*}
  and apply Lemma~\ref{lem:10}. For the target axiom
  $\tau_{ij}\pi_1 = \gamma_j \tau_{Gi,Gj} F_{ij}$, we have
  \begin{equation*}
    \tau_{ij}\pi_1\gamma_{i}^\ast(s_f) = \tau_{ij}s_f\gamma_i =
    Pf \, \gamma_i = \gamma_j \tau_{Gi,Gj}Gf = \gamma_j
    \tau_{Gi,Gj} \tilde
    G_{ij} \gamma_i^\ast(s_f)
  \end{equation*}
  whence the result, again by Lemma~\ref{lem:10}.

  \textbf{$F$ satisfies the unit and multiplication axioms}. For the unit axiom, we have
  \begin{equation*}
    \eta_{Gi} = G(1_i) = {F}_{ii} \gamma_i^\ast(s_{1_i}) =
    {F}_{ii}(\eta_i G_i, 1)
  \end{equation*}
  as required. Finally, for the composition axiom, let $f \in \A(i,j)$
  and $g \in \A(j,k)$. By exactly the same calculations as
  in~\eqref{eq:14} and~\eqref{eq:13}, we have that
  \begin{align*}
    Gg \ast Gf &= \mu_{Gi,Gj,Gk} \circ ({F}_{jk} \times 1) \circ (1 \times
    {F}_{ij}) \circ \pi_1^\ast\tau_{ij}^\ast (s_g) \circ
    \gamma_i^\ast (s_f)\\
    G(gf) &= {F}_{ik} \circ (1 \times \mu_{jk\ell}) \circ
    \pi_1^\ast \tau_{ij}^\ast(s_g) \circ \gamma_i^\ast(s_f)\rlap{ ,}
  \end{align*}
  and since $Gg \ast Gf = G(gf)$, the composites to the right are
  equal. But by Lemma~\ref{lem:26}, the partial sections
  $\{\pi_1^\ast \tau_{ij}^\ast(s_g) \gamma_i^\ast(s_f)\}$ are a basis
  for the local homeomorphism
  $\pi_3 \colon \Psi P_{jk} \times_{Pj} \Psi P_{ij} \times_{Pi} B_{Pi}
  \rightarrow B_{Pi}$, we conclude by Lemma~\ref{lem:10} that 
  $\mu_{Gi,Gj,Gk} \circ ({F}_{jk} \times 1) \circ (1 \times {F}_{ij})
  = {F}_{ik} \circ (1 \times \mu_{jk\ell})$ as required.
\end{proof}

\subsection{The fixpoints}
\label{sec:fixpoints-1}
We have thus succeeded in constructing the adjunction~\eqref{eq:7}; we
now identify its left and right fixpoints.

\begin{Prop}
  \label{prop:8}
  The unit~\eqref{eq:16} at $P \colon \A \rightarrow \C$ of the
  adjunction $\Psi \dashv \Phi$ is invertible if and only if $P$ lies
  in $\hcon$.
\end{Prop}
\begin{proof}
  If $\eta_P$ is invertible then $\A$ is in $\hcon$ since the
  isomorphic $\Phi\Psi P$ is so. Suppose conversely that $P$ is in
  $\hcon$; we must show that $\eta_P$ is invertible. Since it is
  already the identity on objects, we need only prove full fidelity.
  Looking at the action on maps of $\eta_P$, this means
  showing that each local section of
  $\sigma_{ij} \colon \Psi P_{ij} \rightarrow Pi$ is of the form $s_f$
  for a unique $f \in \A(i,j)$.

  As in Example~\ref{ex:4}, since $P$ is hyperconnected, the hom-set
  $\A(i,j)$ is a sheaf on $Pi \in \C$, and by Remark~\ref{rk:2}
  the equality $\sheq f g$ of this sheaf
  is the same as the $P$-valued equality $\fneq f g$. Thus, the local
  homemorphism $\sigma_{ij} \colon \Psi P_{ij} \rightarrow Pi$ is
  \emph{exactly} $\Delta(\A(i,j)) \rightarrow Pi$. So by
  Lemma~\ref{lem:6}, each partial section of $\sigma_{ij}$ is
  of the form $s_f$ for a unique $f \in \A(i,j)$ as desired.
\end{proof}

\begin{Prop}
  \label{prop:9}
  The counit
  $\varepsilon_\mathbb{A} \colon \Psi\Phi\mathbb{A} \rightarrow
  \mathbb{A}$ of the adjunction $\Psi \dashv \Phi$ is invertible at a
  partite internal category $\mathbb{A}$ if and only if $\mathbb{A}$
  is source-\'etale.
\end{Prop}
\begin{proof}
  The partite category $\Psi\Phi\mathbb{A}$ has the same indexing
  set $I$ as $\mathbb{A}$ and the same objects of objects
  $(A_i : i \in I)$. Since $\pi_\mathbb{A} \colon \Phi\mathbb{A} \rightarrow \C$ is
  hyperconnected, the discussion of the preceding proposition entails
  that the source projections
  $\Psi\Phi\mathbb{A}_{ij} \rightarrow A_i$ of
  $\Psi\Phi\mathbb{A}$ are obtained as the glueing of the sheaf
  $\Phi \mathbb{A}(i,j)$ on $A_i \in \C$. By inspection, this is
  precisely the sheaf $\Gamma(\sigma_{ij})$ of sections of
  $\sigma_{ij} \colon A_{ij} \rightarrow A_i$, so that
  $\Psi\Phi\mathbb{A}_{ij} \rightarrow A_i$ is \emph{exactly}
  $p_{\Gamma(\sigma_{ij})} \colon \Delta\Gamma(\sigma_{ij})
  \rightarrow A_i$.

  In these terms, the counit $\varepsilon_\mathbb{A} \colon
  \Psi\Phi\mathbb{A} \rightarrow \mathbb{A}$ is the partite
  cofunctor which is the identity on components, the identity on
  objects, and has action on arrows $\Psi\Phi\mathbb{A}_{ij}
  \rightarrow A_{ij}$ given by the counit maps
  \begin{equation*}
    \cd[@!C@C-2em@-1em]{
      {\Delta\Gamma(\sigma_{ij})} \ar[rr]^-{\varepsilon_{\sigma_{ij}}}
      \ar[dr]_-{p_{\Gamma(\sigma_{ij})}} & &
      {A_{ij}} \ar[dl]^-{\sigma_{ij}} \\ &
      {A_{i}}
    }
  \end{equation*}
  of the adjunction $\Delta \dashv \Gamma$ of~\eqref{eq:25}. Clearly
  $\varepsilon_\mathbb{A}$ is invertible if and only if each
  $\varepsilon_{\sigma_{ij}}$ as displayed is so; which, by
  Proposition~\ref{prop:19}, happens just when each $\sigma_{ij}$ is a
  local homeomorphism, as desired.
\end{proof}
From this our main theorem immediately follows:
\begin{Thm}
\label{thm:2}
Restricting~\eqref{eq:7} to its fixpoints yields an
equivalence
  \begin{equation}\label{eq:33}
    \cd{
      {\parte} \ar@<-4.5pt>[r]_-{\Phi} \ar@{<-}@<4.5pt>[r]^-{\Psi} \ar@{}[r]|-{\sim} &
      {\hcon} 
    }
  \end{equation}
  between the category of source-\'etale partite internal categories
  in $\C$, and the category of join restriction categories
  hyperconnected over $\C$.
\end{Thm}

Since each $\Psi P$ is source-\'etale, hence a left fixpoint, we also
obtain:

\begin{Cor}
  \label{cor:1}
  The adjunction~\eqref{eq:7} is Galois. In particular,
  $\hcon$ is reflective in $\rcat$ via $\Phi \Psi$, while
  $\parte$ is coreflective in $\partc$ via $\Psi \Phi$.
\end{Cor}

By analogy with Section~\ref{sec:fixpoints}, the reflector
$\Phi \Psi \colon \rcat \rightarrow \hcon$ can be seen as
``sheafification'', and much as there, we can describe it explicitly.

\begin{Prop}
\label{prop:26}
  For any $P \colon \A \rightarrow \C$ in $\rcat$, its reflection
  $\pi_{\Psi P} \colon \Phi \Psi P  \rightarrow \C$ into $\hcon$ can be
  described as follows.
  \begin{itemize}
  \item \textbf{Objects} of $\Phi\Psi P$ are those of $\A$;
  \item \textbf{Morphisms} $\theta \in \Phi\Psi P(i,j)$ are
    families $\bigl(\theta_{f} \in \O(Pi) : f \in \A(i,j)\bigr)$ such
    that
    \begin{equation}\label{eq:42}
      \theta_g \theta_f \leqslant \fneq f g \quad \text{and} \quad
      \theta_g \fneq f g \leqslant \theta_f
      \qquad \text{for all }f,g \in \A(i,j)\rlap{ ;}
    \end{equation}
  \item The \textbf{identity map} $1_{i} \in \Phi\Psi P(i,i)$ is given by the family
    $(1_i)_f = \fneq {1_i} f$;
  \item \textbf{Composition} of $\theta \in \Phi\Psi P(i,j)$ and
    $\psi \in \Phi\Psi P(j,k)$ is given by the family
    \begin{equation*}
      (\psi \circ \theta)_{h} = \bigvee_{\substack{f \in \A(i,j)\\ g \in \A(j,k)}} \fneq {gf}
      h \,\overline{\psi_g Pf} \,\theta_f\rlap{ ;}
    \end{equation*}
  \item \textbf{Restriction} of $\theta \in \Phi\Psi P(i,j)$ is the family
    $\overline{\theta}_f = \fneq {1_i} f \bigvee_g \theta_g$;
  \item The \textbf{projection} functor
    $\pi_{\Psi P} \colon \Phi\Psi P \rightarrow \C$ acts as $P$ does
    on objects, and on morphisms by sending $\theta \in \Phi\Psi P(i,j)$ to
    $\bigvee_{f \in \A(i,j)} Pf \circ \theta_f$;
  \item The \textbf{unit} functor $\eta_P \colon \A \rightarrow \Phi\Psi P$ is
    the identity on objects, and on morphisms sends $g \in \A(i,j)$ to
    $\bigl(\fneq f g : f \in \A(i,j)\bigr)$.
  \end{itemize}
\end{Prop}
\begin{proof}
  Objects of $\Phi \Psi P$ are clearly those of $\A$, while morphisms
  $i \rightarrow j$ are partial sections of
  $\sigma_{ij} \colon \Psi P_{ij} \rightarrow Pi$. This map has basis
  of local sections $\bigl(s_f : f \in \A(i,j)\bigr)$ and associated
  local atlas $\bigl(\fneq f g : f,g \in \A(i,j)\bigr)$; so by
  Lemma~\ref{lem:1}, its partial sections $s$ correspond to families
  $(\theta_f)$ as in~\eqref{eq:42} via the assignations
  \begin{equation*}
    s \mapsto \bigl(p_f s : f \in \A(i,j)\bigr) \qquad \text{and}
    \qquad
    \bigl(\theta_f : f \in \A(i,j)\bigr) \mapsto \textstyle\bigvee_f
    s_f \theta_f\rlap{ .}
  \end{equation*}

  Under this bijection, the identity section
  $s_{1_i} \in \Phi\Psi P(i,i)$ corresponds to the family with
  components $p_f s_{1_i} = \fneq {1_i} f$, as desired.

  For composition, if $\theta \in \Phi\Psi P(i,j)$ and
  $\psi \in \Phi\Psi P(j,k)$, then their composite is the family with
  component at $h \in \A(i,k)$ given by
  \begin{equation*}
    p_h \Bigl((\textstyle \bigvee_g s_g \psi_g) \ast (\bigvee_f
    s_f \theta_f)\Bigr) =
    \textstyle \bigvee_{f,g} p_h \bigl((s_g \psi_g) \ast (s_f \theta_f)\bigr)\rlap{ .} 
  \end{equation*}
  Now clearly
  $(s_g \psi_g) \ast (s_f \theta_f) \leqslant s_g \ast s_f = s_{gf}$,
  and its restriction satisfies
  \begin{align*}
    \pi_{\Psi P}\bigl(\overline{(s_g \psi_g) \ast (s_f \theta_f)}\bigr) =
    \overline{\pi_{\Psi P}(s_g \psi_g)\pi_{\Psi P}(s_f \theta_f)}
    =
    \overline{Pg \psi_g Pf \theta_f} = \overline{\psi_g
      Pf}\theta_f\end{align*}
  so that we have $(s_g \psi_g) \ast (s_f \theta_f) = s_{gf}\overline{\psi_g
      Pf}\theta_f$ and thus $(\psi \circ \theta)_h$ given by
  \begin{equation*}
    \textstyle \bigvee_{f,g} p_h \bigl((s_g
    \psi_g) \ast (s_f \theta_f)\bigr) = \textstyle \bigvee_{f,g} p_h
    s_{gf} \overline{\psi_g
      Pf}\theta_f = \textstyle \bigvee_{f,g} \fneq{gf} h \overline{\psi_g
      Pf}\theta_f
  \end{equation*}
  as required. Further, the restriction of
  $\theta \in \Phi\Psi P(i,j)$ is the family
  \begin{equation*}
    \overline{\theta}_f = \textstyle p_f \eta_i \overline{ \bigvee_g s_g \theta_g} = p_f
    s_{1_i} \bigvee_g \theta_g = \fneq {1_i} f \bigvee_g \theta_g\rlap{ .}
  \end{equation*}

  This proves that $\Phi \Psi P$ is as described. Next consider
  $\pi_{\Psi P} \colon \Phi \Psi P \rightarrow \C$. This acts as $P$
  does on objects, while on maps it sends 
  $\theta \in \Phi \Psi P(i,j)$, corresponding to the partial section
  $\bigvee_f s_f \theta_f$, to $\tau_{ij}\bigvee_f s_f \theta_f =
  \bigvee_f Pf \circ \theta_f$ as desired.

  Finally, consider $\eta_P \colon \A \rightarrow \Phi \Psi P$; this
  is the identity on objects, and sends $g \in \A(i,j)$ to the partial
  section $s_g$, corresponding to the family of idempotents $\bigl(p_f s_g : f\in
  \A(i,j)\bigr) = \bigl(\fneq f g : f\in
  \A(i,j)\bigr)$ as claimed.
\end{proof}

\section{The groupoid case}
\label{sec:groupoid-case}
As explained in the introduction, our main result generalises along
four different axes the correspondence between \'etale topological
groupoids and pseudogroups. We now begin to examine the effect of
rolling back these generalisations. Once again, $\C$ will be any join
restriction category with local glueings.

It is trivial to see that~\eqref{eq:33} restricts back to an
equivalence between source-\'etale ($1$-partite) internal categories
in $\C$ and join restriction monoids hyperconnected over $\C$. More
interesting are the adaptations required for the groupoidal case; the
goal of this section is to describe these.

\subsection{Groupoids and \'etale join restriction categories}
\label{sec:groupoids-etale-join}
To one side of our restricted adjunction and equivalence will be the
partite internal groupoids:
\begin{Defn}
  \label{def:31}
  An \emph{$I$-partite internal groupoid} in the join restriction
  category $\C$ is an $I$-partite internal category $\mathbb{A}$
  endowed with total maps
  $\iota_{ij} \colon A_{ij} \rightarrow A_{ji}$ which are involutions
  in the sense that $\iota_{ij}\iota_{ji} = 1$, are compatible with source and
  target in the sense that $\sigma_{ji} \iota_{ij} = \tau_{ij}$ (and
  so also $\tau_{ji}\iota_{ij} = \sigma_{ij}$), and which render
  commutative each square to the left in:
  \begin{equation*}
    \cd{
      A_{ij} \ar[d]_-{\sigma_{ij}} \ar[r]^-{(\iota_{ij}, 1)} & A_{ji}
      \times_{A_j} A_{ij} \ar[d]^-{\mu_{iji}} &
            A_{ij} \ar[d]_-{\tau_{ij}} \ar[r]^-{(1,\iota_{ij})} & A_{ij}
      \times_{A_i} A_{ji} \ar[d]^-{\mu_{jij}} \\
      A_i \ar[r]^-{\eta_i} & A_{ii} &
      A_j \ar[r]^-{\eta_j} & A_{jj} 
    }
  \end{equation*}
  (and so also each square to the right). The usual argument adapts to
  show that the inverse maps $\iota_{ij}$ are unique, if they exist;
  and this justifies us in writing $\mathrm{p}\cat{Gpd}_c(\C)$ and
  $\mathrm{pe}\cat{Gpd}_c(\C)$ for the full subcategories of $\partc$
  and $\parte$ on the partite internal groupoids.
\end{Defn}

\begin{Rk}
\label{rk:7}
  Since a source-\'etale partite internal groupoid satisfies
  $\sigma_{ij}\iota_{ij} = \tau_{ij}$, its target maps, as well as its
  source maps, are local homeomorphisms; and in fact, the groupoid
  axioms together with Lemma~\ref{lem:34} ensure that \emph{all} of
  the groupoid structure maps are local homeomorphisms. Thus, we may
  say ``\'etale'' rather than ``source-\'etale'' when dealing with
  partite internal groupoids.
\end{Rk}

The entities to the other side of the equivalence can described in two
different ways: either as the join inverse categories hyperconnected
over $\C$, or as the \'etale join restriction categories
hyperconnected over $\C$ (recall from Definition~\ref{def:24} that a
join restriction category is \emph{\'etale} if every map therein is a join of
partial isomorphisms). While the two formulations are equivalent by
Corollary~\ref{cor:6}, each has its own advantages, and so we will
discuss both. We first consider the one involving \'etale join
restriction categories; this has the advantage of being a genuine
special case of~\eqref{eq:33}.
\begin{Thm}
\label{thm:3}
The equivalence~\eqref{eq:33} restricts back to an equivalence
  \begin{equation}\label{eq:19}
    \cd{
      {\mathrm{pe}\cat{Gpd}_c(\C)} \ar@<-4.5pt>[r]_-{\Phi} \ar@{<-}@<4.5pt>[r]^-{\Psi} \ar@{}[r]|-{\sim} &
      {\hcon[ej]}
    }
  \end{equation}
  between the category of \'etale partite internal groupoids in
  $\C$ and the category of \'etale join restriction categories
  hyperconnected over $\C$.
\end{Thm}
Before proving this, we need a couple of preliminary lemmas. The first
is a purely technical result about equality of partial isomorphisms;
the second is rightly a part of the theorem, but we separate it out
for later re-use.
\begin{Lemma}
  \label{lem:36}
  If $f, g \colon A \rightarrow B$ are partial isomorphisms in a join
  restriction category then $f\sheq f g = 
  \sheq {f^\ast}{g^\ast} f$.
\end{Lemma}
\begin{proof}
  Recall from Lemma~\ref{lem:30} that $f \sheq f g$ is the meet
  $f \wedge g$ in $\C(A,B)$, so it suffices to prove
  $\sheq {f^\ast}{g^\ast} f$ is too. Clearly it is below $f$, while
  from $g^\ast \sheq {f^\ast} {g^\ast} = f^\ast \sheq{f^\ast}{g^\ast}$
  we conclude by taking partial inverses that
  $\sheq {f^\ast} {g^\ast} f = \sheq {f^\ast} {g^\ast} g \leqslant g$.

  Suppose now that $h \in \C(A,B)$ is below $f$ and $g$. Such an $h$
  is itself a partial isomorphism with $h^\ast = \overline{h}f^\ast =
  \overline{h}g^\ast$. It follows that $h^\ast \leqslant f^\ast \wedge
  g^\ast$, i.e., $\overline{h^\ast} \leqslant \sheq {f^\ast}
  {g^\ast}$, and so that $h = h \overline{f} = hf^\ast f = h
  \overline{h}f^\ast f = hh^\ast f = \overline{h^\ast}f \leqslant
  \sheq{f^\ast}{g^\ast}f$ as desired.
\end{proof}

\begin{Lemma}
  \label{lem:35}
  If $\mathbb{A}$ is a partite internal groupoid in $\C$, then
  $\pi_\mathbb{A} \colon \Phi \mathbb{A} \rightarrow \C$ reflects
  partial isomorphisms and \'etale maps.
\end{Lemma}
\begin{proof}
  Let $f \in \Phi \mathbb{A}(i,j)$ be such that
  $\pi_\mathbb{A}(f) = \tau_{ij}f \colon A_i \rightarrow A_j$ is a
  partial isomorphism in $\C$. Consider the composite
  \begin{equation*}
    g \defeq A_j \xrightarrow{(\tau_{ij}f)^\ast} X_i \xrightarrow{f} A_{ij}
    \xrightarrow{\iota_{ij}} A_{ji}
  \end{equation*}
  in $\C$. Note that
  $\sigma_{ji} g = \sigma_{ji} \iota_{ij} f(\tau_{ij}f)^\ast =
  \tau_{ij}f(\tau_{ij}f)^\ast = \overline{(\tau_{ij}f)^\ast}\leqslant
  1$ so that $g$ is a partial section of $\sigma_{ji}$ with
  $\overline{g} = \overline{(\tau_{ij}f)^\ast}$. We claim that
  $g \in \Phi \mathbb{A}(j,i)$ is the desired partial inverse of
  $f \in \Phi\mathbb{A}(i,j)$. First we calculate that
  \begin{align*}
    g\tau_{ij}f &= \iota_{ij} f(\tau_{ij}f)^\ast\tau_{ij}f = \iota_{ij} f
    \overline{\tau_{ij}f} = \iota_{ij} f
    \overline{f} = \iota_{ij} f\\
    \text{and } f \tau_{ji}g &= f \tau_{ji} \iota_{ij} f(\tau_{ij}f)^\ast = f
    \sigma_{ij} f(\tau_{ij}f)^\ast = f \overline{f} (\tau_{ij}f)^\ast = f(\tau_{ij}f)^\ast
  \end{align*}
  with respective restrictions
  \begin{align*}
    \overline{g\tau_{ij}f} = \overline{\iota_{ij}f} = \overline{f} \quad 
    \text{and}\quad \overline{f\tau_{ji}g} = \overline{f(\tau_{ij}f)^\ast}
    = \overline{\tau_{ij}f(\tau_{ij}f)^\ast} =
    \overline{(\tau_{ij}f)^\ast} = \overline{g}
  \end{align*}
  from which we obtain the desired equalities
  \begin{align*}
    f \ast g &= \mu_{jij}\tau_{ji}^\ast(f)g = \mu_{jij}(f\tau_{ji}g,
    g\overline{f\tau_{ji}g}) = \mu_{jij}(f(\tau_{ij}f)^\ast,
    g) \\ &= \mu_{jij}(1, \iota_{ij})f(\tau_{ij}f)^\ast = \eta_j
    \tau_{ij}f(\tau_{ij}f)^\ast = \eta_j \overline{g} = \tilde g\\
    \text{and } g \ast f &= \mu_{iji}\tau_{ij}^\ast(g)f = \mu_{iji}(g\tau_{ij}f,
    f\overline{g\tau_{ij}f}) = \mu_{iji}(\iota_{ij}f, f) \\&=
    \mu_{iji}(\iota_{ij},1)f = \eta_i \sigma_{ij}f = \eta_i \overline{f}
    = \tilde f\text{ .}
  \end{align*}

  So $\pi_\mathbb{A}$ reflects partial isomorphisms as claimed. Since it
  is also hyperconnected, it follows that it reflects \'etale maps.
  Indeed, if $\pi_\mathbb{A}(f \colon i \rightarrow j)$ is a join of
  partial isomorphisms $\bigvee_k p_k$, then taking $e_k \in \O(i)$ to
  be unique such that $\pi_\mathbb{A}(e_k) = \overline{p_k}$, we have
  $\pi_\mathbb{A}(fe_k) = p_k$ for each $i$. Since $\pi_\mathbb{A}$
  reflects partial isomorphisms, each $fe_k$ is a partial isomorphism,
  and so $f = \bigvee_k fe_k$ is \'etale.
\end{proof}

We are now ready to give:
\begin{proof}[Proof of Theorem~\ref{thm:3}]
  Suppose first that $\mathbb{A}$ is an \'etale partite internal
  groupoid in $\C$; we will show that every map in $\Phi \mathbb{A}$
  is \'etale. Since each $\iota_{ij}$ is invertible and each
  $\sigma_{ij}$ is a local homeomorphism, each
  $\tau_{ij} = \sigma_{ji}\iota_{ij}$ is a local homeomorphism. Now
  any $s \in \Phi\mathbb{A}(i,j)$ is a partial section of
  $\sigma_{ij}$, and hence a partial isomorphism, whence for each
  $s \in \Phi\mathbb{A}(i,j)$, the map
  $\pi_{\mathbb{A}}(s) = \tau_{ij}s$ is \'etale. Since by
  Lemma~\ref{lem:35}, $\pi_{\mathbb{A}}$ reflects \'etale maps, $s$ is
  itself \'etale, as desired.

  Suppose conversely that $P \colon \A \rightarrow \C$ is an \'etale
  join restriction category hyperconnected over $\C$; we show that
  $\Psi P$ is a partite internal groupoid. As in
  Proposition~\ref{prop:8}, each partial section of the source map
  $\sigma_{ij} \colon \Psi P_{ij} \rightarrow Pi$ is of the form $s_f$
  for some $f \in \A(i,j)$; but since $\A$ is \'etale, we can write
  each such $s_f$ as $\bigvee_i s_{f_i}$ where each $f_i$ is a partial
  isomorphism. Thus by Lemma~\ref{lem:27} the family of sections
  $\{s_f : f \in \A(i,j)\text{ a partial isomorphism}\}$ is a basis
  for $\sigma_{ij}$. We can thus determine $\iota_{ij}$ by asking it
  to be the unique total map $\Psi P_{ij} \rightarrow \Psi P_{ji}$
  such that
  \begin{equation*}
    \iota_{ij} \circ s_f = s_{f^\ast} \circ f \qquad \text{for all $f
      \in \A(i,j)$ a partial isomorphism.}
  \end{equation*}
  For this to be well-defined, we must check compatibility of the
  family of maps $s_{f^\ast} f \colon Pi \rightarrow \Psi P_{ji}$ with
  the associated local atlas for $\sigma_{ij}$; by
  Proposition~\ref{prop:8} again, this atlas is given by
  $\varphi_{fg} = \sheq f g$, and so we have
  \begin{align*}
    s_{g^\ast}g \sheq f g = s_{g^\ast}f \sheq f g &= s_{g^\ast} \sheq
    {f^\ast} {g^\ast} f = s_{g^\ast \sheq {f^\ast} {g^\ast}} f \leqslant s_{f^\ast} f \\
    \text{and } \overline{s_{f^\ast} f} &= \overline{\overline{s_{f^\ast}} f} =
    \overline{\overline{f^\ast}f} = \overline{f} = \sheq f f
  \end{align*}
  as required. So $\iota_{ij}$ is well-defined; it remains to check the
  groupoid axioms.

  First, 
  $\iota_{ji}\iota_{ij} s_f = \iota_{ji}s_{f^\ast}f = s_{f}f^\ast f =
  s_f \overline{f} = s_f$ for each partial isomorphism $f$, so that
  $\iota_{ji} \iota_{ij} = 1$ by Lemma~\ref{lem:27}. Similarly, we
  have that
  $\sigma_{ij}\iota_{ij}s_f = \sigma_{ij}s_{f^\ast}f =
  \overline{f^\ast}f = ff^\ast f = f = \tau_{ij}s_f$, so that
  $\sigma_{ij}\iota_{ij} = \tau_{ij}$. Finally, we have
  \begin{align*}
    \mu_{iji}(\iota_{ij},1)s_f &= \mu_{iji}(s_{f^\ast}f,s_f\overline{f}) =
    \mu_{iji}(s_{f^\ast}\tau_{ij}s_f,s_f\overline{s_{f^\ast}\tau_{ij}s_f})
    =
    \mu_{iji}(s_{f^\ast}\tau_{ij}, \overline{s_{f^\ast}\tau_{ij}})s_f
    \\&= \mu_{iji}\tau_{ij}^\ast(s_{f^\ast})s_f = s_{f^\ast} \ast s_f =
    s_{f^\ast f} = s_{\overline{f}} = \eta_i \overline{f} = \eta_i \sigma_{ij} s_f
  \end{align*}
  for each partial isomorphism $f$, so that
  $\mu_{iji}(\iota_{ij},1) = \eta_i \sigma_{ij}$, as required.
\end{proof}


\subsection{Groupoids and join inverse categories}
\label{sec:group-join-inverse}
We now turn to the formulation of the equivalence~\eqref{eq:19}
involving inverse categories. This account has two advantages: it
matches up more closely with the classical \'etale
groupoid--pseudogroup correspondence, and it allows us to give a
groupoid version not only of the equivalence~\eqref{eq:33}, but also
of the adjunction~\eqref{eq:7}. The disadvantage is that we have to
alter the adjunction slightly.

\begin{Defn}
  \label{def:36}
  Let $\mathbb{A}$ be a partite internal groupoid in $\C$. A 
  \emph{partial bisection} of the source map
  $\sigma_{ij} \colon A_{ij} \rightarrow A_i$ is a partial section
  $s \colon A_i \rightarrow A_{ij}$ for which the composite
  $\tau_{ij}s$ is a partial isomorphism. We write $\Phi_g \mathbb{A}$
  for the subcategory of $\Phi \mathbb{A}$ with the same objects, but
  only those maps $s \in \Phi\mathbb{A}(i,j)$ which are partial bisections,
  and reuse the notation
  $\pi_\mathbb{A} \colon \Phi_g \mathbb{A} \rightarrow \C$ for the
  functor sending $s$ to $\tau_{ij}s$.
\end{Defn}

\begin{Prop}
  \label{prop:6}
  The assignation
  $\mathbb{A} \mapsto (\pi_\mathbb{A} \colon \Phi_g \mathbb{A}
  \rightarrow \C)$ is the action on objects of a well-defined functor
  $\Phi_g \colon \mathrm{p}\cat{Gpd}_c(\C) \rightarrow \rcat[i]$ taking
  values in $\hcon[ji]$.
\end{Prop}
\begin{proof}
  Note that $\Phi_g \mathbb{A}$ comprises just those maps of
  $\Phi \mathbb{A}$ which
  $\pi_\mathbb{A} \colon \Phi \mathbb{A} \rightarrow \C$ sends to
  partial isomorphisms. Since by by Lemma~\ref{lem:35},
  $\pi_\mathbb{A}$ reflects partial isomorphisms, we conclude that
  $\Phi_g\mathbb{A} = \mathrm{PIso}(\Phi\mathbb{A})$. We can thus
  obtain $\Phi_g$ as
  \begin{equation}\label{eq:2}
    \Phi_g = \mathrm{p}\cat{Gpd}_c(\C) \xrightarrow{\Phi} \rcat
    \xrightarrow{\mathrm{PIso}/\!/\C} \rcat[i]
  \end{equation}
  where the second component is the functor which acts on objects by
  restricting back $P \colon \A \rightarrow \C$ along the inclusion
  $\mathrm{PIso}(\A) \subseteq \A$. By Proposition~\ref{prop:16} and
  the fact that $\mathrm{PIso}(\A) \subseteq \A$ is a hyperconnected
  inclusion, $\mathrm{PIso}/\!/\C$ maps $\hcon$ into $\hcon[ji]$; it
  follows that the image of $\Phi_g$ must land in $\hcon[ji]$.
\end{proof}

We now use this result to exhibit an adjunction
whose fixpoints will provide the desired equivalence
$\mathrm{pe}\cat{Gpd}_c(\C) \simeq \hcon[ji]$.

\begin{Thm}
  \label{thm:6}
  The functor $\Psi \colon \rcat \rightarrow \partc$
  maps $\rcat[i]$ into $\mathrm{p}\cat{Gpd}_c(\C)$, and the restricted
  functor induced in this way provides a left adjoint to $\Phi_g$ in
  \begin{equation}\label{eq:34}
    \cd{
      {\mathrm{p}\cat{Gpd}_c(\C)} \ar@<-4.5pt>[r]_-{\Phi_g} \ar@{<-}@<4.5pt>[r]^-{\Psi} \ar@{}[r]|-{\bot} &
      {\rcat[i]} \rlap{ .}
    }
  \end{equation}
\end{Thm}
\begin{proof}
  Let $P \colon \A \rightarrow \C$ be an object of $\rcat[i]$ and
  consider $\Psi P \in \mathrm{r}\cat{Cat}(\C)$. Each
  $\sigma_{ij} \colon \Psi P_{ij} \rightarrow Pi$ has a basis
  $\{s_f : f \in \A(i,j) \text{ a partial isomorphism}\}$ since
  \emph{every} map in $\A$ is a partial isomorphism. So the same proof
  as in Theorem~\ref{thm:3} shows that $\Psi P$ is a partite internal
  groupoid. For the final claim, consider the adjunction
  \begin{equation*}
    \cd{
      {\partc} \ar@<-4.5pt>[r]_-{\Phi} \ar@{<-}@<4.5pt>[r]^-{\Psi} \ar@{}[r]|-{\bot} &
      {\rcat} \ar@<-4.5pt>[r]_-{\mathrm{PIso}/\!/\C} \ar@{<-}@<4.5pt>[r]^-{\text{include}} \ar@{}[r]|-{\bot} &
      {\rcat[i]\rlap{ .}}
    }
  \end{equation*}
  We just showed that the left adjoint lands in
  $\mathrm{p}\cat{Gpd}_c(\C)$, and by~\eqref{eq:2}, the right adjoint
  agrees with $\Phi_g$ on $\mathrm{p}\cat{Gpd}_c(\C)$.
  This yields the desired adjunction~\eqref{eq:34}.
\end{proof}
And we now deduce:
\begin{Thm}
  \label{thm:5}
  Restricting~\eqref{eq:34} to its fixpoints yields an equivalence
  \begin{equation}\label{eq:35}
    \cd{
      {\mathrm{pe}\cat{Gpd}_c(\C)} \ar@<-4.5pt>[r]_-{\Phi_g} \ar@{<-}@<4.5pt>[r]^-{\Psi} \ar@{}[r]|-{\sim} &
      {\hcon[ji]} 
    }
  \end{equation}
  between the category of \'etale partite internal groupoids in
  $\C$, and the category of join inverse categories hyperconnected
  over $\C$.
\end{Thm}
\begin{proof}
  The left and right fixpoints are contained in the respective
  essential images of $\Psi$ and $\Phi_g$, which are in turn contained
  in $\mathrm{pe}\cat{Gpd}_c$ and $\hcon[ji]$ respectively. So we have
  a restricted adjunction as in~\eqref{eq:35}, and it suffices to show
  this is an equivalence. But the right adjoint $\Phi_g$ is equally
  the composite of the equivalence
  $\Phi \colon \mathrm{pe}\cat{Gpd}_c(\C) \rightarrow \hcon[ej]$ of
  Theorem~\ref{thm:3} with the equivalence
  $\mathrm{PIso}/\!/_h \C \colon \hcon[ej] \rightarrow \hcon[ji]$
  obtained from Corollary~\ref{cor:6}.
\end{proof}


\begin{Cor}
\label{cor:3}
The adjunction~\eqref{eq:34} is Galois. In particular, $\hcon[ji]$ is
reflective in $\rcat[i]$ via $\Phi_g \Psi$, while
$\mathrm{pe}\cat{Gpd}_c(\C)$ is coreflective in
$\mathrm{p}\cat{Gpd}_c(\C)$ via $\Psi \Phi_g$.
\end{Cor}

Like in Section~\ref{sec:fixpoints-1}, we can describe
``sheafification'' $\Phi_g \Psi \colon \rcat[i] \rightarrow \hcon[ji]$
explicitly. As in Proposition~\ref{prop:6}, we have
$\Phi_g \Psi P = \mathrm{PIso}(\Phi \Psi P)$, and so our description
of the sheafification will follow from:
\begin{Lemma}
  \label{lem:7}
  Let $P \colon \A \rightarrow \C$ in $\rcat[i]$. A map
  $\theta \in \Phi \Psi P(i,j)$, given by a family
  $\bigl(\theta_f \in \O(Pi) : f \in \A(i,j)\bigr)$ satisfying~\eqref{eq:42}, is
  a partial isomorphism just when
  \begin{equation}\label{eq:46}
    (Pf \circ \theta_f \circ Pf^\ast) (Pg \circ \theta_g \circ
    Pg^\ast) \leqslant \fneq{f^\ast}{g^\ast} \qquad \text{for all $f,g
      \in \A(i,j)$.}
  \end{equation}
\end{Lemma}
\begin{proof}
  Since $\Psi P$ is an internal groupoid by Theorem~\ref{thm:6},
  $\pi_{\Psi P} \colon \Phi \Psi P \rightarrow \C$ reflects partial
  isomorphisms by Lemma~\ref{lem:35}. So $\theta \in \Phi \Psi P(i,j)$
  is a partial isomorphism just when
  $\pi_{\Psi P}(\theta) = \bigvee_f Pf \circ \theta_f$ is so. In turn,
  this will be so just when the compatible family of partial
  isomorphisms $\bigl(Pf \circ \theta_f : f \in \A(i,j)\bigr)$ is
  bicompatible. This is true just when, for all $f,g \in \A(i,j)$, we
  have
  $\theta_f \circ Pf^\ast \circ \overline{\theta_g \circ Pg^\ast}
  \leqslant Pg^\ast$, or, equivalently, when
  \begin{equation*}
    \theta_f \circ Pf^\ast \circ Pg \circ \theta_g \circ Pg^\ast \leqslant Pg^\ast\rlap{ .}
  \end{equation*}
  But
  $\theta_f \circ Pf^\ast = Pf^\ast \circ \overline{\theta_f \circ
    Pf^\ast} = Pf^\ast \circ Pf \circ \theta_f \circ Pf^\ast$, so this
  happens just when
  \begin{equation*}
    Pf^\ast (Pf \circ \theta_f \circ Pf^\ast) (Pg \circ \theta_g \circ
    Pg^\ast) \leqslant Pg^\ast\rlap{ ,}
  \end{equation*}
  i.e., precisely when the condition~\eqref{eq:46} holds, as claimed.
\end{proof}

It follows that the reflector
$\Phi_g \Psi \colon \rcat[i] \rightarrow \hcon[ji]$ has exactly the
same description as the reflector
$\Phi \Psi \colon \rcat \rightarrow \hcon$ in
Proposition~\ref{prop:26}, except that maps of $\Phi_g \Psi P$ are
families satisfying not only~\eqref{eq:42}, but also~\eqref{eq:46}.

\subsection{The one-object case}
\label{sec:one-object-case}

By combining the groupoidal and the one-object specialisations of our
main result, we obtain an equivalence between \'etale internal
groupoids in $\C$, and join restriction monoids hyperconnected over
$\C$. Since this is the form of our theorem which most resembles the
classical situation, we spell it out in detail. In the following
definition, as in the introduction, we write $\I(X)$ for the join
inverse monoid of partial automorphisms of some $X \in \C$.

\begin{Defn}
  \label{def:37}
  The category $\cat{Ps\G rp}(\C)$ has:
  \begin{itemize}[itemsep=0.2\baselineskip]
  \item As objects, \emph{complete $\C$-pseudogroups} $(S,X,\theta)$,
    comprising a join inverse monoid $S$, an object $X \in \C$ and a
    monoid homomorphism $\theta \colon S \rightarrow \I(X)$ which
    restricts to an isomorphism between the idempotents of $S$ and the
    idempotents of $\I(X)$.
\item As maps
  $(S,X,\theta) \rightarrow (T,Y,\gamma)$, pairs $(f,g)$ of a monoid
  homomorphism $f \colon S \rightarrow T$ and a total map
  $g \colon Y \rightarrow X$ in $\C$ such that, for all $s \in S$ with
  $f(s) = t$, we have $\theta(s)g = g \gamma(t)$.
\end{itemize}
\end{Defn}
\begin{Thm}
  \label{thm:7}
  The equivalence~\eqref{eq:35} restricts to an equivalence
  \begin{equation}\label{eq:36}
    \cd{
      {\mathrm{e}\cat{Gpd}_c(\C)} \ar@<-4.5pt>[r]_-{\Phi_g} \ar@{<-}@<4.5pt>[r]^-{\Psi} \ar@{}[r]|-{\sim} &
      {\cat{Ps\G rp}(\C)} 
    }
  \end{equation}
  between the category of \'etale internal groupoids and
  internal cofunctors in $\C$, and the category of complete
  $\C$-pseudogroups and their homomorphisms.
\end{Thm}

Even in the case where $\C = \cat{Top}_p$, this result goes beyond the
ones in the literature due to the extended functoriality of our
correspondence. We now explain how this narrower functoriality may be
re-found.

\section{Localic and hyperconnected maps}
\label{sec:local-hyperc-maps}

\subsection{Covering functors and localic morphisms}
\label{sec:cover-funct-local}

In~\cite{Lawson2013Pseudogroups}, the morphisms considered between
\'etale topological groupoids are not cofunctors, but the
so-called \emph{covering functors}. We now describe the analogue of
this notion in our context.

\begin{Defn}
  \label{def:43}
  Let $\mathbb{B}$ be an $I$-partite internal category and
  $\mathbb{A}$ a $J$-partite internal category in $\C$. A
  \emph{partite internal functor}
  $G \colon \mathbb{B} \rightarrow \mathbb{A}$ comprises an
  assignation on components $i \mapsto Gi$, assignations on objects
  $G_i \colon B_i \rightarrow A_{Gi}$, and assignations on morphisms
  $G_{ij} \colon B_{ij} \rightarrow A_{Gi,Gj}$, subject to the
  following partite
  analogues of the usual functor axioms (for all $i,j,k \in I$):
  \begin{align*}
    G_i \sigma_{ij} &= \sigma_{Gi,Gj} G_{ij} & G_{ii}\eta_i &= \eta_{Gi}G_i
    \\  G_j \tau_{ij} &= \tau_{Gi,Gj} G_{ij} &
     G_{ik}\mu_{ijk} &= \mu_{Gi,Gj,Gk}(G_{jk} \times_{G_j} G_{ij})\rlap{ .}
  \end{align*}
  A partite internal functor is a \emph{covering functor} if the
  assignation on components $i \mapsto Gi$ is invertible, and each
  square as below is a pullback in $\C$. We write
  $\mathrm{p}\cat{Cat}_\mathrm{cov}(\C)$ for the category of partite internal
  categories and partite covering functors in $\C$.
  \begin{equation}\label{eq:47}
    \cd{
      {B_{ij}} \ar[r]^-{G_{ij}} \ar[d]_{\sigma_{ij}} &
      {A_{Gi,Gj}} \ar[d]^{\sigma_{Gi,Gj}} \\
      {B_i} \ar[r]^-{G_i } &
      {A_{Gi}}\rlap{ .}
    }
  \end{equation}
\end{Defn}
This is an appropriate generalisation of the notion
in~\cite{Lawson2013Pseudogroups}: indeed, when $\C = \cat{Top}_p$, the
covering functors between $1$-partite internal groupoids are precisely
the covering functors of \emph{ibid}., p.~126. We now show that
partite covering functors can be identified with a special class of
partite internal cofunctors.
\begin{Defn}
  \label{def:38}
  Let $\mathbb{A}$ and $\mathbb{B}$ be internal partite categories in
  $\C$, and $F \colon \mathbb{A} \rightsquigarrow \mathbb{B}$ an
  internal cofunctor. We say that $F$ is:
  \begin{itemize}
  \item \emph{Bijective on components} if the function $i \mapsto Fi$
    is invertible;
  \item \emph{Bijective on objects} if each map $F_i \colon B_{Fi}
    \rightarrow A_i$ is invertible;
  \item \emph{Bijective on arrows} if each map
    $F_{ij} \colon A_{ij} \times_{A_i} B_{Fi} \rightarrow B_{Fi,Fj}$
    is invertible.
  \end{itemize}
\end{Defn}
\begin{Prop}
  \label{prop:22}
  The category $\mathrm{p}\cat{Cat}_\mathrm{cov}(\C)$ is contravariantly isomorphic
  to the subcategory of $\partc$ comprising all objects together with
  the maps which are bijective on components and arrows.
\end{Prop}
\begin{proof}
  Given a partite cofunctor
  $F \colon \mathbb{A} \rightsquigarrow \mathbb{B}$ which is bijective
  on components and arrows, the corresponding covering functor
  $G \colon \mathbb{B} \rightarrow \mathbb{A}$ has action on
  components and objects determined by the formulae $G(Fi) = i$
  and $G_{Fi} = F_i \colon B_{Fi} \rightarrow A_{i}$, while its
  action on arrows is determined by
  \begin{equation*}
    G_{Fi,Fj} = B_{Fi,Fj}
    \xrightarrow{F_{ij}^{-1}} A_{ij} \times_{A_i} B_{Fi}
    \xrightarrow{\pi_1} A_{ij}\rlap{ .}
  \end{equation*}
  Conversely, if $G \colon \mathbb{B} \rightarrow \mathbb{A}$ is a
  partite covering functor, the corresponding partite cofunctor
  $F \colon \mathbb{A} \rightsquigarrow \mathbb{B}$ has action on
  components and objects determined by $F(Gi) = i$ and
  $F_{Gi} = G_i \colon B_{i} \rightarrow A_{Gi}$, and
  action on arrows determined by
  \begin{equation*}
    F_{Gi,Gj} = (\pi_1, \pi_2) \colon A_{Gi,Gj} \times_{A_{Gi}} B_{i}
    \xrightarrow{\cong} B_{ij} 
  \end{equation*}
  the unique isomorphism induced by~\eqref{eq:42}'s being a pullback.
  It is easy to check that $F \mapsto G$ and $G
  \mapsto F$ are well-defined, mutually inverse assignations.
\end{proof}

Now, under our main equivalence $\parte \simeq \hcon$, the subcategory
of $\parte$ on the bijective-on-components-and-arrows cofunctors
corresponds to some subcategory of $\hcon$---one which, by
the preceding result, is contravariantly equivalent to
$\mathrm{pe}\cat{Cat}_{\mathrm{cov}}(\C)$. We now identify the maps in this subcategory.

\begin{Defn}
  \label{def:41}
  A join restriction functor $F \colon \A \rightarrow \B$ between join
  restriction categories is said to be \emph{localic} if it is
  bijective on objects, and moreover:
  \begin{enumerate}[(i)]
  \item For all $f, g \in \A(i,j)$, we have $Ff \wedge Fg = F(f \wedge
    g)$;
 \item For all $g \in \B(Fi,Fj)$, we have
    $g = \bigvee_{f \in \A(i,j)} g \wedge Ff$.
  \end{enumerate}
\end{Defn}
Here, $\wedge$ denotes meet with respect to the natural partial order
on each hom; these meets always exist by Lemma~\ref{lem:30}. Note that
these morphisms are the natural generalisation of the
\emph{hypercallitic} morphisms
of~\cite[p.~128]{Lawson2013Pseudogroups}.

We now show that, under the equivalence $\parte
\simeq \hcon$, the bijective-on-components-and-arrows partite
cofunctors correspond to the maps
\begin{equation}\label{eq:41}
  \cd[@-0.6em]{
    {\A} \ar[rr]^-{F} \ar[dr]_-{P} & \ltwocello{d}{\alpha} &
    {\B} \ar[dl]^-{Q} \\ &
    {\C}
  }
\end{equation} of $\hcon$ for which $F$ is localic. In fact, it will be
convenient for later use to prove something slightly more general.
\begin{Prop}
  \label{prop:21}
  Consider a map~\eqref{eq:41} of $\rcat[jr]$. We have that:
  \begin{itemize}
  \item $\Psi(F, \alpha)$ is bijective on components precisely when
    $F$ is bijective on objects;
  \item $\Psi(F, \alpha)$ is bijective on objects precisely when
    $\alpha$ is invertible;
  \item $\Psi(F, \alpha)$ is bijective on arrows precisely when:
    \begin{enumerate}[(i)]
    \item For all $f, g \in \A(i,j)$, we have
      $Q(Ff \wedge Fg) = QF(f \wedge g)$;
    \item For all $g \in \B(Fi,Fj)$, we have
      $Qg = \bigvee_{f \in \A(i,j)} Q(g \wedge Ff)$.
    \end{enumerate}
  \end{itemize}
\end{Prop}
\begin{proof}
  From the proof of Theorem~\ref{thm:4}, we see that $\Psi(F,\alpha)$
  is the partite cofunctor which acts:
  \begin{itemize}
  \item On component-sets via the action of $F$ on objects;
  \item On objects via the maps $F_i = \alpha_i \colon QFi \rightarrow
    Pi$;
  \item On arrows via the unique maps $F_{ij} \colon \Psi P_{ij}
    \times_{Pi} QFi \rightarrow \Psi Q_{Fi,Fj}$ for which
    \begin{equation}\label{eq:38}
      s_{Ff} = QFi \xrightarrow{\alpha_{i}^\ast(s_f)} \Psi P_{ij}
      \times_{Pi} QFi \xrightarrow{F_{ij}} \Psi Q_{Fi,Fj} \ \text{for
        all $f \in \A(i,j)$.}
    \end{equation}
  \end{itemize}
  The first two claims of the proposition are therefore immediate. For
  the third, recall that by Lemma~\ref{lem:26}, the family of sections
  $\{\alpha_i^\ast(s_f)\}$ constitute a basis for the local
  homeomorphism
  $\pi_2 \colon \Psi P_{ij} \times_{Pi} QFi \rightarrow QFi$ with
  induced local atlas
  \begin{equation*}
    \overline{\fneq f g \alpha_i} = \overline{P\sheq f g \alpha_i} =
    \overline{\alpha_i QF \sheq f g} = QF\sheq f g\rlap{ ,}
  \end{equation*}
  here using Remark~\ref{rk:2}, naturality of $\alpha$, and totality
  of $\alpha$. Thus, by Lemma~\ref{lem:10}, $F_{ij}$ is invertible
  precisely when the family of sections~\eqref{eq:38} are a basis for
  $\sigma_{Fi,Fj} \colon \Psi Q_{Fi,Fj} \rightarrow QFi$, and the
  induced local atlas is $QF\sheq f g$.

  Now, by construction, the family~\eqref{eq:38} induces the local
  atlas $\fneq[Q]{Ff}{Fg} = Q\sheq {Ff}{Fg}$; so the latter
  requirement says that $Q\sheq {Ff}{Fg} = QF\sheq f g$, or
  equivalently, by Lemma~\ref{lem:30}(iv), that
  $Q(Ff \wedge Fg) = QF(f \wedge g)$. On the other hand, given that
  the sections $\bigl(s_g : g \in \B(Fi,Fj)\bigr)$ are a basis for
  $\sigma_{Fi,Fj}$, to ask that $\bigl(s_{Ff} : f \in \A(i,j)\bigr)$ is also a basis is
  equally, by Lemma~\ref{lem:27}, to ask that
  \begin{equation*}
    s_g = \textstyle\bigvee_{f \in \A(i,j)} s_{Ff} \sheq {s_g} {s_{Ff}} \colon
    QFi \rightarrow \Psi Q_{Fi,Fj} \quad \text{for all $g \in \B(Fi,Fj)$.}
  \end{equation*}
  Now, by Lemma~\ref{lem:30}(iv) we always have the inequality
  $\geqslant$ above, and so have equality just when both sides have
  the same restriction. Since the sections $\bigl(s_g : g \in
  \B(Fi,Fj)\bigr)$ generate the local atlas $Q\sheq f g$, this is
  equally to ask that
  $Q \overline{g} = \bigvee_{f \in \A(i,j)} Q\sheq g {Ff}$ which, by
  Lemma~\ref{lem:30}(iv), is equivalent to asking that $Qg =
  \bigvee_{f \in \A(i,j)} Q(g \wedge Ff)$.
\end{proof}
\begin{Cor}
  \label{cor:5}
  Consider a morphism~\eqref{eq:41} of $\rcat$. If $F$ is localic then
  $\Psi(F, \alpha)$ is bijective on components and arrows, and the
  converse is true whenever $Q$ is hyperconnected. In particular,
  under the equivalence $\parte \simeq \hcon$, the cofunctors which
  are bijective on components and arrows correspond to the
  maps~\eqref{eq:41} of $\hcon$ in which $F$ is localic.
\end{Cor}
\begin{proof}
  Clearly, if $F$ is localic, then it satisfies conditions (i) and
  (ii) of Proposition~\ref{prop:21}, so that $\Psi(F, \alpha)$ is
  bijective on components and arrows. Suppose now that $Q$ is
  hyperconnected, and that $\Psi(F, \alpha)$ is bijective on
  components and arrows. Thus $Q(Ff \wedge Fg) = QF(f \wedge g)$, so
  that $Q\sheq {Ff}{Fg} = QF\sheq f g$; but now
  $\sheq {Ff}{Fg} = F\sheq f g$ by hyperconnectedness of $Q$, and so
  $Ff \wedge Fg = F f \wedge g$. So $F$ satisfies condition (i) to be
  localic, and a similar argument verifies (ii).
\end{proof}

Putting together Proposition~\ref{prop:22} and Corollary~\ref{cor:5},
we obtain the desired:
\begin{Thm}
  \label{thm:9}
  The main equivalence~\eqref{eq:33} restricts back to an equivalence
  \begin{equation*}
    \cd{
      {\mathrm{pe}\cat{Cat}_\mathrm{cov}(\C)^\mathrm{op}} \ar@<-4.5pt>[r]_-{\Phi} \ar@{<-}@<4.5pt>[r]^-{\Psi} \ar@{}[r]|-{\sim} &
      {\mathrm{jr}\cat{Cat}\mathord{\mkern1mu/\!/_{\!h\ell}\mkern1mu} \C} 
    }
  \end{equation*}
  between the opposite of the category of partite internal categories
  and internal covering functors, and the subcategory of $\hcon$
  comprising all the objects and the localic morphisms between them.
\end{Thm}
This result transcribes perfectly to the groupoidal situation. Any
join inverse category has joins of restriction idempotents, and so
admits equalities $\sheq f g$ and meets $f \wedge g$, so that
Definition~\ref{def:41} makes perfect sense for join-preserving
functors between join inverse categories. Tracing through the rest of
the arguments, \emph{mutatis mutandis}, now gives:
\begin{Thm}
  \label{thm:9}
  The equivalence~\eqref{eq:35} restricts back to an
  equivalence
  \begin{equation*}
    \cd{
      {\mathrm{pe}\cat{Gpd}_\mathrm{cov}(\C)^\mathrm{op}} \ar@<-4.5pt>[r]_-{\Phi_g} \ar@{<-}@<4.5pt>[r]^-{\Psi} \ar@{}[r]|-{\sim} &
      {\mathrm{ji}\cat{Cat}\mathord{\mkern1mu/\!/_{\!h\ell}\mkern1mu} \C} 
    }
  \end{equation*}
  between the opposite of the category of partite internal groupoids
  and internal covering functors, and the subcategory of $\hcon[ji]$
  comprising all the objects and the localic morphisms between them.
\end{Thm}
Finally, in the one-object case, we have:
\begin{Thm}
  \label{thm:10}
  The equivalence~\eqref{eq:36} restricts back to an
  equivalence
  \begin{equation*}
    \cd{
      {\cat{Gpd\C ov}_e(\C)^\mathrm{op}} \ar@<-4.5pt>[r]_-{\Phi_g} \ar@{<-}@<4.5pt>[r]^-{\Psi} \ar@{}[r]|-{\sim} &
      {\cat{Ps\G rp}(\C)_\ell} 
    }
  \end{equation*}
  between the opposite of the category of internal groupoids in $\C$
  and covering functors, and the category of complete $\C$-pseudogroups
  and localic homomorphisms.
\end{Thm}

\subsection{The (localic, hyperconnected) factorisation system on $\mathrm{jr}\cat{Cat}$}
\label{sec:local-hyperc-fact}
As an application of the concepts developed in the previous section,
we show that there is a \emph{factorisation system} on the category of
join restriction categories $\mathrm{jr}\cat{Cat}$ whose two classes
are the localic and the hyperconnected morphisms. 
%
We first recall:

\begin{Defn}
  \label{def:39}
  A \emph{factorisation system} $(\E, \M)$ on a category $\D$
  comprises two classes $\E$ and $\M$ of maps, each of which
  is closed under composition with isomorphisms, and which satisfy the
  following two axioms.
  \begin{enumerate}[(i),itemsep=0.25\baselineskip]
  \item \emph{Factorisation}: each $f \colon A \rightarrow B$ can be
    written as $f = me \colon A \rightarrow C \rightarrow B$ for some
    $e \in \E$ and $m \in \M$;
  \item \emph{Orthogonality}: each $e \in \E$ is orthogonal to each
    $m \in \M$; this is to say that, for any commuting square as in the
    solid part of
    \begin{equation}\label{eq:37}
      \cd{
        {A} \ar[r]^-{f} \ar[d]_{e} &
        {C} \ar[d]^{m} \\
        {B} \ar[r]_-{g} \ar@{-->}[ur]_-{j}&
        {D}
      }
    \end{equation}
    there is a unique map $j
    \colon B \rightarrow C$ as displayed making both triangles commute.
  \end{enumerate}
\end{Defn}

The two classes of a factorisation system $(\E, \M)$ determine each
other: a map lies in $\E$ \emph{precisely} when it is orthogonal to
each map in $\M$, and vice versa. Thus, given any class of maps $\M$
in a category $\D$, there is at most one factorisation system on $\D$
whose right class is given by $\M$.
In~\cite[Section~2]{Cockett2014Restriction}, the present authors
showed that when $\M$ is the class of hyperconnected maps in
$\mathrm{r}\cat{Cat}$, such a factorisation system exists. We will now
show that the same is true when $\M$ is the class of hyperconnected
maps in $\mathrm{jr}\cat{Cat}$.

\begin{Thm}
  \label{thm:11}
  (Localic,
  hyperconnected) is a factorisation system on $\mathrm{jr}\cat{Cat}$.
\end{Thm}
\begin{proof}
  First, given a commuting square in $\mathrm{jr}\cat{Cat}$, as
  to the left in:
  \begin{equation*}
    \cd{
      {\A} \ar[r]^-{F} \ar[d]_{L} &
      {\C} \ar[d]^{H} & &
            (\A, GL) \ar[r]^-{(F,1)} \ar[d]_-{(L,1)} &
      (\C, H) \\
      {\B} \ar[r]_-{G} \ar@{-->}[ur]_-{J} &
      {\D} & & 
      (\B, G) \ar@{-->}[ur]_-{(J,1)}
    }
  \end{equation*}
  with $L$ localic and $H$ 
  hyperconnected, we will exhibit a unique $J$ as shown.
  
  Suppose first that $\D$ admits local glueings. We can encode the
  data of the given commuting square as a span in
  $\mathrm{r}\cat{Cat}\mathord{\mkern1mu/\!/\mkern1mu} \D$, as right
  above, and a map $J \colon \B \rightarrow \C$ is a filler for the
  square as to the left just when it fits into a commuting triangle as
  to the right. But as $L$ is bijective on objects, \emph{any}
  $(J, \alpha)$ fitting into such a commuting triangle must have
  $\alpha = 1$. So we are left with proving that there is a unique
  factorisation of $(F,1)$ through $(L,1)$. Since $(\C, H)$ lies in
  the reflective subcategory
  $\mathrm{jr}\cat{Cat}\mathord{\mkern1mu/\!/_{\!h}\mkern1mu} \D$ of
  $\mathrm{r}\cat{Cat}\mathord{\mkern1mu/\!/\mkern1mu} \D$, it
  suffices to show that $(L,1)$ is inverted by the reflector
  $\Phi \Psi$ into this subcategory. In fact, by
  Proposition~\ref{prop:21}, $\Psi(L,1)$ is already invertible, since
  $L$ is localic and $1$ is invertible.

  If $\D$ does not admit local glueings, then we embed it into its
  local glueing completion via
  $\iota \colon \D \rightarrow \mathrm{Gl}(\D)$ as in
  Proposition~\ref{prop:15}, and apply the preceding argument to
  obtain a unique diagonal filler in
  \begin{equation*}
    \cd{
      {\A} \ar[r]^-{F} \ar[d]_{L} &
      {\C} \ar[d]^{\iota H} \\
      {\B} \ar[r]_-{\iota G} \ar@{-->}[ur]_-{J} &
      {\mathrm{Gl}(\D)}\rlap{ .}
    }
  \end{equation*}
  Now as $\iota \colon \D \rightarrow \mathrm{Gl}(\D)$ is monic, this
  $J$ is also a unique diagonal filler for the original square. This
  proves the orthogonality of localic and hyperconnected maps; it
  remains to show factorisation.

  Consider a map $F \colon \C \rightarrow \D$ in
  $\mathrm{jr}\cat{Cat}$; we again start by assuming that $\D$ admits
  local glueings. In this case, we can form the reflection
  \begin{equation*}
    \cd[@!C@C-2em]{
      {\C} \ar[rr]^-{\eta_F} \ar[dr]_-{F} & \twoeq{d} &
      {\Phi \Psi F} \ar[dl]^-{\pi_{\Psi F}} \\ &
      {\D}
    }
  \end{equation*}
  of $F \colon \C \rightarrow \D$ in $\rcat$ into $\hcon$. By
  construction, $\pi_{\Psi F}$ is hyperconnected; further, since
  $\Psi \dashv \Phi$ is a Galois adjunction, the unit
  $\eta_\C \colon (\C, F) \rightarrow (\Phi \Psi F, \pi_{\Psi F})$ is
  inverted by $\Psi$, whence, by Corollary~\ref{cor:5}, $\eta_F$ is
  localic. So we have the desired (localic, hyperconnected)
  factorisation of $F$.
  
  If $\D$ does not admit local glueings, we can like before consider
  the embedding $\iota \colon \D \rightarrow \mathrm{Gl}(\D)$ into the
  local glueing completion. By the preceding argument we have a
  (localic, hyperconnected) factorisation of $\iota F$, as in the
  solid part of:
  \begin{equation*}
    \cd{
      {\C} \ar[r]^-{F} \ar[d]_{L} &
      {\D} \ar[d]^{\iota} \\
      {\E} \ar[r]_-{H} \ar@{-->}[ur]_-{K}&
      {\mathrm{Gl}(\D)}\rlap{ .}
    }
  \end{equation*}
  Since $L$ is localic, and $\iota$ is (fully faithful and hence)
  hyperconnected, there is a unique diagonal filler as displayed.
  Since $\iota$ and $\iota K = H$ is hyperconnected, so is $K$, and so
  $KL$ is the required (localic, hyperconnected) factorisation of $F$.
\end{proof}

Note that, in the case where $\D$ has local glueings,
Proposition~\ref{prop:26} provides an explicit description of the
(localic, hyperconnected) factorisation of
$F \colon \C \rightarrow \D$ in $\mathrm{jr}\cat{Cat}$, and it is easy
to see that this same description is also valid for a general $\D$.
The following result gives a corresponding explicit description of the
orthogonal liftings of localic maps against hyperconnected ones.

\begin{Prop}
  \label{prop:25}
  Given a commuting square in $\mathrm{jr}\cat{Cat}$
  \begin{equation*}
        \cd{
      {\A} \ar[r]^-{F} \ar[d]_{L} &
      {\C} \ar[d]^{H} \\
      {\B} \ar@{-->}[ur]_-{J} \ar[r]_-{G} &
      {\D}
    }
  \end{equation*}
  with $L$ localic and $H$ hyperconnected, the unique diagonal filler
  $J \colon \B \rightarrow \C$ is determined as follows.
  \begin{itemize}
  \item  On \textbf{objects} by requiring that $J(Li) = Fi$;
  \item On \textbf{morphisms} by requiring that
    $J(g \colon Li \rightarrow Lj) = \bigvee_{f \in \A(i,j)} Ff \circ
    \varphi_{fg}$, where $\varphi_{fg} \in \O(Fi)$ is unique such that
    $H\varphi_{fg} = G\sheq g {Lf}$.
  \end{itemize}
\end{Prop}
\begin{proof}
  The given conditions completely specify $J$'s action since $L$ is
  bijective on objects. The condition on objects is clearly necessary.
  As for the condition on morphisms,  given
  $g \colon Li \rightarrow Lj$ in $\B$, we have since $L$ is localic that
  \begin{align*}
     g &= \textstyle\bigvee_{f \in \A(i,j)} g \wedge Lf = \bigvee_{f
      \in \A(x,y)} Lf \sheq g {Lf}\\
\text{and so }\ \      Jg &=\textstyle \bigvee_{f \in \A(x,y)} JLf \circ J\sheq g {Lf} = \bigvee_{f
      \in \A(x,y)} Ff \circ J\sheq g {Lf}\rlap{ .}
  \end{align*}
  Writing $\varphi_{fg} \defeq J\sheq g {Lf}$, we have that
  $H\varphi_{fg} = HJ\sheq g {Lf} = G\sheq g {Lf}$. Since $H$ is
  hyperconnected, this property uniquely determines $\varphi_{fg} \in \O(Fi)$.
\end{proof}

\subsection{The (localic, hyperconnected) factorisation system on
  $\hcon$}
\label{sec:local-hyperc-fact-1}

We now show that the factorisation system of the previous section
lifts to a factorisation system on $\hcon$ whenever $\C$ is a join
restriction category with local glueings. The classes of maps are
given by:

\begin{Defn}
  \label{def:40}
  A map $(F, \alpha) \colon (\A, P) \rightarrow (\B, Q)$ of $\hcon$ is
  said to be \emph{hyperconnected} if $\alpha$ is invertible, and
  \emph{localic} if $F$ is localic in $\mathrm{jr}\cat{Cat}$.
\end{Defn}
Note that, if $(F, \alpha)$ is hyperconnected, then $QF \cong P$ is
hyperconnected; since $Q$ is hyperconnected, it follows that $F$ is
also hyperconnected. However $F$ could be hyperconnected without
$\alpha$ being invertible, so that hyperconnectedness of $(F, \alpha)$
is \emph{strictly} stronger than hyperconnectedness of $F$.
\begin{Prop}
  \label{prop:24}
  Let $\C$ be a join restriction category with local glueings.
  (Localic, hyperconnected) is a factorisation system on $\hcon$.
\end{Prop}
\begin{proof}
  Given a map $(F, \alpha) \colon (\A, P) \rightarrow (\B, Q)$ of
  $\hcon$, we may form a (localic, hyperconnected) factorisation $F =
  HL \colon \A \rightarrow \D \rightarrow \B$
  of $F \colon \A \rightarrow \B$ in $\mathrm{jr}\cat{Cat}$; the
  desired factorisation in $\hcon$ is now given by
  \begin{equation*}
    (\A, P) \xrightarrow{(L, \alpha)} (\D, QH) \xrightarrow{(H, 1)}
    (\D, Q)\rlap{ .}
  \end{equation*}
  For orthogonality, consider a square in $\hcon$ as to the left in
  \begin{equation*}
    \cd{
      {(\A, P)} \ar[r]^-{(F, \alpha)} \ar[d]_{(L, \gamma)} &
      {(\D, R)} \ar[d]^{(H, \delta)} & &
      {\A} \ar[r]^-{F} \ar[d]_{L} &
      {\D} \ar[d]^{H} \\
      {(\B, Q)} \ar[r]_-{(G, \beta)} \ar@{-->}[ur]_-{(J, \theta)}&
      {(\E, S)} & & 
      {\B} \ar[r]_-{G} \ar@{-->}[ur]_-{J} &
      {\E}
    } 
  \end{equation*}
  where $L$ is localic and $\delta$ is invertible. As argued above, it
  follows that $H$ is hyperconnected, and so using orthogonality in
  $\mathrm{jr}\cat{Cat}$, we obtain a unique diagonal filler as to the
  right above. We wish to lift to this a filler as to the left. Since
  $\delta$ is invertible, the \emph{unique} possible $\theta$
  satisfying $\theta \circ \delta J  = \beta$ is
  $\theta = \beta \circ \delta^{-1} J$, and this also satisfies
  $\gamma \circ \theta L = \gamma \circ \beta L \circ \delta^{-1} JL =
  \alpha \circ \delta F \circ \delta^{-1} F = \alpha$. So $(J,
  \theta)$ is the desired unique filler for the square to the left.
\end{proof}

By transporting this factorisation system across the equivalence
$\hcon \simeq \parte$, we obtain a factorisation system on $\parte$.
In fact, this is easy to describe.

\begin{Prop}
  \label{prop:27}
  Under the equivalence $\hcon \simeq \parte$, the (localic,
  hyperconnected) factorisation system on $\hcon$ corresponds to the
  factorisation system (bijective on components and arrows, bijective
  on objects) on $\parte$.
\end{Prop}
\begin{proof}
  This is immediate from Proposition~\ref{prop:21} and
  Corollary~\ref{cor:5}.
\end{proof}

We leave it as an instructive exercise to the reader to give an
explicit description of factorisations and orthogonal liftings for
this factorisation system on $\parte$.

%
%
%
%
%

\section{Applications}
\label{sec:applications}

In this section, we instantiate our main result and its variants at
particular choices of join restriction category $\C$. This will allow
us to recapture and extend existing correspondences in the literature,
and also to construct various \emph{completions}. 

\subsection{Haefliger groupoids}
\label{sec:germ-groupoids}
In~\cite{Ehresmann1954Structures}, Ehresmann describes for a pair of
spaces $E$ and $E'$, the space $\Pi(E, E')$ of germs of partial
homeomorphisms $E \rightarrow E'$, and explains that, when $E = E'$,
we have a groupoid $\Pi(E)$. 
Later, Haefliger~\cite{Haefliger1971Homotopy} focussed on the
\emph{smooth} variant of the germ groupoid---that is, the groupoid of
all local diffeomorphisms of a smooth manifold $M$---and this has
subsequently come to be known as the \emph{Haefliger groupoid}. In our
language, the Haefliger groupoid is the internal \'etale groupoid in
$\cat{Smooth}_p$ associated to the complete
$\mathrm{Smooth}_p$-pseudogroup $\I(M) \rightarrow \I(M)$. We can
adapt this construction to the more general setting of our main
theorem as follows.

\begin{Defn}
  \label{def:55}
  Let $\C$ be a join restriction category with local glueings. The
  \emph{Haefliger category} of $\C$ is the partite source-\'etale
  internal category $\H(\C)$ in $\C$ associated via $\Psi$ to the
  join restriction functor $1_\C \colon \C \rightarrow \C$. 
\end{Defn}

The components of $\H(\C)$ are indexed by objects of $\C$; the
object of objects associated to $A \in \mathrm{ob} \C$ is $A$ itself; and
the object of morphisms $\H(\C)_{AB}$ is the appropriate analogue of
Ehresmann's $\Pi(A,B)$. 
\begin{Ex}
  \label{ex:13} ~
  \begin{itemize}[itemsep=0.25\baselineskip]
  \item When $\C = \cat{Set}_p$, elements of $\H(\C)_{AB}$ with source
    $a \in A$ and target $b \in B$ are germs of partial functions
    $A \rightarrow B$ sending $a$ to $b$. There is exactly \emph{one}
    such germ---represented by the partial function with graph
    $\{(a,b)\}$---so that $\H(\C)_{AB} = A \times B$. (Alternatively,
    since every map in $\cat{Set}_p$ is \'etale, we can derive this
    conclusion from Proposition~\ref{prop:34} below).
    
  \item When $\C = \cat{Top}_p$, elements of $\H(\C)_{AB}$ with source
    $a$ and target $b$ are
germs of partial continuous maps $A \rightarrow B$ sending $a$ to $b$.
\item When $\C = \mathrm{Smooth}_\mathbb{R}$,  $\H(\C)_{AB}$
  is the space of germs of partial smooth maps. In particular,
  $\H(\C)_{M \mathbb{R}}$ incarnates the ring of smooth functions on
  $M$; more specifically, elements of $\H(\C)_{M \mathbb{R}}$ with
  source and target projections $m \in M$ and $r \in \mathbb{R}$ are
  the germs at $m$ of elements $f \in C_{\infty}(U)$ where
  $U \subseteq M$ is open, $m \in U$ and $f(m) = r$.
    
\item In a similar spirit, when $\C = \mathrm{Sch}_p$ and
    $\mathbf{A}^1 =
    (\mathrm{Spec}\,\mathbb{Z}[x],\O_{\mathrm{Spec}\,\mathbb{Z}[x]})$
    we have that
    $\H(\C)_{X, \mathbf{A}^1}$ is $(\O_X, \O_X^\ast \O_X)$, i.e., the
    structure sheaf of $X$ seen as a local isomorphism over $X$.
  \item Let $\C = \cat{Bun}(\cat{Set}_p)$, the join restriction
    category of discrete bundles. The source-target span of its
    Haefliger category at objects $\xi \colon X' \rightarrow X$ and
    $\gamma \colon Y' \rightarrow Y$ is a diagram of sets and (total)
    functions of the form
  \begin{equation}\label{eq:62}
    \cd{
      X' \ar[d]_-{\xi} & \H(\C)_{\xi \gamma} \times_{X} X' \pullbackcorner[dl] \ar[d]^-{\pi_1}
      \ar[l]_-{\pi_2} \ar[r]^-{\tau'_{\xi \gamma}} & Y' \ar[d]_-{\gamma} \\
      X & \H(\C)_{\xi \gamma} \ar[l]_-{\sigma_{\xi \gamma}} \ar[r]^-{\tau_{\xi
          \gamma}} & Y\rlap{ ,}
    }
  \end{equation}
  and we can calculate that elements of $\H(\C)_{\xi \gamma}$ with
  source $x \in X$ and target $y \in Y$ are functions
  $f \colon \xi^{-1}(x) \rightarrow \gamma^{-1}(y)$ between the
  fibres. The map $\tau'_{\xi \gamma}$ sends such an element $f$ and
  an element $x' \in \xi^{-1}(x)$ to $f(x') \in \gamma^{-1}(y)$.
  In particular, for a single object $\xi \colon X \rightarrow X'$,
  the internal category $\H(\C)_{\xi\xi} \rightrightarrows X$
  in $\cat{Set}$ is the so-called \emph{internal full subcategory}
  associated to the map $\xi$, and $\tau'_{\xi\xi}$ exhibits $\xi$ as
  an internal presheaf over this internal full subcategory.

\item Let $\C = \cat{Bun}(\cat{Top}_p)$, the join restriction category
  of topological bundles. Again, the Haefliger category involves
  diagram of the form~\eqref{eq:62}, where this time elements of
  $\H(\C)_{\xi \gamma}$ over $x \in X$ and $y \in Y$ are germs of
  partial continuous maps sending $x$ to $y$ in the base, and lifting
  to a partial continuous map on the fibres. The maps
  $\tau'_{\xi \gamma}$ encode an action by $\H(\C)$ on the family of
  all bundles $\xi \colon X' \rightarrow X$. This
  generalises~\cite{Haefliger1958Structures}'s construction of the
  ``groupoid of germs of local automorphisms of a fibred space
  $p \colon E \rightarrow B$'' (\sec I.4) and the action of this
  groupoid on $p$ (\sec I.5).
\end{itemize}
\end{Ex}

\begin{Rk}
  \label{rk:6}
  The Haefliger category of $\C$ is ``generic'' among partite internal
  categories built out of $\C$, in the following sense. Suppose that
  $F \colon \C \rightarrow \D$ is a hyperconnected join restriction
  functor. On the one hand, we can view $F$ as an object of
  ${\mathrm{jr}\cat{Cat}\mathord{\mkern1mu/\!/_{\!h}\mkern1mu} \D}$,
  and construct a source-\'etale partite internal category $\Phi(F)$
  in $\D$. On the other hand, we can form the source-\'etale partite
  internal category $\H(\C)$ in $\C$. Now $F$ preserves total maps,
  local homeomorphisms, and pullbacks along local homeomorphisms, and
  so also preserves source-\'etale partite internal categories. It is
  now easy to see that applying $F$ to the source-\'etale internal
  category $\H(\C)$ in $\C$ yields, to within isomorphism, $\Phi(F)$
  in $\D$.
\end{Rk}

Of course, we can also associate a \emph{Haefliger groupoid} $\Pi(\C)$
to any join restriction category with local glueings, taking it to be
the partite \'etale internal groupoid $\Pi(\C)$ in $\C$ associated via
$\Psi$ to the join inverse category $\mathrm{PIso}(\C)$ over $\C$.
So, for example, the Haefliger groupoid $\Pi(\cat{Top}_p)$ has spaces
of morphisms $\Pi(\cat{Top}_p)_{AB}$ given \emph{exactly} by
Ehresmann's spaces $\Pi(A,B)$.

In fact, the Haefliger groupoid of a join restriction category $\C$ is
a special case of a Haefliger category. Indeed, by Theorem~\ref{thm:3}
we can describe $\Pi(\C)$ as the partite \'etale groupoid obtained by
applying $\Psi$ to the \'etale join restriction category
$\mathrm{Et}(\C) \rightarrow \C$ over $\C$ where, as in
Corollary~\ref{cor:6}, we write $\mathrm{Et}(\C)$ for the category of
\'etale maps in $\C$. Since $\mathrm{Et}(\C)$ is closed in $\C$ under
local glueings, this is equally the result of applying $\Psi$ to
$1 \colon \mathrm{Et}(\C) \rightarrow \mathrm{Et}(\C)$; thus we have
$\Pi(\C) = \H(\mathrm{Et}(\C))$. We can use this observation to obtain
an alternative understanding of the ``objects of arrows''
$\Pi(\C)_{AB}$ in the Haefliger groupoid.

\begin{Prop}
  \label{prop:34}
  Let $\C$ be a join restriction category with glueings. The
  source-target span $\sigma_{AB} \colon A \leftarrow \Pi(\C)_{AB}
  \rightarrow B \colon \tau_{AB}$ exhibits $\Pi(\C)_{AB}$ as a product
  of $A$ and $B$ in the category $\mathrm{Lh}(\C)$ of local
  homeomorphisms between objects of $\C$.
\end{Prop}

\begin{proof}
  Let $u \colon A \leftarrow X \rightarrow B \colon v$ be a span in
  $\mathrm{Lh}(\C)$. We must show there is a \emph{unique} local
  homeomorphism $h$ rendering commutative the diagram:
  \begin{equation}\label{eq:59}
    \cd[@-1em]{
    & X \ar[d]_-{h} \ar@/_1em/[ddl]_-{u} \ar@/^1em/[ddr]^-{v} \\
    & \Pi(\C)_{AB} \ar[dl]|-{\sigma_{AB}} \ar[dr]|-{\tau_{AB}} \\
    A & & B\rlap{ .}
    }
  \end{equation}

  Now, since $\C$ admits local glueings, the category
  $\mathrm{Lh}(\C)$ has a pullback-stable initial object $0$, obtained
  by glueing the empty atlas on \emph{any} object of $\C$. We can thus
  construct a two-component source-\'etale partite internal category
  $\mathbb{X}$ in $\mathrm{Et}(\C)$, where $\mathbb{X}_0 = A$,
  $\mathbb{X}_1 = B$, $\mathbb{X}_{01} = X$, $\mathbb{X}_{10} = 0$,
  $\mathbb{X}_{00} = A$, $\mathbb{X}_{11} = B$, and with remaining
  data obtained in the obvious manner. This $\mathbb{X}$ is such that
  maps $h$ rendering~\eqref{eq:59} commutative correspond precisely to
  \emph{identity-on-objects} partite cofunctors
  $\mathbb{X} \rightsquigarrow \Pi(\C)$ in $\mathrm{Et}(\C)$. Since
  $\Pi(\C) = \Phi(1_\mathrm{Et}(\C))$, such maps correspond in turn to
  maps $(\Phi \mathbb{X}, \pi_{\mathbb{X}}) \rightarrow
  (\mathrm{Et}(\C), 1_\mathrm{Et}(\C))$ of
  ${\mathrm{jr}\cat{Cat}\mathord{\mkern1mu/\!/_{\!h}\mkern1mu}
    \mathrm{Et}(\C)}$ whose $2$-cell component is the \emph{identity}.
  Clearly $(\pi_\mathbb{X}, 1)$ is the unique such map.
\end{proof}

\begin{Exs}
  \label{ex:6}~
  \begin{itemize}[itemsep=0.25\baselineskip]
    \item When $\C = \cat{Top}_p$, this recaptures a result of
      Selinger~\cite{Selinger1994LH-has-nonempty}, proving the existence
      of binary products in the category of local homeomorphisms between
      topological spaces.
    \item When $\C = \cat{Set}_p$ every map is already \'etale, so
      that $\Pi(\C) = \H(\C)$; likewise every total map is already a
      local homeomorphisms, so that $\Pi(\C)_{AB} = \H(\C)_{AB} = A
      \times B$, the cartesian product of sets.
    \item When $\C = \mathrm{Bun}(\cat{Set}_p)$, the category
      $\mathrm{Lh}(\C)$ is the category $\cat{Cart}$ whose objects are
      functions between sets, and whose morphisms are pullback
      squares. It follows that $\cat{Cart}$ has binary products.
      Explicitly, if $\xi \colon X' \rightarrow X$ and
      $\gamma \colon Y' \rightarrow Y$ in $\cat{Cart}$, then their
      product is $\zeta \colon Z' \rightarrow Z$ where
      \begin{equation*}
        Z = \{(x,y,\theta) : x \in X, y \in Y, \theta \colon
        \xi^{-1}(x) \cong \gamma^{-1}(y)\}
      \end{equation*}
      and an element of $\zeta^{-1}(x,y,\theta)$ is a pair $\bigl(x' \in
      \xi^{-1}(x), y' \in \gamma^{-1}(y)\bigr)$ with $\theta(x') = y'$.
    \end{itemize}
\end{Exs}

\subsection{The Resende correspondence}
\label{sec:resende-corr}

In~\cite{Resende2007Etale}, Resende describes a
correspondence between join inverse monoids---\emph{abstract complete
  pseudogroups} in his terminology---and localic \'etale groupoids. He
establishes this by way of a third notion which he terms an
\emph{inverse quantal frame}. More precisely, he establishes
in~\cite[Theorem~4.15]{Resende2007Etale} an equivalence between the
categories of join inverse monoids and inverse quantal frames, and a
non-functorial correspondence---described
following~\cite[Corollary~5.12]{Resende2007Etale}---between inverse
quantal frames and join inverse monoids.

In fact, Resende's abstract complete pseudogroups are exactly our
complete $\cat{Loc}_p$-pseudogroups. Indeed, if
$\theta \colon S \rightarrow \I(X)$ is a complete
$\cat{Loc}_p$-pseudogroup, then since $\theta$ restricts back to an
isomorphism $E(S) \rightarrow E(\I(X)) \cong X$, the locale $X$ can be
identified with the locale of idempotents of $S$; whereupon
functoriality of $\theta$ forces
$\theta(s) \colon E(S) \rightharpoonup E(S)$ to be the partial locale
map given by $e \mapsto s^\ast e s$. Thus, our Theorem~\ref{thm:7}
specialised to the case $\C = \cat{Loc}_p$ establishes a
\emph{functorial} equivalence between the category of join inverse
monoids (= abstract complete pseudogroups) and the category of localic
\'etale groupoids and cofunctors.


The identification of abstract complete pseudogroups and complete
$\cat{Loc}_p$-pseudogroups has an analogue for join restriction
categories, described in~\cite{Cockett2014Restriction}; using this, we
may generalise Resende's correspondence further. The key definition~is:


\begin{Defn}
  \label{def:30} (cf.~\cite[\sec 4.1]{Cockett2002Restriction}) The
  \emph{fundamental functor} $\O_\A \colon \A \rightarrow \cat{Loc}_p$ of
  a join restriction category $\A$ is given on objects by
  $X \mapsto \O(X)$, and on maps by sending
  $f \colon X \rightarrow Y$ to the partial locale map
  $\O(f) \colon \O(X) \rightharpoonup \O(Y)$ with
  $\O(f)^\ast(e) = \overline{ef}$.
\end{Defn}

The crucial fact about the fundamental functor of $\A$ is that it is
the essentially-unique hyperconnected functor from $\A$ to
$\cat{Loc}_p$. The precise result is the following one, which combines
Proposition~3.3 and Proposition~7.3 of~\cite{Cockett2014Restriction}.

%
%

\begin{Prop}
  \label{prop:29}
  The fundamental functor $\O_\A \colon \A \rightarrow \cat{Loc}_p$ is
  a hyperconnected join restriction functor. For any
  $H \colon \A \rightarrow \cat{Loc}_p$, the following are equivalent:
  \begin{enumerate}[(i)]
  \item $H$ is hyperconnected;
  \item $H$ is a terminal object in $\mathrm{jr}\cat{Cat}(\A,
    \cat{Loc}_p)$;
  \item $H$ is naturally isomorphic to the fundamental functor
    $\O_\A \colon \A \rightarrow \cat{Loc}_p$.
\end{enumerate}
\end{Prop}

We thus obtain the following generalisation of the correspondence
between abstract complete pseudogroups and complete
$\cat{Loc}_p$-pseudogroups.

\begin{Cor}
  \label{cor:8} (\cite[Theorem~7.14]{Cockett2014Restriction}) There is
  an equivalence of categories
  \begin{equation*}
    \cd{
      {\mathrm{jr}\cat{Cat}\mathord{\mkern1mu/\!/_{\!h}\mkern1.5mu}
        \cat{Loc}_p} \ar@<-4.5pt>[r]_-{U} \ar@{<-}@<4.5pt>[r]^-{V} \ar@{}[r]|-{\sim} &
      {\mathrm{jr}\cat{Cat}}
    }
  \end{equation*}
  where $U$ forgets the projection down to $\cat{Loc}_p$, and where
  $V$ sends $\A$ to $(\A, \O_\A)$ and sends
  $F \colon \A \rightarrow \B$ to
  $(F, \alpha) \colon (\A, \O_\A) \rightarrow (\B, \O_\B)$, where
  $\alpha$ has components $\alpha_i \colon \O_B(Fi) \rightarrow
  \O_\A(i)$ given by $\alpha_i^\ast(e) = Fe$ for all $e \in \O_\A(i)$.
\end{Cor}

Taking this together with our main results, we thus obtain:

\begin{Thm}
  \label{thm:12}
  There are equivalences of categories
  \begin{equation}\label{eq:61}
    \cd[@C-0.5em]{
      {\mathrm{pe}\cat{Cat}_c(\cat{Loc}_p)} \ar@<-4.5pt>[r]_-{U \Phi}
      \ar@{<-}@<4.5pt>[r]^-{\Psi V} \ar@{}[r]|-{\sim} &
      {\mathrm{jr}\cat{Cat}}
    } \quad \text{and} \quad 
    \cd[@C-0.5em]{
      {\mathrm{pe}\cat{Gpd}_c(\cat{Loc}_p)} \ar@<-4.5pt>[r]_-{U \Phi_g}
      \ar@{<-}@<4.5pt>[r]^-{\Psi V} \ar@{}[r]|-{\sim} &
      {\mathrm{ji}\cat{Cat}}
    }
  \end{equation}
  together with their obvious restrictions to the one-object case.
\end{Thm}
For the left-to-right direction of these equivalences, we cannot
improve upon the descriptions given in
Sections~\ref{sec:right-adjoint} and~\ref{sec:group-join-inverse}. For
the right-to-left direction, we can use the characterisation of local
glueings in $\cat{Loc}_p$ from Example~\ref{ex:7} to say more. It
suffices to do this in the most general case.

\begin{Prop}
  \label{prop:30}
  Let $\A$ be a join restriction category with object-set $I$. The
  corresponding $I$-partite source-\'etale localic category has:
  \begin{itemize}
  \item Locale
  of objects $X_i = \O(i)$ for $i \in \A$;
\item Locale of arrows $X_{ij}$ for $i,j \in \A$ given by
  \begin{equation}\label{eq:48}
  X_{ij} = \bigl\{\,\bigl(\theta_f \in \O(i) : f \in \A(i,j)\bigr) \mid \theta_f e = \theta_{fe} \text{
  for $f \in \A(i,j)$, $e \in \O(i)$}\,\bigr\}\rlap{ .}
\end{equation}
\item Source--target maps $\sigma_{ij} \colon X_i \leftarrow X_{ij}
  \rightarrow X_j \colon \tau_{ij}$ given by
\begin{equation*}
  \sigma_{ij}^\ast(d) = (\overline{fd} : f \in \A(i,j)) \qquad
  \tau_{ij}^\ast(d) = (\overline{df} : f \in \A(i,j))\rlap{ ;}
\end{equation*}
\item The partial section
$s_f \colon X_i \rightarrow X_{ij}$ of $\sigma_{ij}$ associated to
the map $f \in \A(i,j)$ given by
$s_f^\ast(\theta) = \theta_f$.
\item Identities $\eta_i
\colon X_i \rightarrow X_{ii}$ given by $\eta_{i}^\ast(\theta) =
\theta_{1_i}$;
\item Multiplication $\mu_{ijk} \colon X_{jk} \times_{X_j} X_{ij}
  \rightarrow X_{ik}$ given by
  \begin{equation*}
\mu_{ijk}^\ast(\theta) = \bigl(\theta_{gf}
: f \in \A(i,j), g \in \A(j,k)\bigr)\rlap{ ,}
\end{equation*}
under the identification of
  $X_{jk} \times_{X_j} X_{ij}$ with the locale of all families
\begin{equation}\label{eq:50}
  \bigl\{\,\bigl(\psi_{f,g} \in \O(i) :
  f \in \A(i,j), g \in \A(j,k)\bigr) \mid \psi_{f,g} e =
  \psi_{fe,g} \text{, } \psi_{f,g} \overline{df} = \psi_{f,gd}\,\bigr\}\rlap{ .}
\end{equation}
\end{itemize}
\end{Prop}
\begin{proof}
  The identification of $X_i$ is clear. By definition, $X_{ij}$ is the
  glueing in $\cat{Loc}_p$ of the local atlas
  $\bigl(\sheq f g : f,g \in \A(i,j)\bigr)$ on $\O(i)$, and so by
  Example~\ref{ex:7} comprises the locale of all families
  \begin{equation}\label{eq:49}
    \bigl\{\,\bigl(\theta_f \in \O(i) : f \in \A(i,j)\bigr) \mid
    \theta_f \leqslant \overline{f} \text{ and } \theta_g \sheq f g \leqslant
    \theta_f\,\bigr\}\rlap{ .}
  \end{equation}
  We claim these are the same families as in~\eqref{eq:48}. First,
  given a family as in~\eqref{eq:48} we have
  $\theta_f = \theta_{f \overline{f}} = \theta_{f} \overline{f}$, so
  that $\theta_f \leqslant \overline{f}$, and also that
  $\theta_g \sheq f g = \theta_{g \sheq f g} = \theta_{f \sheq f g} =
  \theta_{f}\sheq f g \leqslant \theta_f$. Conversely, given a family
  as in~\eqref{eq:49}, we have
  $\theta_f e = \theta_f \overline{fe} = \theta_f \sheq f {fe}
  \leqslant \theta_{fe}$ and
  $\theta_{fe} = \theta_{fe}\overline{fe}e = \theta_{fe}\sheq {fe} fe
  \leqslant \theta_f e$ so that $\theta_f e = \theta_{fe}$.
  So~\eqref{eq:48} is a correct description of $X_{ij}$.

  The descriptions of $\sigma_{ij}$ and the partial sections $s_f$ now
  follow directly from Example~\ref{ex:7}; this gives also
  $\eta_i = s_{1_i}$. For the target map, it suffices by the
  description in Definition~\ref{def:7} to show that
  $\tau_{ij}s_f = \O(f) \colon \O(i) \rightharpoonup \O(j)$ for all
  $f \in \A(i,j)$, which is so since
  $s_f^\ast \tau_{ij}^\ast(d) = \overline{df}$. Finally, for the
  multiplication, we observe that $X_{jk} \times_{X_j} X_{ij}$ is, as
  in the proof of Proposition~\ref{prop:4}, a glueing of the local
  atlas
  \begin{equation*}
    \bigl(\sheq f h \overline{\sheq g k f} : (f, g), (h,k) \in \A(i,j)
    \times \A(j,k)\bigr)
  \end{equation*}
  with associated family of local sections $\tau_{ij}^\ast(s_g) s_f$.
  A short calculation similar in nature to that given above shows
  that~\eqref{eq:50} is a valid description of this local glueing, and
  that in these terms, the section $\tau_{ij}^\ast(s_g) s_f$ is given
  by $\theta \mapsto \theta_{(f,g)}$. Given this, the description of
  $\mu_{ijk}$ is validated by observing that it satisfies the
  condition $\mu_{ijk}\tau_{ij}^\ast(s_g) s_f = s_{gf}$ which uniquely
  characterises it in Definition~\ref{def:7}.
\end{proof}


\subsection{The Lawson--Lenz correspondence}
\label{sec:laws-lenz-corr}

In~\cite{Lawson2013Pseudogroups}, Lawson and Lenz describe a Galois
adjunction between join inverse monoids and \'etale topological
groupoids, inducing an equivalence between the categories of
fixpoints. They call these fixpoints the \emph{spatial} join inverse
monoids and the \emph{sober} \'etale groupoids; the nomenclature draws
on the Galois adjunction between topological spaces and locales, whose
fixpoints are the sober spaces and the spatial locales.
\begin{Defn}
  \label{def:45}
  A topological space is \emph{sober} if each irreducible closed set
  is the closure of a unique point. A \emph{point} of a locale $X$ is
  a function $p \colon X \rightarrow 2 = \{\bot \leqslant \top\}$
  which preserves finite meets and all joins. A locale $X$ is
  \emph{spatial} if whenever $x \neq y \in X$ there is a point $p$
  with $p(x) \neq p(y)$.
\end{Defn}

We now explain how the Lawson--Lenz adjunction can be re-derived
from our Theorem~\ref{thm:10} by composing the equivalence given there
with the space--locale adjunction.
We begin by describing this latter adjunction in a manner which is
amenable for our applications.

\begin{Lemma}
  \label{lem:12}
  The fundamental functor
  $\O \colon \cat{Top}_p \rightarrow \cat{Loc}_p$ has a right adjoint
  $\mathrm{pt} \colon \cat{Loc}_p \rightarrow \cat{Top}_p$ in
  $\mathrm{jr}\cat{Cat}$. The underlying adjunction of categories
  $\O \dashv \mathrm{pt}$ is Galois, and its fixpoints are the sober
  spaces, respectively, the spatial locales.
\end{Lemma}
\begin{proof}
  Let $X$ be a locale. We define $\mathrm{pt}(X)$ to be the space
  of points of $X$ endowed with the topology with open sets
  $[x] = \{p \in \mathrm{pt}(X) : p(x) = \top\}$ for each $x \in X$.
  We define a total locale map
  $\varepsilon_X \colon \O(\mathrm{pt}(X)) \rightarrow X$ by
  $\varepsilon_X^\ast(x) = [x]$. For any space $Y$ and partial locale
  map $f \colon \O(Y) \rightharpoonup X$, we have a partial continuous map
  $g \colon Y \rightharpoonup \mathrm{pt}(X)$ defined on
  $y \in f^\ast(\top) \subseteq Y$ by $g(y)(x) = \top$ just when
  $y \in f^\ast(x)$; this is the \emph{unique} map
  with $\varepsilon_X \circ \O(g) = f$, and so
  $\O \dashv \mathrm{pt}$ in $\cat{Cat}$.

  By construction, the counit $\varepsilon$ is total; by inspection,
  the unit $\eta$ is also total, and the functor
  $\mathrm{pt} \colon \cat{Loc}_p \rightarrow \cat{Top}_p$ preserves
  restriction. It follows that $\O \dashv \mathrm{pt}$ in
  $\mathrm{r}\cat{Cat}$. Finally, by observing that
  \begin{equation*}
    \cd[@-1em@C-0.5em]{
      {\cat{Loc}_p} \ar[rr]^-{\mathrm{pt}} \ar@{=}[dr] & \ltwocello{d}{\varepsilon} &
      {\cat{Top}_p} \ar[dl]^-{\O} \\ &
      {\cat{Loc}_p}
    }
  \end{equation*}
  is a triangle in
  ${\mathrm{jr}\cat{Cat}\mathord{\mkern1mu/\!/_{\!h}\mkern1.5mu}
    \cat{Loc}_p}$ and applying Lemma~\ref{lem:33}, we see that
  $\mathrm{pt}$ is join-preserving, so that $\O \dashv \mathrm{pt}$
  in $\mathrm{jr}\cat{Cat}$. To see that the underlying adjunction is
  Galois, and the fixpoints are as described, see, for example,~\cite[\sec
  II.1.7]{Johnstone1982Stone}.
\end{proof}

\begin{Cor}
  \label{cor:9}
  There is a Galois adjunction
  \begin{equation}\label{eq:52}
    \cd{
      {\mathrm{jr}\cat{Cat}\mathord{\mkern1mu/\!/_{\!h}\mkern1.5mu}
        \cat{Top}_p} \ar@<-4.5pt>[r]_-{U} \ar@{<-}@<4.5pt>[r]^-{W} \ar@{}[r]|-{\bot} &
      {\mathrm{jr}\cat{Cat}}
    }
  \end{equation}
  where $U$ forgets the projection down to $\cat{Top}_p$, and where
  $W$ sends $\A$ to $(\A, \mathrm{pt} \circ \O_\A)$ and sends
  $F \colon \A \rightarrow \B$ to
  $(F, \mathrm{pt} \circ \alpha)
  $, where $\alpha$ is defined as in Corollary~\ref{cor:8}.
\end{Cor}
\begin{proof}
  Compose the equivalence
  of Corollary~\ref{cor:8} with the induced Galois adjunction
  $\O \circ (\thg) \colon 
  {\mathrm{jr}\cat{Cat}\mathord{\mkern1mu/\!/_{\!h}\mkern1.5mu}
    \cat{Top}_p} \leftrightarrows
  {\mathrm{jr}\cat{Cat}\mathord{\mkern1mu/\!/_{\!h}\mkern1.5mu}
    \cat{Loc}_p} \colon \mathrm{pt} \circ (\thg)$.
\end{proof}

Taking this together with the case $\C = \cat{Top}_p$ of our main
result, we recover the Lawson--Lenz correspondence and its
generalisations. 

\begin{Thm}
  \label{thm:12}
  There are Galois adjunctions
  \begin{equation*}
    \cd[@C-0.5em]{
      {\mathrm{pe}\cat{Cat}_c(\cat{Top}_p)} \ar@<-4.5pt>[r]_-{U \Phi}
      \ar@{<-}@<4.5pt>[r]^-{\Psi W} \ar@{}[r]|-{\bot} &
      {\mathrm{jr}\cat{Cat}}
    } \quad \text{and} \quad 
    \cd[@C-0.5em]{
      {\mathrm{pe}\cat{Gpd}_c(\cat{Top}_p)} \ar@<-4.5pt>[r]_-{U \Phi_g}
      \ar@{<-}@<4.5pt>[r]^-{\Psi W} \ar@{}[r]|-{\bot} &
      {\mathrm{ji}\cat{Cat}}
    }
  \end{equation*}
  together with their obvious restrictions to the one-object case. The
  fixpoints to each side are those $X \in
  \mathrm{pe}\cat{Cat}_c(\cat{Top}_p)$ with  sober spaces of objects,
  and those $\A \in \mathrm{jr}\cat{Cat}$ with spatial locales of
  restriction idempotents.
\end{Thm}

As we have already discussed in Section~\ref{sec:local-hyperc-maps},
even the one-object inverse case of this theorem is more general
than~\cite{Lawson2013Pseudogroups}, due to the more generous notions
of morphism to each side. To recapture the precise form of the
Lawson--Lenz equivalence, we may employ Theorem~\ref{thm:10} in place of
Theorem~\ref{thm:7}.

\begin{Rk}
  \label{rk:4}
  In~\cite{Lawson2013Pseudogroups}, the \'etale groupoid associated to
  a join inverse monoid $S$ is described in terms of \emph{completely
    prime filters} on $S$: subsets of $S$ which are upwards closed and
  downwards directed, and which contain a join $\bigvee_{i \in I} s_i$
  precisely when they contain at least one of the $s_i$'s. We can
  recover this description, and its generalisation to the other cases
  of our correspondence, by combining the explicit description of
  $\mathrm{pt} \colon \cat{Loc}_p \rightarrow \cat{Top}_p$ from
  Lemma~\ref{lem:12} and the description of local glueings in
  $\cat{Top}_p$ from Example~\ref{ex:5}.
\end{Rk}

\subsection{The Ehresmann--Schein--Nambooripad correspondence}
\label{sec:ehresm-sche-namb}
The correspondence we consider next was first made explicit by Lawson
in~\cite[Chapter~4]{Lawson1998Inverse}, bringing together
contributions by the three named authors. Lawson's version correlates
inverse semigroups with a certain category of ordered groupoids; the
version we state here, for inverse categories, is due to DeWolf and
Pronk~\cite{DeWolf2018The-Ehresmann-Schein-Nambooripad}.

\begin{Defn}
  \label{def:49}
  A partite internal groupoid in $\cat{Pos}_p$ is \emph{inductive} if
  each source (and hence each target) map is a discrete fibration, and
  each poset of objects has finite meets. We write
  $\mathrm{p}\cat{Ind}\cat{Gpd}$ for the category of partite inductive
  groupoids, where maps are partite internal functors whose object
  part preserves finite meets.
\end{Defn}
\begin{Defn}
  \label{def:50}
  If $\A$ is an inverse category with object-set $I$, then the
  $I$-partite inductive groupoid $\G \A$ is defined as follows.
  The poset of objects $\G \A_i$ is $(\O(i), \leqslant)$ while the
  poset of arrows $\G \A_{ij}$ is $(\A(i,j), \leqslant)$. The source
  and target maps are given by $\sigma_{ij}(f) = f^\ast f$ and
  $\tau_{ij}(f) = ff^\ast$, while the identity and composition maps
  $\iota_i$ and $\mu_{ijk}$ are given by identities and composition in $\A$.
\end{Defn}
\begin{Thm}(\cite[Theorem~3.16]{DeWolf2018The-Ehresmann-Schein-Nambooripad})
  \label{thm:16}
  The assignment $\A \mapsto \G(\A)$ is the action on objects of an
  equivalence of categories
  $\G \colon \mathrm{i}\cat{Cat} \rightarrow \mathrm{p}\cat{Ind}\cat{Gpd}$.
\end{Thm}
Our goal is to explain how this equivalence can be obtained from our
main correspondence. We should say up front that the proof we describe
has many more moving parts than the approach
of~\cite{Lawson1998Inverse,
  DeWolf2018The-Ehresmann-Schein-Nambooripad}; nonetheless, we
consider it illuminating to see how this correspondence fits into our
framework. In doing so, we make use of the following mild refinement
 of our main result.

\begin{Defn}
  \label{def:46}
  Let $\C$ be a join restriction category with local glueings, and let $\C'$
  be a subcategory of the category of total maps in $\C$. We write
  $\rcatt$ for the subcategory of $\rcat$ with:
  \begin{itemize}
  \item \textbf{Objects}: those $P \colon \A \rightarrow \C$ for which
    each object $Pi$ lies in $\C'$;
  \item \textbf{Morphisms}: those $(F, \alpha) \colon (\A, P)
    \rightarrow (\B, Q)$ for which each component $\alpha_i \colon QFi
    \rightarrow Pi$ lies in $\C'$.
  \end{itemize}
  On the other hand, we write $\partcc$ (resp.,
  $\mathrm{p}\cat{Cat}(\C', \C)$) for the subcategory of
  $\partc$ (resp., $\mathrm{p}\cat{Cat}(\C)$) with:
  \begin{itemize}
  \item \textbf{Objects}: those partite internal categories
    $\mathbb{A}$ for which each $A_i$ lies in $\C'$;
  \item \textbf{Morphisms}: those partite internal cofunctors (resp., internal
    functors) whose object mappings all lie in $\C'$.
  \end{itemize}
\end{Defn}
\begin{Thm}
  \label{thm:13}
  The adjunction~\eqref{eq:7} restricts back to a Galois adjunction as
  to the left below, and taking fixpoints yields an equivalence as to
  the right.
  \begin{equation*}
  \cd{
    {\partcc} \ar@<-4.5pt>[r]_-{\Phi} \ar@{<-}@<4.5pt>[r]^-{\Psi} \ar@{}[r]|-{\bot} &
    {\rcatt}
  } \qquad
  \cd{
    {\partee} \ar@<-4.5pt>[r]_-{\Phi} \ar@{<-}@<4.5pt>[r]^-{\Psi} \ar@{}[r]|-{\sim} &
    {\hconn}\rlap{ .} 
  }
  \end{equation*}
\end{Thm}
\begin{proof}
  Direct from the constructions giving Theorem~\ref{thm:4} and
  Theorem~\ref{thm:2}.
\end{proof}

We begin by recalling from Example~\ref{ex:12} that the local
homeomorphisms in $\cat{Pos}_p$ are \emph{exactly} the discrete
fibrations. Thus, writing $\cat{Msl} \subseteq \cat{Pos}_p$ for the
subcategory of meet-semilattices and finite-meet-preserving total
maps, we have 
\begin{equation}\label{eq:53}
\mathrm{p}\cat{Ind}\cat{Gpd} = \mathrm{pe}\cat{Gpd}(\cat{Msl},
\cat{Pos}_p)\rlap{ .}
\end{equation}

We now transform the right-hand side of this equality by way of the
fundamental functor $\O \colon \cat{Pos}_p \rightarrow \cat{Loc}_p$.
This takes a poset $P$ to the locale $\O(P)$ of downsets in $P$, and
takes a poset map $f \colon P \rightarrow Q$ to the locale map
$\O(f) \colon \O(P) \rightarrow \O(Q)$ given by
$\O(f)^\ast(B \subseteq Q) = f^{-1}(B) \subseteq P$. It is easy to see
that $\O$ is faithful and is full on isomorphisms; it therefore
establishes an equivalence between $\cat{Pos}_p$ and its replete image
in $\cat{Loc}_p$. Since, moreover, $\O$ is hyperconnected, it induces
as in Remark~\ref{rk:5} an equivalence between local homeomorphisms
over $P \in \cat{Pos}_p$ and over $\O(P) \in \cat{Loc}_p$. It follows
that a partite internal groupoid in $\cat{Loc}_p$ is in the replete
image of $\O$ just when each of its objects of objects is so. Thus,
writing $\O(\cat{Msl})$ for the replete image of
$\cat{Msl} \subseteq \cat{Pos}_p$ in $\cat{Loc}_p$, we conclude that
the action of $\O$ induces an equivalence of categories
\begin{equation}\label{eq:56}
  \O \colon \mathrm{pe}\cat{Gpd}(\cat{Msl},
\cat{Pos}_p) \rightarrow \mathrm{pe}\cat{Gpd}(\O(\cat{Msl}),
\cat{Loc}_p)\rlap{ .}
\end{equation}


The next step is delicate: we transform the category of partite
groupoids and \emph{functors} to the right above into a category of
partite groupoids and \emph{cofunctors}. To do so, we must explicitly
identify the subcategory $\O(\cat{Msl}) \subseteq \cat{Loc}_p$.

\begin{Defn}
  \label{def:53}
  An element $\ell$ of a locale $L$ is \emph{supercompact} if
  $\ell \leqslant \bigvee D$ implies $\ell \leqslant d$ for some
  $d \in D$. We call $L$ \emph{supercoherent} if the supercompact
  elements form a meet-semilattice, and each $\ell \in L$ is a join of
  supercompact elements. We write $\mathrm{sc}\cat{Loc}$ for the
  category whose objects are supercoherent locales, and whose maps
  $f \colon L \rightarrow M$ are total locale maps for which $f^\ast$
  preserves supercompactness.
\end{Defn}
The objects of $\O(\cat{Msl})$ are exactly the supercoherent locales.
On the other hand, a morphism $f \colon M \rightarrow L$ of
$\O(\cat{Msl})$ is a total locale map for which
$f^\ast \colon L \rightarrow M$ has a finite-meet-preserving left
adjoint $f_! \colon M \rightarrow L$. This $f_!$ is then the inverse
image of a supercoherent locale map $f^\vee \colon M \rightarrow L$,
and every supercoherent map arises in this way. We thus have an
identity-on-objects isomorphism of categories
$(\thg)^\vee \colon \O(\cat{Msl})^\mathrm{op} \rightarrow
\mathrm{sc}\cat{Loc}$. We claim that this induces an
identity-on-objects isomorphism of categories
\begin{equation}\label{eq:55}
  (\thg)^\vee \colon \mathrm{pe}\cat{Gpd}(\O(\cat{Msl}),
  \cat{Loc}_p) \rightarrow \mathrm{pe}\cat{Gpd}_c(\mathrm{sc}\cat{Loc},
\cat{Loc}_p)\rlap{ .}
\end{equation}

The key observation is as follows. Given a total locale map
$f \colon L \rightarrow M$, pullback along $f$ gives
a functor
$\Delta_f \colon \cat{Loc}_p \mathord{\mkern1.5mu/_{\!\ell
    h}\mkern1.5mu}M \rightarrow \cat{Loc}_p
\mathord{\mkern1.5mu/_{\!\ell h}\mkern1.5mu}L$. When $f$
is a map of $\O(\cat{Msl})$, we also have the adjoint map
$f^\vee \colon M \rightarrow L$ and in this case, we
have that
$\Delta_{f^\vee} \dashv \Delta_f$.
Indeed, on identifying
${\cat{Loc}_p \mathord{\mkern1.5mu/_{\!\ell h}\mkern1.5mu}M}$ with
$\cat{Sh}(M)$ via
Theorem~\ref{thm:1}, this follows
from the $2$-functoriality~\cite[\sec
C1.4]{Johnstone2002Sketches} of the assignation
$M \mapsto \cat{Sh}(M)$.

In concrete terms, the
adjointness~$\Delta_{f^\vee} \dashv \Delta_f$ states that if $p \colon A \rightarrow L$
and $q \colon B \rightarrow M$ are local homeomorphisms, then there is
a bijection between total locale
maps $g$ as to the left, and total locale maps $\tilde g$ as to the right in:
\begin{equation*}
  \cd{
    {A} \ar[r]^-{g} \ar[d]_{p} &
    {B} \ar[d]^{q} \ar@{}[drr]|-{\displaystyle\leftrightsquigarrow}& &
    {A} \ar@{<-}[r]^-{} \ar[d]_{p} &
    {(f^\vee)^\ast(A)} \ar[d]^{} \pullbackcorner[dl] \ar[r]^-{\tilde
      g} & B\rlap{ ;} \ar[dl]^-{q} \\
    {L} \ar[r]^-{f} &
    {M} & &
    {L} \ar@{<-}[r]^-{f^\vee} &
    {M}
  }
\end{equation*}
using this, we can describe the isomorphism~\eqref{eq:55} as sending a
partite internal functor $F \colon \mathbb{A} \rightarrow \mathbb{B}$
to the partite internal cofunctor $F^\vee \colon \mathbb{A}
\rightsquigarrow \mathbb{B}$ with action on
components, objects and arrows  $F^{\vee}i = Fi$,  $(F^{\vee})_i = (F_i)^\vee$,
and  $(F^\vee)_{ij} = \smash{\widetilde{F_{ij}}}$.

We now arrive at the final step, which is to exhibit the right-hand
side of~\eqref{eq:55} as equivalent to the category of inverse
categories. The starting point is the \emph{join completion} of an
inverse category.
\begin{Defn}
  \label{def:47}
  Let $\I$ be an inverse category. Its \emph{join completion}
  $\mathrm{j}(\I)$ is the join inverse category with the same objects,
  and with maps $S \colon A \rightarrow B$ being downclosed
  bicompatible families $S \subseteq \I(A,B)$.
\end{Defn}
That $\mathrm{j}(\I)$ is a join inverse category can be verified by
following exactly the same argument as
in~\cite[Theorem~23]{Lawson1998Inverse}; while by following Theorem~24
of \emph{loc.~cit.}, we see that $\mathrm{j(\I)}$ is the free join
restriction category on $\I$, in the sense of providing the value at
$\I$ of a left $2$-adjoint
$\mathrm{j} \colon \mathrm{i}\cat{Cat} \rightarrow
\mathrm{ji}\cat{Cat}$ to the obvious forgetful $2$-functor.

We now wish to characterise the objects and morphisms in the image of
$\mathrm{j}$. We do so by following the approach
of~\cite[Section~3]{Lawson2013Pseudogroups}. 

\begin{Defn}
  \label{def:58}
  A join inverse category $\C$ is called \emph{supercoherent} if
  $\O(A)$ is a supercoherent locale for all $A \in \C$. A
  functor $F
  \colon \C \rightarrow \D$ between supercoherent join inverse
  categories is called \emph{supercoherent} if each function $\O(A) \rightarrow \O(FA)$ is the inverse
  image of a supercoherent locale map. We write
  $\mathrm{scji}\cat{Cat}$ for the subcategory of
  $\mathrm{ji}\cat{Cat}$ determined by the supercoherent objects and morphisms.
\end{Defn}
\begin{Prop}
  \label{prop:36}
  The functor
  $\mathrm{j} \colon \mathrm{i}\cat{Cat} \rightarrow
  \mathrm{ji}\cat{Cat}$ lands inside the subcategory
  $\mathrm{scji}\cat{Cat}$, and when restricted to this codomain
  yields an equivalence
  \begin{equation}\label{eq:51}
      \cd{
    {\mathrm{scji}\cat{Cat}} \ar@<-4.5pt>[r]_-{K_0} \ar@{<-}@<4.5pt>[r]^-{\mathrm{j}} \ar@{}[r]|-{\sim} &
    {\mathrm{i}\cat{Cat}} \rlap{ ,}
  } 
\end{equation}
where $K_0$ sends $\C$ to the subcategory $K_0(\C)$ composed of all
the objects, and all maps $s \colon A \rightarrow B$ for which
$s^\ast s \in \O(A)$ is supercompact.
\end{Prop}
\begin{proof}
  On substituting ``coherent'' for ``supercoherent'', this is,
  \emph{mutatis mutandis}, the argument
  of~\cite[Lemma~3.3,~Lemma~3.4~\&~Proposition~3.5]{Lawson2013Pseudogroups}.
\end{proof}

Now, it is direct from the definition of supercoherence
that the equivalence to the right of~\eqref{eq:61} restricts to an
equivalence
${\mathrm{ji}\cat{Cat}\mathord{\mkern1mu/\!/_{\!h}\mkern1.5mu}
  (\mathrm{sc}\cat{Loc}, \cat{Loc}_p)} \simeq
{\mathrm{scji}\cat{Cat}}$; while from
 Theorem~\ref{thm:13} we have 
${\mathrm{pe}\cat{Gpd}
  (\mathrm{sc}\cat{Loc}, \cat{Loc}_p)} \simeq {\mathrm{ji}\cat{Cat}\mathord{\mkern1mu/\!/_{\!h}\mkern1.5mu}
  (\mathrm{sc}\cat{Loc}, \cat{Loc}_p)}$. Putting these together
with~\eqref{eq:51}, we obtain the desired equivalence
\begin{equation*}
      \cd{
    {\mathrm{pe}\cat{Gpd}
  (\mathrm{sc}\cat{Loc}, \cat{Loc}_p)} \ar@<-4.5pt>[r]_-{K_0U\Phi_g}
\ar@{<-}@<4.5pt>[r]^-{\Psi W\mathrm{j}} \ar@{}[r]|-{\sim} &
    {\mathrm{i}\cat{Cat}} \rlap{ .}
  } 
\end{equation*}

Combining this with the equality~\eqref{eq:53}, the
equivalence~\eqref{eq:56} and the isomorphism~\eqref{eq:55}
reconstructs the Ehresmann--Schein--Nambooripad correspondence
$\mathrm{p}\cat{Ind}\cat{Gpd} \simeq \mathrm{i}\cat{Cat}$. However, it
is perhaps more illuminating to observe that we have a
pseudo-commuting triangle of equivalences
\begin{equation}\label{eq:60}
  \cd[@!C@-0.5em@C-5em]{
    & {\mathrm{i}\cat{Cat}} \ar[dl]_-{\G} \ar[dr]^-{\Psi V \mathrm{j}} \\
    {\mathrm{p}\cat{Ind}\cat{Gpd}} \ar[rr]_-{(\thg)^\vee \circ \O} & &
    {\mathrm{pe}\cat{Gpd}_c(\mathrm{sc}\cat{Loc}, \cat{Loc}_p)}\rlap{ .}
  }
\end{equation}

On objects, this says that the (partite) inductive groupoid associated
to an inverse category is really a presentation of the associated
(partite) supercoherent localic groupoid. Equivalently, since
supercoherent locales are spatial, it is a presentation of the
associated supercoherent \emph{topological} groupoid---which,
from the definition of
$\mathrm{pt} \colon \cat{Loc}_p \rightarrow \cat{Top}_p$, we see to be
Paterson's universal groupoid of an inverse semigroup, described as
in~\cite{Lawson2013The-etale} as a groupoid of filters.

To conclude this section, let us comment briefly on the
\emph{non-groupoidal} analogue of the theory presented above.
In~\cite{Gould2010Restriction}, Gould and Hollings describe a version
of the Ehresmann--Schein--Nambooripad theorem which, to the one side
replaces inverse monoids by restriction monoids. To the other side,
their result replaces inductive groupoids with what they call
\emph{inductive constellations}.

Although it appears that the details are even more delicate, it
appears that, in exactly the same way that inductive groupoids present the
\'etale localic groupoids associated to free join inverse monoids,
inductive constellations present the source-\'etale localic categories
associated to free join restriction categories. We will leave the
details of this claim to future work.

%

\subsection{Completion processes}
\label{sec:completion-processes}

Above we have discussed how our result reconstructs various
equivalences from the literature. In this final section, we describe
how we may exploit the larger adjunctions to construct various
\emph{completions}.

\subsubsection{Full monoids}
\label{sec:full-inverse}
Let $M$ be a (discrete) monoid which acts by continuous maps on a
topological space $X$. We can view $M$ as a one-object restriction
category $\M$ wherein every map is total, and view the action of $M$
on $X$ as a restriction functor
$\alpha \colon \M \rightarrow \cat{Top}_p$ sending the unique object
of $\M$ to the space $X$.

Applying the functor $\Psi$ to
$\alpha \colon \M \rightarrow \cat{Top}_p$ yields a source-\'etale
topological category whose space of objects is $X$, and whose space of
arrows is the glueing of the $M$-object local atlas $\varphi$ with
$\varphi_{mn} = 1_X$ if $m = n$ and $\varphi_{mn} = \bot$ otherwise.
It is easy to see from Example~\ref{ex:5} that this glueing is the
product space $M \times X$, where $M$ is endowed with the discrete
topology. The source and target maps
$\sigma,\tau \colon M \times X \rightarrow X$ are given by
$\sigma(m,x) = x$ and $\tau(m,x) = m \cdot x$; the identity map is
$\iota(x) = (e_M,x)$; while composition is given by
$\mu\bigl((n,y),(m,x)\bigr) = (nm, x)$.

So $\Psi(\alpha)$ is the well-known \emph{action category} of $M$
acting on $X$. It follows that $\Phi\Psi(\alpha)$ is the join
restriction monoid whose elements are pairs $(U,s)$, where
$U \subseteq X$ and $s \colon U \rightarrow M$ is a continuous (i.e.,
locally constant) function. The identity is $(X, x \mapsto e)$, while
the composite $(V,t)(U,s)$ is given by $(W,u)$, where
\begin{equation*}
  W = \{x \in X : x \in U \text{ and } s(x) \cdot x \in V\} \qquad
  \text{and} \qquad 
  u(x) = t(s(x) \cdot x) s(x)\rlap{ .}
\end{equation*}
In particular, the monoid of total elements in $\Phi\Psi(\alpha)$
comprises all continuous functions $s \colon X \rightarrow M$ under
the multiplication $(t \circ s)(x) = t(s(x) \cdot x) s(x)$. When $M$
acts faithfully on $X$, we can identify this monoid with a submonoid
of $\I(X)$, comprising all endomorphisms of $M$ which act locally like
an element of $M$; we might reasonably call this the \emph{topological
  full monoid} of $M$ in $X$.

If now $M$ is a group, then $\M$ is an inverse category, and so
$\Psi(\alpha)$ is by Theorem~\ref{thm:3} a groupoid. In this case, we
have not only the \'etale join restriction monoid $\Phi\Psi(\alpha)$,
but also the join inverse monoid $\Phi_g\Psi(\alpha)$; this has as
elements those $(U,s)$ for which the mapping $x \mapsto s(x) \cdot x$
is an open injection. Finally, the group of units of
$\Phi_g\Psi(\alpha)$ may be identified with the group of those
continuous functions $s \colon X \rightarrow M$ for which
$x \mapsto s(x) \cdot x$ is a homeomorphism. Like before, when $M$
acts faithfully on $X$, this yields the topological full group of $M$
in $X$.

Of course, everything described above works equally well for actions
on spaces by discrete categories or groupoids; we leave the
adaptations to the reader.

\subsubsection{Relative join completions}
\label{sec:cover-join-restr}
\cite{Lawson2013Pseudogroups} introduces the notion of a
\emph{coverage} $C$ on an inverse semigroup $S$, and describes in
particular cases the \emph{relative join completion} of $S$ with
respect to $C$; this is the free join inverse semigroup admitting a
map from $S$ which sends each cover in $C$ to a join.

The purpose of this section is to construct relative join completions
in greater generality by exploiting our main result. We begin with the
necessary definitions.

\begin{Defn}
  \label{def:54}
  A \emph{sieve} on a map $f \colon i \rightarrow j$ of a restriction
  category $\A$ is a down-closed subset of
  $\mathord \downarrow g = \{g \in \A(i,j) : g \leqslant f\}$. A \emph{coverage} on $\A$ is
  given by specifying, for each $f \in \A(i,j)$, a collection $C(f)$
  of sieves on $f$, called \emph{covers}, satisfying the following
  axioms:
  \begin{enumerate}[(i)]
  \item If $f \in
    \A(i,j)$ and $g \in \A(j,k)$, then $X \in C(g)$ implies $Xf \in
    C(gf)$.
  \item $X \in C(f)$ if and only if $\overline{X} \in C(\overline{f})$.
  \end{enumerate}
\end{Defn}
Our definition of coverage is modelled after the notion of coverage on
a meet-semilattice in~\cite{Johnstone1982Stone}, and indeed reduces to
it in the case where $\C$ is the one-object join restriction category
$\Sigma M$ associated to a meet-semilattice $M$. We omit some of the
additional clauses for a coverage listed
in~\cite{Lawson2013Pseudogroups}, which merely impose inessential
additional ``saturation'' conditions on the class of covers. We will,
however, note the following consequence of the axioms, which will be
useful later.
\begin{Lemma}
  \label{lem:40}
  Let $C$ be a coverage on $\A$. If $f \in \A(i,j)$ and $g \in
  \A(j,k)$, then $X \in C(f)$ implies $gX \in C(gf)$.
\end{Lemma}
\begin{proof}
  We first show that $\overline{g} X \in C(\overline{g}f)$. First,
  since $\overline{g}f = f \overline{gf}$, we have by axiom (i) that
  $X\overline{gf} \in C(\overline{g}f)$. But $X\overline{gf} = \{x
  \overline{gf} : x \in X\}$, and since $x \leqslant f$ we have
  $x\overline{gf} = f \overline{x}\overline{gf} = f \overline{gf}
  \overline{x} =\overline{g}f\overline{x} = \overline{g}x$; whence $X
  \overline{gf} = \overline{g}X \in C(\overline{g}f)$ as claimed.

  We now show $gX \in C(gf)$. Since $\overline{g} X \in
  C(\overline{g}f)$ we have by axiom (ii) that $\overline{\overline{g}
    X} \in C(\overline{\overline{g} f})$, whence $\overline{gX} \in
  C(\overline{gf})$, whence by axiom (ii) again, $gX \in C(gf)$.
\end{proof}

\begin{Defn}
  \label{def:56}
  Let $\A$ be a restriction category with a coverage $C$. If $\B$ is a
  join restriction category, then a restriction functor
  $F \colon \A \rightarrow \B$ is called a \emph{cover-to-join map}
  if, for each $X \in \C(f)$, we have $\bigvee_{x \in X} Fx = Ff$. By
  a \emph{relative join completion} of $\A$ with respect to $C$, we
  mean a join restriction category $\mathrm{j}_C(\A)$ endowed with a
  cover-to-join functor $\eta \colon \A \rightarrow \mathrm{j}_C(\A)$,
  such that any other cover-to-join map $F \colon \A \rightarrow \B$
  factors through $\eta$ via a unique join restriction functor
  $F' \colon \mathrm{j}_C(\A) \rightarrow \B$.
\end{Defn}

Our objective is to construct relative join completions by exploiting
the adjunction of our main theorem together with the known
construction of the relative join completion for meet-semilattices. We
begin by discussing the latter. As noted above, a coverage in the
sense of~\cite{Johnstone1982Stone} on a meet-semilattice $M$ is the
same as a coverage in our sense on $\Sigma M$, so that by~\cite[\sec
II.2.11]{Johnstone1982Stone} we have:
\begin{Prop}
  \label{prop:37}
  Let $C$ be a coverage on a meet semilattice $M$. The relative join
  completion of $\Sigma M$ with respect to $C$ is
  $\Sigma \bigl(C\text-\mathrm{Idl}(M)\bigr)$, where
  $C\text-\mathrm{Idl}(M)$ is the locale of \emph{$C$-closed ideals}
  in $M$, whose elements are down-closed subsets $D \subseteq M$ with
  the property that $A \subseteq D$ and $A \in C(x)$ imply $x \in D$.
  The universal cover-to-join map
  $\eta \colon \Sigma M \rightarrow
  \Sigma\bigl(C\text-\mathrm{Idl}(M)\bigr)$ sends $m \in M$ to the
  $C$-closed ideal generated by $m$.
\end{Prop}
We now exploit this result to construct a variant of the fundamental
functor for a restriction category $\A$ endowed with a coverage $C$.
\begin{itemize}
\item For each object $i \in \A$, the coverage $C$ restricts to a coverage of the
same name on the meet-semilattice $\O(i)$. We write $\O_C(i)$ for
the locale $C\text-\mathrm{Idl}\bigl(\O(i)\bigr)$.
\item For each map $f \colon i \rightarrow j$ in $\A$, we claim there
  is a unique partial locale map
  $\O_C(f) \colon \O_C(i) \rightharpoonup \O_C(j)$ whose inverse image
  map renders commutative the square of binary-meet-preserving
  functions to the left in:
  \begin{equation}\label{eq:63}
    \cd[@C+1em]{
      {\O(j)} \ar[r]^-{\overline{(\thg)f}} \ar[d]_{\eta} &
      {\O(i)} \ar[d]^{\eta} &
            {\O(j)} \ar[r]^-{\overline{(\thg)f}} \ar[d]_{\eta} &
            {\mathord \downarrow \overline{f}} \ar[d]^{\eta} \\
      {\O_C(j)} \ar[r]^-{\O_C(f)^\ast} &
      {\O_C(i)} &
      {\O_C(j)} \ar@{-->}[r] &
      {\mathord \downarrow \eta(\overline{f})}\rlap{ .}
    }
  \end{equation}
  This is equally to say there is a unique \emph{total} locale
  map whose inverse image renders commutative the right square of
  finite-meet-preserving functions. By Proposition~\ref{prop:37}, it
  suffices for this to show that the upper right composite is a
  cover-to-join map. But for all $e \in \O(j)$ and $X \in C(e)$, the
  coverage axioms imply that $\overline{Xf} \in C(\overline{ef})$;
  since $\eta$ is a cover-to-join map, we conclude that
  $\eta(\overline{ef}) = \bigvee_{x \in X} \eta(\overline{xf})$ in
  $\O_C(i)$, and hence also in
  ${\mathord \downarrow \eta(\overline{f})}$, as required.
\end{itemize}

It follows from the unicity in~\eqref{eq:63} that the above
assignations underlie a restriction functor
$\O_C \colon \A \rightarrow \cat{Loc}_p$. We now show that $\O_C$ has a
universal characterisation similar in spirit to
Proposition~\ref{prop:29} above.
\begin{Prop}
  \label{prop:38}
  $\O_C \colon \A \rightarrow \cat{Loc}_p$ is a terminal object in the
  category whose objects are cover-to-join maps
  $\A \rightarrow \cat{Loc}_p$ and whose maps are total
  transformations.
  Furthermore, any $F \colon \A \rightarrow \cat{Loc}_p$ which does admit
  a total transformation $F \Rightarrow \O_C$ is necessarily a
  cover-to-join map.
\end{Prop}
\begin{proof}
  First we show that $\O_C$ is indeed a cover-to-join map. Let
  $f \in \A(i,j)$ and $X \in C(f)$. We must show that
  $\O_C(f) = \bigvee_{x \in X} \O_C(x)$. By unicity in~\eqref{eq:63},
  it suffices for this to exhibit an equality of
  binary-meet-preserving maps
  \begin{equation*}
    \eta \circ \overline{(\thg)f} = \bigvee_{x \in X} \eta \circ
    \overline{(\thg)x} \colon \O(j) \rightarrow \O_C(i)\rlap{ .}
  \end{equation*}
  Evaluating at $e \in \O(j)$ and using the fact that
  $\eta$ is a cover-to-join map, it suffices to show that
  $\overline{eX} \in C(\overline{ef})$---which follows from $X \in
  C(f)$ using the axioms and Lemma~\ref{lem:40}.

  We now show that $\O_C$ is terminal among cover-to-join maps.
  Indeed, suppose that $F \colon \A \rightarrow \cat{Loc}_p$ is
  another such. If we were to have a total natural transformation
  $\alpha \colon F \rightarrow \O_C$, then for every $i \in \A$ and
  $e \in \O(e)$, we would have the naturality square of inverse image
  mappings as to the left in:
  \begin{equation*}
    \cd{
      {Fi} \ar@{<-}[r]^-{\alpha_i^\ast} \ar@{<-}[d]_{Fe^\ast} &
      {\O_C(i)} \ar@{<-}[d]^{\O_C(e)^\ast} &
      {Fe^\ast(\top) = \alpha_i^\ast(\eta(e))} \ar@{<-|}[r]^-{\alpha_i^\ast} \ar@{<-|}[d]_{Fe^\ast} &
      {\eta(e)} \ar@{<-|}[d]^{\O_C(e)^\ast} \\
      {Fi} \ar@{<-}[r]^-{\alpha_i^\ast} &
      {\O_C(i)} &
      {\top} \ar@{<-|}[r]^-{\alpha_i^\ast} &
      {\top}
    }
  \end{equation*}
  and so evaluating at the top element of $\O_C(i)$, the equality to
  the right. So the join- and finite-meet-preserving map
  $\alpha_i^\ast$ must map $\eta(e)$ to $Fe^\ast(\top)$ for each
  $e \in \O(i)$, and by the universal property of
  $\eta \colon \O(i) \rightarrow \O_C(i)$, this completely determines
  $\alpha_i^\ast$. So there is \emph{at most} one total transformation
  $\alpha \colon F \rightarrow \O_C$. It remains to show that defining
  $\alpha_i^\ast$ in this way yields a well-defined total
  transformation.

  For $\alpha_i^\ast$ to be well-defined, we must show that
  $e \mapsto Fe^\ast(\top)$ is a finite-meet-preserving cover-to-join
  map $\O(i) \rightarrow Fi$. Finite meet preservation is as
  in~\cite[Proposition 3.2]{Cockett2014Restriction}; on the other
  hand, if $e \in \O(i)$ and $X \in C(e)$, then, since $F$ is a
  cover-to-join map, we must have $Fe = \bigvee_x Fx$, and so
  $Fe^\ast(\top) = \bigvee_x Fx^\ast(\top)$, as desired. Finally, for
  naturality of $\alpha$, we again argue as in~\cite[Proposition
  3.2]{Cockett2014Restriction}.

  It remains to prove the final claim of the proposition. We will
  prove more generally that, if $F \Rightarrow G \colon \A \rightarrow
  \B$ is a total transformation and $G$ is a cover-to-join map, then
  $F$ is too. So let $f \in \A(i,j)$ and $X \in C(f)$. We know that
  $\bigvee_{x \in X}Fx \leqslant Ff$ and so it suffices to show both
  sides have the same restriction, for which we calculate (using
  totality of $\alpha$) that:
  \begin{equation*}
    \overline{Ff} = \overline{\alpha_j Ff} = \overline{Gf \alpha_i} =
    \overline{\bigvee_x Gx \circ \alpha_i} = \overline{\bigvee_x
      \alpha_j Fx}  = 
    \overline{\bigvee_x Fx}\text{. }\qedhere
  \end{equation*}

\end{proof}

We can now use this to prove:
\begin{Thm}
  \label{thm:17}
  Let $C$ be a coverage on the restriction category $\A$. The unit 
  \begin{equation*}
    \cd[@!C@C-1em]{
      {\A} \ar[rr]^-{\eta} \ar[dr]_-{\O_C} & &
      {\Phi \Psi(\O_C)} \ar[dl]^-{\pi_{\Psi \O_C} = \O_{\Phi \Psi(\O_C)}} \\ &
      {\cat{Loc}_p}
    }
  \end{equation*}
  at $\O_C$ of our main adjunction~\eqref{eq:7} exhibits $\Phi
  \Psi(\O_C)$ as the relative join completion of $\A$ with respect to $\C$.
\end{Thm}
\begin{proof}
  We claim that, for any join restriction category $\B$, a restriction
  functor $F \colon \A \rightarrow \B$ is a cover-to-join map if
  and only if it can be extended, necessarily uniquely, to a map
  \begin{equation*}
    \cd{
      {\A} \ar[rr]^-{F} \ar[dr]_-{\O_C} & \ltwocell{d}{} &
      {\B} \ar[dl]^-{\O_\B} \\ &
      {\cat{Loc}_p}
    }
  \end{equation*}
  of
  ${\mathrm{r}\cat{Cat}\mathord{\mkern1mu/\!/\mkern1mu}\cat{Loc}_p}$.
  Indeed, from the fact that $\O_\B$ is hyperconnected it follows
  easily that $F$ is a cover-to-join map if and only if $\O \circ F$
  is so; and now by Proposition~\ref{prop:38}, $\O \circ F$ is a
  cover-to-join map if and only if it admits a, necessarily unique
  total transformation as displayed to $\O_C$.
  
  The theorem now follows immediately from the above observation and
  the universal property of the adjunction $\Psi \dashv \Phi$.
\end{proof}

\bibliographystyle{acm}
\bibliography{bibdata}

\end{document}